\documentclass[12pt,reqno,english]{amsart}
\usepackage{fullpage}
\usepackage{amsfonts}
\usepackage{amssymb}
\usepackage{hyperref}
\usepackage[T1]{fontenc}
\usepackage[utf8]{inputenc}
\usepackage{babel}

\usepackage{csquotes}
\usepackage{times}
\usepackage{amsmath}
\usepackage{url}
\usepackage{blindtext}
\usepackage{tikz-cd}
\usepackage{mathtools}
\usepackage{graphicx}
\usepackage{amsmath,calligra,mathrsfs}
\newtheorem{thm}{Theorem}[section]
\newtheorem{cor}[thm]{Corollary}
\newtheorem{lem}[thm]{Lemma}
\newtheorem{prop}[thm]{Proposition}
\theoremstyle{definition}
\newtheorem{Def}[thm]{Definition}

\newtheorem{conj}[thm]{Conjecture}
\newtheorem{const}[thm]{Construction}
 \newtheorem{remark}[thm]{Remark}
 \newtheorem{conv}[thm]{Convention}
  \newtheorem{cond}[thm]{Conditions}

\newcommand{\OSp}{O_{S}^{\text{perf}}}

\newcommand{\gerbeE}{\mathcal{E}}
\newcommand{\id}{\text{id}}

\newcommand{\A}{\mathbb{A}}

\newcommand\restr[2]{{
  \left.\kern-\nulldelimiterspace 
  #1 
  \vphantom{\big|} 
  \right|_{#2} 
  }}

\newcommand{\Hom}{\text{Hom}}
\newcommand*{\sheafhom}{\mathrm{H}\kern -.5pt om}
\newcommand{\Spec}{\text{Spec}}
\DeclareMathOperator{\sHom}{\mathscr{H}\text{\kern -3pt {\calligra\large om}}\,}
\newcommand{\Z}{\mathbb{Z}}

\begin{document}

\title[Short version of title]{Rigid inner forms over global function fields}

\author{Peter Dillery}
\thanks{This research was partially supported by NSF grant DMS-1840234 and the Rackham pre-doctoral fellowship while the author was at the University of Michigan. This paper was edited while the author was supported by the Brin Postdoctoral Fellowship at the University of Maryland.}
\thanks{Keywords: Global function fields, fppf cohomology, Langlands conjectures, endoscopy}
\thanks{2020 MSC classification:  11F70, 11R58, 11E72, 11S37,  14F20 18F20.}
\address{University of Maryland, Department of Mathematics}
\email{dillery@umd.edu}
\maketitle

\begin{abstract}
We construct an fpqc gerbe $\gerbeE_{\dot{V}}$ over a global function field $F$ such that for a connected reductive group $G$ over $F$ with finite central subgroup $Z$, the set of $G_{\gerbeE_{\dot{V}}}$-torsors contains a subset $H^{1}(\gerbeE_{\dot{V}}, Z \to G)$ which allows one to define a global notion of ($Z$-)rigid inner forms. There is a localization map $H^{1}(\gerbeE_{\dot{V}}, Z \to G) \to H^{1}(\gerbeE_{v}, Z \to G)$, where the latter parametrizes local rigid inner forms (cf. \cite{Tasho, Dillery}) which allows us organize local rigid inner forms across all places $v$ into coherent families. Doing so enables a construction of (conjectural) global $L$-packets and a conjectural formula for the multiplicity of an automorphic representation $\pi$ in the discrete spectrum of $G$ in terms of these $L$-packets. We also show that, for a connected reductive group $G$ over a global function field $F$, the adelic transfer factor $\Delta_{\A}$ for the ring of adeles $\A$  of $F$ serving an endoscopic datum for $G$ decomposes as the product of the normalized local transfer factors from \cite{Dillery}. 
\end{abstract}

\section{Introduction}

\subsection{Motivation}
The goal of this paper is to develop a notion of rigid inner forms over a global function field $F$ in order to relate the local constructions in \cite{Dillery} to the global Langlands correspondence for a connected reductive group $G$ over $F$. This global construction allows one to relate the adelic transfer factor $\Delta_{\A}$ serving an endoscopic datum for $G$ to the normalized transfer factors serving the localizations of this datum constructed in \cite{Dillery} and formulate precise conjectures concerning the global $L$-packet $\Pi_{\varphi}$ for an admissible tempered discrete homomorphism $\varphi \colon L_{F} \to \prescript{L}{}G$, where $L_{F}$ is the conjectural Langlands dual group of $F$. Previously, such descriptions were only possible in the case when $G$ is quasi-split.

We first summarize the local situation discussed in \cite{Dillery}: Let $F_{v}$ be the completion of $F$ at a place $v$ with absolute Galois group $\Gamma_{v}$ ($\Gamma$ denotes the absolute Galois group of $F$), and fix $Z \to G$ a finite central $F_{v}$-subgroup. Recall that, given an inner twist $G \xrightarrow{\psi} G'$, a \textit{local $Z$-rigid inner twist} enriching it is a pair $(\mathscr{T}, \bar{h})$, where $\mathscr{T}$ is a(n) (fpqc) $G$-torsor on a canonically-defined local gerbe $\gerbeE_{v} \to \text{Schemes}/F_{v}$ which descends to a torsor over $F_{v}$ after modding out by $Z$ and $\bar{h}$ identifies this descent with the $G_{\text{ad}}$-torsor canonically associated to $\psi$.

A (tempered) irreducible representation of $(G', \mathscr{T}, \bar{h})$ is a 4-tuple $(G', (\mathscr{T}, \bar{h}), \pi)$, where $\pi$ is an (tempered) irreducible representation of $G'(F_{v})$. The set of all equivalence classes of (tempered) irreducible representations of rigid inner forms of $G$ is denoted by $\Pi^{\text{rig}}(G)$ ($\Pi^{\text{rig}}_{\text{temp}}(G)$), where this equivalence relation is defined on set of the 4-tuples $(G', (\mathscr{T}, \bar{h}), \pi)$ giving the data of such representations, not just on the representations $\pi$. We then predict the following picture:
\begin{conj}[\cite{Dillery}, Conjecture 7.14] Given a tempered $L$-parameter $\varphi_{v} \colon W_{F_{v}}' \to \prescript{L}{}G$, there is a finite subset $\Pi_{\varphi_{v}} \subset \Pi_{\text{temp}}^{\text{rig}}(G)$ and a commutative diagram
\begin{equation}\label{introconj}
\begin{tikzcd}
\Pi_{\varphi_{v}} \arrow["\iota_{\varphi_{v},\mathfrak{w}_{v}}"]{r} \arrow{d} & \text{Irr}(\pi_{0}(S_{\varphi_{v}}^{+})) \arrow{d} \\
H^{1}(\gerbeE_{v}, Z \to G^{*}) \arrow{r} & \pi_{0}(Z(\widehat{\overline{G}})^{+,v})^{*},
\end{tikzcd}
\end{equation}
\end{conj}
\noindent where $G^{*}$ is a quasi-split $Z$-rigid inner twist of $G$, $H^{1}(\gerbeE_{v}, Z \to G^{*})$ denotes the set of isomorphism classes of all the torsors $\mathscr{T}$ described above, $\mathfrak{w}_{v}$ is a choice of Whittaker datum for $G^{*}$, the bottom map a generalization of the local Kottwitz pairing (and Tate-Nakayama duality), both horizontal maps are bijective, and $S^{+}_{\varphi_{v}}$,  $Z(\widehat{\overline{G}})^{+,v}$ are the preimages of $Z_{\widehat{G}}(\varphi_{v})$ and  $Z(\widehat{G})^{\Gamma_{v}}$ (respectively) in $\widehat{\overline{G}} := \widehat{G/Z}$. Note that one may have to enlarge $Z$ to find $G^{*}$ as above.

It should be possible to conjecturally describe the global $L$-packet $\Pi_{\varphi}$ for an admissible tempered discrete homomorphism $\varphi \colon L_{F} \to \prescript{L}{}G$ using the $L$-packets for its localizations $\varphi_{v}$ and then study it using \eqref{introconj} for each $v$. The key to this approach is organizing families of representations of local rigid inner twists of $G_{F_{v}}$ into \textit{coherent families}, which is to say, finding a notion of a \textit{global rigid inner twist} coming from a \textit{global gerbe} $\gerbeE_{\dot{V}}$ which localizes in an appropriate way to such a family. In order for the corresponding family of homomorphisms $\{H^{1}(\gerbeE_{v}, Z \to G) \to \pi_{0}(Z(\widehat{\overline{G}})^{+,v})^{*}\}_{v}$ to behave in a reasonable manner (such as having a well-defined product over all places), one would like a homomorphism $$H^{1}(\gerbeE_{\dot{V}}, Z \to G) \to [\pi_{0}(Z(\widehat{\overline{G}})^{+})]^{*}$$ that equals the product of all of its local analogues (note that if $Z(\widehat{\overline{G}})^{+}$ is the preimage of $Z(\widehat{G})^{\Gamma}$, then we have maps $\pi_{0}(Z(\widehat{\overline{G}})^{+})\to \pi_{0}(Z(\widehat{\overline{G}})^{+,v})$ for all $v$). 

The combination of the construction of the local gerbe in \cite{Dillery} and the global Galois gerbe $\gerbeE_{\dot{V}}$ for number fields in \cite{Tasho2} gives a blueprint for the aforementioned global gerbe for function fields (and thus for global rigid inner forms). As in the local case, the gerbe $\gerbeE_{\dot{V}}$ will be banded by a canonically-defined profinite group $P_{\dot{V}}$ and we will extract the gerbe by proving the existence of a canonical class in $H^{2}_{\text{fppf}}(F, P_{\dot{V}})$; unlike in the local case, producing this class requires significant work---in particular, we must study gerbes over $\text{Spec}(\A)$ and generalize the notion of \textit{complexes of tori}, as in \cite{KS1}, to function fields. 

After producing the gerbe $\mathcal{E}_{\dot{V}}$ we define the cohomology set $H^{1}(\gerbeE_{\dot{V}}, Z \to G)$, a global analogue of $H^{1}(\gerbeE_{v}, Z \to G)$; as one would hope, there are morphisms $\gerbeE_{v} \to \gerbeE_{\dot{V}}$ which give us localization maps between these two sets and a global duality result for $H^{1}(\gerbeE_{\dot{V}}, Z \to G)$ which among other properties, gives the homomorphism $H^{1}(\gerbeE_{\dot{V}}, Z \to G) \to [\pi_{0}(Z(\widehat{\overline{G}})^{+})]^{*}$ described above. This construction lets us call a family of rigid inner twists $\{(G_{F_{v}}, (\mathscr{T}_{v}, \overline{h}_{v}))\}_{v}$ \textit{coherent} if each torsor $\mathscr{T}_{v}$ is the localization of the same global torsor $\mathscr{T}$, $[\mathscr{T}] \in H^{1}(\gerbeE_{\dot{V}}, Z \to G^{*})$ (for $G^{*}$ the quasi-split inner form of $G$ and some appropriate choice of $Z$). Given a such family, we can then define the (conjectural) global $L$-packet $\Pi_\varphi$ for a fixed $\varphi$ (picking a global Whittaker datum $\mathfrak{w}$): 
$$\Pi_{\varphi} := \{ \pi = \otimes_{v}' \pi_{v} \mid (G_{F_{v}}, (\mathscr{T}_{v}, \bar{h}_{v}), \pi_{v}) \in \Pi_{\varphi_{v}}, \iota_{\varphi_{v},\mathfrak{w}_{v}}((G_{F_{v}}, (\mathscr{T}_{v}, \bar{h}_{v}), \pi_{v})) = 1\text{ for almost all } v\}.$$ 

We show (Lemma \ref{admissible}) that $\Pi_{\varphi}$ consists of irreducible tempered admissible representations of $G(\A)$ using a torsor-theoretic analogue of \cite[Prop. 6.1.1]{Taibi}. Moreover, given $\pi \in \Pi_{\varphi}$, we give a conjectural description of the multiplicity of $\pi$ in the discrete spectrum of $G$ by defining a pairing $$\langle -, - \rangle \colon \mathcal{S}_{\varphi} \times \Pi_{\varphi} \to \mathbb{C},$$ where $\mathcal{S}_{\varphi}$ is a finite group closely related to the centralizer of $\varphi$ in $\widehat{G}$; the pairing is defined as a product over all places of two factors involving both rows of \eqref{introconj}. This product formula is well-defined because our representation $\pi$ arises from a coherent family of representations of local rigid inner twists. Given this pairing, we have for each $\pi$ and $L$-packet $\Pi_{\varphi}$ containing $\pi$ an integer $$m(\varphi, \pi) := |\mathcal{S}_{\varphi}|^{-1} \sum_{x \in \mathcal{S}_{\varphi}} \langle x, \pi \rangle,$$ and, furthermore, we conjecture:

\begin{conj}[Kottwitz, \cite{Kott84}] The multiplicity of $\pi$ in the discrete spectrum of $G$ is given by the sum $$\sum_{\varphi} m(\varphi, \pi),$$ where the sum is over all $\varphi$ such that $\pi \in \Pi_{\varphi}$.
\end{conj}

Recall that local rigid inner forms were used in \cite[\S 7]{Dillery} to construct a normalized transfer factor (depending on a quasi-split rigid inner twist $(\mathscr{T}_{v}, \bar{h}_{v}))$ of $G_{F_{v}}$ with Whittaker datum $\mathfrak{w}_{v}$) $$\Delta_{v} = \Delta[\mathfrak{w}_{v}, \dot{\mathfrak{e}}_{v}, \mathfrak{z}_{v}, \psi, (\mathscr{T}_{v}, \bar{h}_{v})]$$ serving a fixed endoscopic datum $\mathfrak{e}_{v}$ for $G_{F_{v}}$.  As such, global rigid inner forms give us a method of relating the global adelic transfer factor $\Delta_{\A}$ defined in \cite{LS} (for number fields, but which is easily translated to a global function field) serving a global endoscopic datum to the transfer factors $\Delta_{v}$ serving the localizations of that datum. Using the relationship between the local and global pairings one obtains (Proposition \ref{locglobtransfer}) a product formula (after fixing a coherent family of rigid inner twists and a global Whittaker datum $\mathfrak{w}$ with localizations $\mathfrak{w}_{v}$)
$$\Delta_{\A}(\gamma_{1}, \delta) = \prod_{v \in V} \langle \text{loc}_{v}(\mathscr{T}_{\text{sc}}), \dot{y}_{v}' \rangle \cdot \Delta[\mathfrak{w}_{v}, \dot{\mathfrak{e}}_{v}, \mathfrak{z}_{v}, \psi, (\mathscr{T}_{v}, \bar{h}_{v})](\gamma_{1,v}, \delta_{v})$$
which expresses the value of $\Delta_{\A}$ at a pair of adelic elements $(\gamma_{1}, \delta)$ as a product of each $\Delta_{v}$ at the localizations of these elements, along with some auxiliary factors $\langle \text{loc}_{v}(\mathscr{T}_{\text{sc}}), \dot{y}_{v}' \rangle$ which are harmless and only necessary for technical reasons.

\subsection{Overview}
In \S 2 we prove some preliminary results that allow us to make computations using \v{C}ech cohomology, both with respect to the covers $O_{E,S}/O_{F,S}$, where $S$ is a finite subset of places of $F$ and $E/F$ is a finite (not necessarily Galois) field extension, and the covers $\A_{E}/\A$, where $\A_{E} = E \otimes_{F} \A$. We also review some basic result about projective systems of abstract gerbes and their \v{C}ech cohomology (which comes from \cite[\S 2]{Dillery}). 

In \S 3 we prove an analogue of global Tate duality for the groups $H^{2}_{\text{fppf}}(F, Z)$, where $Z$ is a finite multiplicative $F$-group scheme. After that, we define a projective system of multiplicative group schemes $\{P_{E,\dot{S}_{E},n}\}$ whose limit gives the pro-algebraic group $P_{\dot{V}}$ that will band our global gerbe. Once $P_{\dot{V}}$ is defined, we show that its first fppf cohomology group over $F$ vanishes using local and global class field theory and that its second fppf cohomology group contains a canonical class.
 
After the canonical class is constructed we can define the global gerbe $\gerbeE_{\dot{V}}$, whose cohomology is studied in \S4, building towards proving a duality result for the cohomology sets $H^{1}(\gerbeE_{\dot{V}}, Z \to G)$, where $Z$ is a finite central subgroup of $G$. We also prove a result concerning the localizations of torsor on $\gerbeE_{\dot{V}}$ which will be used in \S 5 to prove that global $L$-packets consist of irreducible, tempered, admissible representations.

In \S 5, we develop endoscopy, defining the adelic transfer factor for function fields, coherent families of rigid inner forms. We relate the local constructions of \cite{Dillery} to global endoscopy, including the adelic transfer factor and the multiplicity formula. In Appendix A, we establish complexes of tori in the setting of \v{C}ech cohomology and prove several results analogous to those in the appendices of \cite{KS1} (that used Galois cohomology) which are used in the proof of the existence of a canonical class in \S 3. 

\subsection{Notation and terminology}
We will always assume that $F$ is a global field of characteristic $p > 0$. For an arbitrary algebraic group $G$ over $F$, $G^{\circ}$ denotes the identity component. For a connected reductive group $G$ over $F$, $Z(G)$ denotes the center of $G$, and for $H$ a subgroup of $G$, $N_{G}(H), Z_{G}(H)$ denote the normalizer and centralizer group schemes of $H$ in $G$, respectively. We will denote by $\mathscr{D}(G)$ the derived subgroup of $G$, by $G_{\text{ad}}$ the quotient $G/Z(G)$, and if $G$ is semisimple, we denote by $G_{\text{sc}}$ the simply-connected cover of $G$; if $G$ is not semisimple, $G_{\text{sc}}$ denotes $\mathscr{D}(G)_{\text{sc}}$. If $T$ is a maximal torus of $G$, denote by $T_{\text{sc}}$ its preimage in $G_{\text{sc}}$. We fix an algebraic closure $\bar{F}$ of $F$, which contains a separable closure of $F$, denoted by $F^{s}$. For $E/F$ a Galois extension, we denote the Galois group of $E$ over $F$ by $\Gamma_{E/F}$, and we set $\Gamma_{F^{s}/F} =: \Gamma$.

We denote by $V$ the set of all places of $F$, and for $E/F$ a finite extension and $S \subseteq V$, we denote by $S_{E}$ the preimage of $S$ in $V_{E}$, the set of all places of $E$. We call a subset of $V$ \textit{full} if it equals $S_{F}$ for some subset $S$ of places of $\mathbb{F}_{p}(t)$ (after choosing an embedding $\mathbb{F}_{p}(t) \to F$). For a finite subset $S \subset V$, we set $\A_{S}:= \prod_{v \in S} F_{v} \times \prod_{v \notin S} O_{F_{v}}$, and set $\A_{E,S} := \A_{E, S_{E}}$.

We call an affine, commutative algebraic group over a ring $R$ \textit{multiplicative} if it is Cartier dual to an \'{e}tale $R$-group scheme. For $Z$ a multiplicative group over $F$, we denote by $X^{*}(Z), X_{*}(Z) (=X_{*}(Z^{\circ}))$ the character and co-character modules of $Z$, respectively, viewed as $\Gamma$-modules. For $\mathcal{H}$ an algebraic group over $\mathbb{C}$ we will frequently denote $\mathcal{H}(\mathbb{C})$ by $\mathcal{H}$. For two $F$-schemes $X, Y$ and $F$-algebra $R$, we set $X \times_{\Spec(F)} Y =: X \times_{F} Y$, or by $X \times Y$ if $F$ is understood, and set $X \times_{F} \text{Spec}(R) =: X_{R}$. We define $X(\text{Spec}(R)) =: X(R)$, the set of $F$-morphisms $\text{Spec}(R) \to X$.

\subsection{Acknowledgements} The author thanks Tasho Kaletha for introducing to him the motivating question of this paper, as well as for his feedback and proof-reading. They also thank Brian Conrad for a helpful discussion about the images of maps of tori in the analytic topology. The author also gratefully acknowledges the support of NSF grant DMS-1840234 and the Rackham Predoctoral Fellowship.

\tableofcontents

\section{Preliminaries}

\subsection{\v{C}ech cohomology over $O_{F,S}$}

Fix a global function field $F$ of characteristic $p > 0$, a finite non-empty set $S$ of places of $F$, and an $F$-torus $T$ which is unramified outside $S$. Let $O_{F,S}$ denote the elements of $F$ whose valuation is non-negative at all places outside $S$, and for a finite Galois extension $K/F$, denote by $O_{K,S}$ the elements of $K$ whose valuation is non-negative at all places outside $S_{K}$, the set of all places of $K$ lying above $S$. We set $O_{S} := \varinjlim_{K/F} O_{K,S}$, where $K/F$ ranges over all finite Galois extensions which are unramified outside of $S$. Denote by $F_{S}$ the maximal field extension of $F$ which is unramified outside $S$, and denote its Galois group over $F$ by $\Gamma_{S}$; note that $F_{S} = \text{Frac}(O_{S})$. The torus $T$ has a canonical model defined over $O_{F,S}$ and we ease notation by denoting the corresponding $O_{F,S}$-scheme also by $T$. 

For all $q>0$, it is a basic fact of fppf cohomology (\cite[Lem. 2.1]{Cesnavicius}) that for a commutative group scheme $\mathscr{G}$ on $O_{F,S}$ which is locally of finite presentation, 
 we have  $H^{q}_{\text{fppf}}(O_{S}, \mathscr{G}) = \varinjlim_{K/F} H^{q}_{\text{fppf}}(O_{K,S}, \mathscr{G}),$ with the transition maps induced by pullback of fppf sheaves (the same is true if we replace ``fppf" by ``\'{e}tale"). We begin with the following commutative-algebraic lemma: 
 
 \begin{lem}\label{etalesplitting} For $K/F$ a finite Galois extension unramified outside $S$ and $n \geq 2$, the natural injection $O_{K,S} ^{\bigotimes_{O_{F,S}} n} \to \prod_{\Gamma_{K/F}^{n-1}} O_{K,S}$ is an isomorphism.
\end{lem}

\begin{proof} By induction, it is enough to prove the result for $n=2$. First, note that $O_{K,S}/O_{F,S}$ is finite \'{e}tale by assumption (since $K/F$ is unramified outside of $S$). In particular, $O_{K,S}$ is finitely-generated and torsion-free as an $O_{F,S}$-module, and both rings are Dedekind domains integrally closed in their fields of fractions. By base-change, we get a finite \'{e}tale extension $O_{K,S} \otimes_{O_{F,S}} O_{K,S}/ O_{K,S}$, which is still finitely-generated, locally free, and torsion-free as an $O_{K,S}$-module (this last fact follows from using the injection $O_{K,S} \otimes_{O_{F,S}} O_{K,S} \hookrightarrow \prod_{\Gamma_{K/F}} O_{K,S}$, under which $O_{K,S}$ maps into the diagonally-embedded copy, which clearly acts on the product without torsion). 

We are thus in the setting of \cite[Thm. 1.3]{Conrad}, which says that the composition $$O_{K,S} \otimes_{O_{F,S}} O_{K,S} \hookrightarrow K \otimes_{O_{F,S}} O_{K,S} \xrightarrow{\sim} \prod_{i} K_{i}$$ maps $O_{K,S} \otimes_{O_{F,S}} O_{K,S}$ isomorphically onto the product of integral closures of $O_{K,S}$ in each $K_{i}$, where $K_{i}$ is some finite separable extension of $K$ and the last isomorphism comes from the fact that $K \otimes_{O_{F,S}} O_{K,S}$ is finite \'{e}tale over $K$ a field. It thus suffices to show that each $K_{i}$ is actually $K$.

The desired result follows from the series of elementary manipulations $$K \otimes_{O_{F,S}} O_{K,S} \xrightarrow{\sim} K \otimes_{F} F \otimes_{O_{F,S}} O_{K,S} \xrightarrow{\sim} K \otimes_{F} K \xrightarrow{\sim} \prod_{\Gamma_{K/F}} K.$$We leave it to the reader to check that the isomorphism $O_{K,S} \otimes_{O_{F,S}} O_{K,S} \xrightarrow{\sim} \prod_{\Gamma_{K/F}} O_{K,S}$ constructed above agrees with the map in the statement of the Lemma.
\end{proof}

Recall from \cite[\S 2.2]{Dillery}, that if we fix a ring homomorphism $R \to S$ and abelian sheaf $\mathscr{F}$ on $R$ (with the fpqc topology), then $\check{H}^{i}(S/R, \mathscr{F})$ denotes the $i$th cohomology group of the complex
$$\mathscr{F}(S) \to \mathscr{F}(S \otimes_{R} S) \to \mathscr{F}(S \otimes_{R} S \otimes_{R} S) \to \dots, $$ 
with differentials the alternating sums of the $n+1$ natural maps $\mathscr{F}(S^{\bigotimes_{R}n}) \to \mathscr{F}(S^{\bigotimes_{R}(n+1)})$.

\begin{cor}\label{cechtogalois} We have a canonical isomorphism $\check{H}^{i}(O_{K,S}/O_{F,S}, G) \xrightarrow{\sim} H^{i}(\Gamma_{K/F}, G(O_{K,S}))$ for any commutative $O_{F,S}$-group $G$. Taking the direct limit also gives a canonical isomorphism $\check{H}^{i}(O_{S}/O_{F,S}, G) \xrightarrow{\sim} H^{i}(\Gamma_{S}, G(O_{S}))$.
\end{cor}

\begin{proof} All that one must check is that the isomorphism of Lemma \ref{etalesplitting} preserves cocycles and coboundaries, which is straightforward.
\end{proof}

In order to compare the \v{C}ech cohomology groups $\check{H}^{i}(O_{S}/O_{F,S}, T)$ with $H^{i}_{\text{fppf}}(O_{F,S}, T)$, we need to prove some cohomological vanishing results. The first result involves \'{e}tale cohomology:

\begin{lem}\label{vanishing1} We have that $H^{i}_{\text{et}}(O_{S}, T_{O_{S}}) = 0$ for all $i > 0$.
\end{lem}

\begin{proof} Since $T$ splits over $O_{S}$ it suffices to prove the result for $T = \mathbb{G}_{m}$. For $i=1$, we have $$H^{1}_{\text{et}}(O_{S}, T_{O_{S}}) = \varinjlim \limits_{F \subset K \subset F_{S} \text{ finite Galois}} H^{1}_{\text{et}}(O_{K,S}, \mathbb{G}_{m}) = \varinjlim \text{Pic}(\text{Spec}(O_{K,S})) = \varinjlim \text{Cl}(O_{K,S}) = 0,$$ where the second equality comes from \cite[Prop. II.2.1]{Milne}) and the third from the proof of \cite[Prop. 8.3.6]{NSW}.

For $i =2$, identifying $H^{2}_{\text{et}}(O_{K,S}, \mathbb{G}_{m}) = \text{Br}(O_{K,S})$ gives (by \cite[6.9.2]{Poonen}) an exact sequence $$0 \to \text{Br}(O_{K,S}) \to \bigoplus_{v \in S_{K}} \text{Br}(K_{v}) \xrightarrow{\sum \text{inv}_{v}} \mathbb{Q}/\mathbb{Z},$$ where $K_{v}$ denotes the completion of $K$ at $v$. Taking the direct limit of the first two terms shows that we have an injective map $\text{Br}(O_{S}) \hookrightarrow \bigoplus_{v \in S_{F_{S}}} \text{Br}(F_{S} \cdot F_{v})$. Note that the field extension $F_{S} \cdot F_{v}$ contains $F^{\text{nr}}_{v}$, the maximal unramified extension of $F_{v}$, using the fact that $F_{S}/F$ contains all finite extensions of the constant field of $F$. Moreover, the valuation ring $O_{F_{S} \cdot F_{v}}$ of this field is Henselian, as it is the direct limit of the Henselian rings $O_{K_{v}}$(\cite[p. 56]{Hochster}), and the previous sentence implies that it has algebraically closed residue field. We may then deduce from the proof of \cite[Prop. I.A.1]{Milne} that $\text{Br}(F_{S} \cdot F_{v}) = 0$, giving the desired result.

Finally, for $i \geq 2$, we have that for any $K/F$ a finite Galois extension, we have $H^{i}_{\text{et}}(O_{K,S}, \mathbb{G}_{m}) = 0$, by \cite{Milne}, Remark II.2.2. Taking the direct limit gives the desired result.
\end{proof}

By combining Lemma \ref{vanishing1} with the spectral sequence from \cite[Thm. 6.7.5]{Poonen} we obtain:
\begin{cor}\label{galoistoetale} We have canonical isomorphisms $H^{i}(\Gamma_{S}, T(O_{S})) \xrightarrow{\sim} H^{i}_{\text{et}}(O_{F,S}, T)$ for all $i \geq 1$.
\end{cor}

\begin{lem}\label{vanishing2} We have $H^{i}_{\text{fppf}}(O_{S}^{\bigotimes_{O_{F,S}} n}, T) = 0$ for all $n,i \geq 1$.
\end{lem}

\begin{proof} Since $T$ is unramified outside $S$ it is enough to show that $$ H^{i}_{\text{fppf}}(O_{S}^{\bigotimes_{O_{F,S}} n}, \mathbb{G}_{m}) = \varinjlim \limits_{F \subset K \subset F_{S} \text{ finite Galois}}  H^{i}_{\text{fppf}}(O_{K,S}^{\bigotimes_{O_{F,S}} n}, \mathbb{G}_{m}) =  0.$$ By Lemma \ref{etalesplitting}, we have a canonical identification $\text{Spec}(O_{K,S}^{\bigotimes_{O_{F,S}} n}) = \coprod_{\underline{\sigma} \in \Gamma_{K/F}^{n-1}} \text{Spec}(O_{K,S})$ and isomorphism $ H^{i}_{\text{fppf}}(\coprod_{\underline{\sigma}} \text{Spec}(O_{K,S}), \mathbb{G}_{m}) \xrightarrow{\sim} \prod_{\underline{\sigma}} H^{i}_{\text{fppf}}(O_{K,S}, \mathbb{G}_{m})$. Also, if $K'/K$ is finite and contained in $F_{S}$, then the map $H^{i}_{\text{fppf}}(O_{K,S}^{\bigotimes_{O_{F,S}} n}, \mathbb{G}_{m}) \to H^{i}_{\text{fppf}}(O_{K',S}^{\bigotimes_{O_{F,S}} n}, \mathbb{G}_{m})$ corresponds via this isomorphism to diagonally embedding each factor of $\prod_{\underline{\sigma} \in \Gamma_{K/F}^{n-1}} H^{i}_{\text{fppf}}(O_{K,S}, \mathbb{G}_{m})$ into some subset of the factors of $\prod_{\underline{\sigma} \in \Gamma_{K'/F}^{n-1}} H^{i}_{\text{fppf}}(O_{K',S}, \mathbb{G}_{m})$ (by means of the pullback map $H^{i}_{\text{fppf}}(O_{K,S}, \mathbb{G}_{m}) \to H^{i}_{\text{fppf}}(O_{K',S}, \mathbb{G}_{m})$). 

Hence, for any $\alpha \in H^{i}_{\text{fppf}}(O_{K,S}^{\bigotimes_{O_{F,S}} n}, \mathbb{G}_{m})$, to show that $\alpha$ vanishes in some $H^{i}_{\text{fppf}}(O_{K',S}^{\bigotimes_{O_{F,S}} n}, \mathbb{G}_{m})$ for large $K'$ it is enough to show that $\varinjlim_{K/F} H^{i}_{\text{fppf}}(O_{K,S}, \mathbb{G}_{m}) = 0$ for all $i$, thus reducing the result to the case $n=1$, which follows from Lemma \ref{vanishing1}.
\end{proof} 

For any abelian fppf group scheme $A$ over $O_{F,S}$ with pro-fppf cover $R/O_{F,S}$ , the Grothendieck spectral sequence gives us a spectral sequence $$E_{2}^{p,q} = \check{H}^{p}(R/O_{F,S}, \underline{H}^{q}(A)) \Rightarrow H^{p+q}_{\text{fppf}}(O_{F,S}, A),$$ where $\underline{H}^{q}(A))$ denotes the presheaf on $\text{Sch}/O_{F,S}$ sending $U$ to $H^{q}(U, A_{U})$ (see \cite[03AV]{Stacksproj}). 

\begin{prop}(\cite[03AV]{Stacksproj})\label{ComparisonIso} If $H^{i}_{\text{fppf}}(R^{\bigotimes_{O_{F,S}}n}, A) = 0$ for all $n, i \geq 1$, then the above spectral sequence induces a canonical isomorphism $\check{H}^{i}(R/O_{F,S}, A) \xrightarrow{\sim} H^{i}_{\text{fppf}}(O_{F,S}, A)$ for all $i$.
\end{prop}

\begin{remark} Strictly speaking, \cite[Lems. 03AZ, 03F7]{Stacksproj} are stated in the setting of an fppf cover $R/O_{F,S}$, but taking the direct limit of spectral sequences gives us the result for pro-fppf covers (rings $R$ which are direct limits of fppf covers, such as $O_{S}$). 
\end{remark}

\begin{cor} We have a canonical isomorphism $\check{H}^{i}(O_{S}/O_{F,S}, T) \xrightarrow{\sim} H^{i}_{\text{fppf}}(O_{F,S}, T)$ for all $i$.
\end{cor}

\begin{proof} Combine Lemma \ref{vanishing2} with Proposition \ref{ComparisonIso}.
\end{proof}

We now consider possibly non-\'{e}tale extensions in order to handle the cohomology of non-smooth finite $F$-groups. For $R$ an $\mathbb{F}_{p}$-algebra, let $R^{\text{perf}} := \varinjlim{R}$, where the direct limit is over all powers of the Frobenius homomorphism. For $R = O_{F,S}$, the ring $O_{F,S}^{\text{perf}}$ is obtained by adjoining all $p$-power roots of elements of $O_{F,S}$ (in a fixed algebraic closure $\overline{F}/F$). We begin by recalling an elementary lemma on the splitting of primes in rings of integers of purely inseparable extensions:

\begin{lem}\label{purelyinsep} Let $F'/F$ be a purely inseparable extension and $\mathfrak{p} \subset O_{F}$ a prime ideal. Then $\mathfrak{p} \cdot O_{F'} = (\mathfrak{p}')^{[F' : F]}$ for some prime $\mathfrak{p}'$ of $O_{F'}$.
\end{lem}

\begin{proof} It is evidently enough to prove this in the case when $[F':F] = p$, which we now assume. We claim that $O_{F'} = O_{F}^{(p)}$, the extension of $O_{F}$ obtained by adjoining all $p$-power roots. There is an obvious inclusion of $O_{F}$-algebras $O_{F'} \hookrightarrow O_{F}^{(p)}$ because $F' = F^{(p)}$. The morphism of smooth projective curves $X' \to X$ corresponding to the inclusion $F \to F'$ is purely inseparable of degree $p$ so by \cite[0CCV]{Stacksproj} we obtain an isomorphism of $O_{F}$-algebras $O_{F'} \xrightarrow{\sim} O_{F}^{(p)}$, giving the claim. The claim implies that, at the level of local rings, a uniformizer $\varpi \in O_{F,\mathfrak{p}}$ has a $p$th root in $O_{F', \mathfrak{p'}}$ for any prime $\mathfrak{p}'$ above $\mathfrak{p}$, giving the desired result. 
\end{proof}

Denote by $F_{m}$ the field extension of $F$ obtained by adjoining all $p^{m}$-power roots; this is a finite, purely inseparable extension. We have the following characterization of the perfect closure $O_{S}^{\text{perf}}$:

\begin{lem}\label{perfection} The canonical map $\varinjlim_{m} O_{S_{m}} \to O_{S}^{\text{perf}}$ is an isomorphism, where $S_{m}$ denotes the preimage of $S$ in $\text{Spec}(O_{F_{m}})$. 
\end{lem}

\begin{proof} For the inclusion of the right-hand side into the left-hand side, note that if $x \in \overline{F}$ is such that $x^{p^{m}} \in O_{E,S}$ for some finite (Galois) $E \subset F_{S}$, then $x \in E' := E \cdot F_{m}$, which is unramified over $F_{m}$ outside of $S_{m}$, and so $x \in O_{E',S_{m}} \subset O_{S_{m}}$. For the other inclusion, consider a finite Galois extension $K'$ of the finite purely inseparable extension $F' := F_{m}/F$ with $S':= S_{m}$. We may factor $K'/F$ as a tower $K'/K/F$, where $K/F$ is the separable (Galois) closure of $F$ in $K'$ and $K'/K$ is purely inseparable. Note that $K \cdot F' = K'$; one containment is clear, and the other follows from the fact that $K$ and $F'$ are linearly disjoint and $[K': F'] = [K: F]$. 

We want to show that $K/F$ is unramified outside $S$; this follows because for any prime $\mathfrak{p}$ of $O_{F}$, we know from Lemma \ref{purelyinsep} that $\mathfrak{p}$ splits as $(\mathfrak{p}')^{[F':F]}$ in $O_{F'}$, and if $\mathfrak{p}'$ is a prime of $O_{F',S'}$, then it factors in $O_{K'}$ as $\mathfrak{P}'_{1} \cdot \dots \cdot \mathfrak{P}'_{r}$, which means that $\mathfrak{p}$ splits in $O_{K'}$ as $(\mathfrak{P}'_{1} \cdot \dots \cdot \mathfrak{P}'_{r})^{[F':F]}$. Since $[F' : F] = [K' : K]$, we know that $\mathfrak{p}$ must not ramify in $O_{K}$, or else the ramification degree would be too large. Now for any element $x \in O_{K',S'}$, we have that $x^{p^{m}} \in K$ and is integral over $O_{F,S}$, and hence lies in $O_{K,S}$, showing that $O_{K',S'} \subseteq O_{K,S}^{(p^{m})} \subset O_{S}^{\text{perf}}$, giving the other inclusion. 
\end{proof}

\noindent We are ready to prove that the cover $O_{S}^{\text{perf}}$ computes the cohomology of multiplicative $O_{F,S}$-groups. 

\begin{lem}\label{toricomparison} For $A$ a multiplicative $F$-group split over $O_{S}$, the groups $H^{i}_{\text{fppf}}((O_{S}^{\text{perf}})^{\bigotimes_{O_{F,S}}n}, A)$ vanish for all $i, n \geq 0$.
\end{lem}

\begin{proof} It is enough to prove the result for $A = \mathbb{G}_{m}$ and $A = \mu_{m}$. We focus on the former first: Since $\mathbb{G}_{m}$ is smooth we can use \cite[Lem. 2.2.9]{Rosengarten} to replace $(O_{S}^{\text{perf}})^{\bigotimes_{O_{F,S}} n}$ by $[(O_{S}^{\text{perf}})^{\bigotimes_{O_{F,S}} n}]_{\text{red}} = (O_{S}^{\text{perf}})^{\bigotimes_{O_{F,S}^{\text{perf}}} n}$ and reduce the claim to showing that each $H^{j}_{\text{fppf}}((O_{S}^{\text{perf}})^{\bigotimes_{O_{F,S}^{\text{perf}}} n}, \mathbb{G}_{m})$ vanishes. By Lemma \ref{perfection}, we have $$(O_{S}^{\text{perf}})^{\bigotimes_{O_{F,S}^{\text{perf}}} n} = \varinjlim_{m} O_{S_{m}}^{ \bigotimes_{O_{F_{m}, S_{m}}} n},$$ and hence it's enough to show that $\varinjlim_{m} H^{j}_{\text{fppf}}(O_{S_{m}}^{ \bigotimes_{O_{F_{m}, S_{m}}} n}, \mathbb{G}_{m}) = 0$ for all $j, n \geq 1$, which folllows from Lemma \ref{vanishing2}.

We now prove the $\mu_{m}$-case. For $i >1$, we immediately deduce that $H^{i}_{\text{fppf}}((O_{S}^{\text{perf}})^{\bigotimes_{O_{F,S}} n}, \mu_{m})$ vanishes from the long exact sequence in fppf cohomology (induced by the Kummer sequence) and the $\mathbb{G}_{m}$-case. For $i=1$, since $H^{1}_{\text{fppf}}((O_{S}^{\text{perf}})^{\bigotimes_{O_{F,S}} n}, \mathbb{G}_{m}) = 0$, we have from the long exact sequence in fppf cohomology that $H^{1}_{\text{fppf}}((O_{S}^{\text{perf}})^{\bigotimes_{O_{F,S}} n}, \mu_{m})$ is the quotient $$((O_{S}^{\text{perf}})^{\bigotimes_{O_{F,S}} n})^{*}/(((O_{S}^{\text{perf}})^{\bigotimes_{O_{F,S}} n})^{*})^{m}.$$ 

Since $O_{S}^{*}$ is $n'$-divisible for $n'$ coprime to $p$ by \cite[Prop. 8.3.4]{NSW}, so is $(O_{S}^{\bigotimes_{O_{F,S}} n})^{*}$ (using Lemma \ref{etalesplitting}). Now $O_{S}^{\text{perf},*}$ is $\mathbb{N}$-divisible, since it is obtained from $O_{S}$ by adjoining all $p$-power roots, and once again this implies that $((O_{S}^{\text{perf}})^{\bigotimes_{O_{F,S}} n})^{*} = (O_{S}^{\bigotimes_{O_{F,S}} n})^{\text{perf}, \ast}$ is as well. 
\end{proof}

\begin{cor}\label{ComparisonIso2} For $A$ as above, we have a canonical isomorphism $$\check{H}^{i}(O_{S}^{\text{perf}}/O_{F,S}, A) \xrightarrow{\sim} H^{i}_{\text{fppf}}(O_{F,S}, A)$$ for all $i$. Moreover, for an $F$-torus $T$ unramified outside $S$, the natural map $$\check{H}^{i}(O_{S}/O_{F,S}, T) \to \check{H}^{i}(O_{S}^{\text{perf}}/O_{F,S}, T)$$ induced by the inclusion $O_{S} \to O_{S}^{\text{perf}}$ is an isomorphism.
\end{cor}
The ring $O_{S}^{\text{perf}}$ will have the role globally that $\overline{F}$ played in the local case, see \cite{Dillery}. We conclude with a useful result concerning the finite-level \v{C}ech cohomology of an $O_{F,S}$-torus $T$ split over $O_{E,S}$. We first recall the following result from \cite{Morris}:

\begin{prop}(\cite[Thm. 3.2]{Morris}) \label{morris} Let $S/R/O_{F,S}$ be two fppf covers of $O_{F,S}$; set $$\Sigma:= [\bigcup_{i} R^{\bigotimes_{O_{F,S}}i}] \cup  [\bigcup_{i} S^{\bigotimes_{O_{F,S}}i}] \cup  [\bigcup_{i} S^{\bigotimes_{R}i}].$$ For $\mathscr{F}$ a sheaf on $(\text{Sch}/O_{F,S})_{\text{fppf}}$ with $H_{\text{fppf}}^{1}(A, \mathscr{F}) = 0$ for all $A \in \Sigma$ there is an exact sequence
$$ 0 \to \check{H}^{2}(R/O_{F,S}, \mathscr{F}) \to \check{H}^{2}(S/O_{F,S}, \mathscr{F}) \to \check{H}^{2}(S/R, \mathscr{F}).$$
\end{prop} 

\begin{cor}\label{EtoEprime} Let $E/F$ be a finite Galois extension, $E'/E$ a finite purely inseparable extension, and $S \subset V$ a finite set of places such that $\text{Cl}(O_{E,S})$ is trivial. Then if $T$ is an $O_{F,S}$-torus split over $O_{E,S}$, the natural map $\check{H}^{2}(O_{E,S}/O_{F,S}, T) \to \check{H}^{2}(O_{E',S}/O_{F,S}, T)$ is an isomorphism.
\end{cor}

\begin{proof} We leave it to the reader to check that the $\Sigma$-condition of Proposition \ref{morris} is satisfied (since everything in $\Sigma$ is an $O_{E,S}$-algebra, we may replace $T$ with $\mathbb{G}_{m}$ for this condition and use the fact that $O_{E,S}$ and $O_{E',S}$ are principal ideal domains, along with \cite[Lem. 2.2.9]{Rosengarten}). It thus suffices to show that $\check{H}^{2}(O_{E',S}/O_{E,S}, \mathbb{G}_{m})$ vanishes. For any $n$, $\mathbb{G}_{m}(O_{E',S}^{\bigotimes_{O_{E,S}}n}) = \mathbb{G}_{m}([O_{E',S}^{\bigotimes_{O_{E,S}}n}]_{\text{red}})$, and now $[O_{E',S}^{\bigotimes_{O_{E,S}}n}]_{\text{red}} = O_{E',S}^{\bigotimes_{O_{E',S}}n} = O_{E,S}$, so our \v{C}ech cohomology computations on this cover reduce to that of the trivial cover $O_{E',S}/O_{E',S}$, giving the desired vanishing.
\end{proof}

\subsection{\v{C}ech cohomology over $\A$}
In this subsection we prove some basic results that allow us to do \v{C}ech cohomology on (covers of) the adele ring $\A$ of our global function field $F$. Let $G$ a multiplicative $F$-group scheme with fixed $O_{F,\Sigma_{0}}$-model $\mathcal{G}$ for a finite subset of places $\Sigma_{0} \subset V$. We begin with some basic results about local fields:

\begin{lem}\label{linearlydisjoint} Let $F' = F_{m}/F$ be a finite, purely inseparable extension. Then $F'$ and $F_{v}$ are linearly disjoint over $F$ inside $\overline{F_{v}}$ (recall that we have fixed such an algebraic closure). 
\end{lem}

\begin{proof} Suppose that we know the result for $F' = F_{1}$. Then, proceeding by induction, $F_{m-1}$ and $F_{v}$ are linearly disjoint, the valuation $v$ extends uniquely to a valuation $v'$ on $F_{m-1}$, and $F_{m-1} \cdot F_{v}$ is the completion of $F_{m-1}$ with respect to $v'$. Thus, $F_{m}/F_{m-1}$ is of degree $p$, and we may replace $F_{v}$ by $(F_{m-1})_{v'}$ and use the $m=1$ case to deduce that $(F_{m-1})_{v'} = F_{m-1} \cdot F_{v}$ and $F_{m}$ are linearly disjoint over $F_{m-1}$, which implies the desired result. 

Setting $F' = F_{1}$, note that the extension $F' \cdot F_{v}/F_{v}$ is either degree $1$ or degree $p$, since $[F' \cdot F_{v} \colon F_{v}] = [F' \colon F_{v} \cap F'] \mid p$, and $F'$ and $F_{v}$ are linearly disjoint if and only if this degree equals $p$. Hence, it's enough to show that $F' \cap F_{v} = F$. Thus, suppose that $x \in F_{v}$ is such that $x^{p} \in F$. If $F(x) \neq F$, then $F(x) = F'$, so that $F_{v}$ contains all $p$th roots of $F$; in particular, $\varpi^{1/p} \in F_{v}$, where $\varpi \in O_{F,v}$ (the localization of $O_{F}$ at $v$) is a $v$-adic uniformizer, which is clearly false. 
\end{proof}

Now let $K/F$ be a finite (not necessarily separable) field extension with completion $K_{w}$ for $w \mid v$. The following result is important for our adelic \v{C}ech cohomology:

\begin{lem}\label{tensorinjectivity} For any $n$, the natural map $O_{K_{w}}^{\bigotimes_{O_{F_{v}}}n} \to K_{w}^{\bigotimes_{F_{v}}n}$ is injective.
\end{lem}

\begin{proof} The ring $O_{K_{w}}$ is finite and torsion-free over the principal ideal domain $O_{F_{v}}$, and is thus free as an $O_{F_{v}}$-module. We may thus pick a basis (which is also an $F_{v}$-basis for $K_{w}$) which allows us to view the map in question as the natural map
$(O_{F_{v}}^{\bigoplus m})^{\bigotimes_{O_{F_{v}}}n}  \to (F_{v}^{\bigoplus m})^{\bigotimes_{F_{v}}n}$, which may be rewritten as the obvious inclusion $O_{F_{v}}^{\bigoplus mn} \hookrightarrow F_{v}^{\bigoplus mn}$, giving the result.  \end{proof}

We can now prove our first adelic result. Note that if $\A_{K} := K \otimes_{F} \A,$ then $\A_{K}^{\bigotimes_{\A}n}  = (K^{\bigotimes_{F}n}) \otimes_{F} \A$. Let $\A_{K,v}$ denote the $F_{v}$-algebra $K \otimes_{F} F_{v}$, and let $O_{K,v}$ denote the $O_{F_{v}}$-algebra $O_{K} \otimes_{O_{F}} O_{F_{v}}$. 

\begin{prop}\label{adelictensorprop} For any finite extension $K/F$, we have a canonical identification 
\[
\begin{tikzcd}\A_{K}^{\bigotimes_{\A}n} \arrow["\sim"]{r} & \prod_{v \in V}' \A_{K,v}^{\bigotimes_{F_{v}}n},
\end{tikzcd}
\]
where the restriction is with respect to the image of the map $O_{K,v}^{\bigotimes_{O_{F_{v}}}n} \to \A_{K,v}^{\bigotimes_{F_{v}}n}$ (in fact, the proof will imply that this map is an inclusion).
\end{prop}

\begin{proof} It's enough to show that $\A_{K}^{\bigotimes_{\A}n} = (K^{\bigotimes_{F}n}) \otimes_{F} \A$ is isomorphic to the restricted product $\prod_{v \in V}' (K^{\bigotimes_{F}n} \otimes_{F} F_{v})$,
where the restriction is with respect to the image of
$O_{K}^{\bigotimes_{O_{F}}n} \otimes_{O_{F}} O_{F_{v}} \to K^{\bigotimes_{F}n} \otimes_{F} F_{v}$;
the claimed isomorphism is defined on simple tensors by sending $x \otimes (a_{v})_{v}$ to $(x \otimes a_{v})_{v}$. To prove that this map gives a well-defined isomorphism, it suffices to show that we have an isomorphism $$O_{K} \otimes_{O_{F}} O_{F_{v}} \xrightarrow{\sim} \prod_{w \mid v} O_{K_{w}}$$ for any $v \in V$. Letting $K'$ be the maximal separable subextension of $K$, we already know that $O_{K'} \otimes_{O_{F}} O_{F_{v}}$ is isomorphic to $\prod_{w' \mid v} O_{K'_{w'}}$, and so we're left with the ring $O_{K} \otimes_{O_{K'}} [\prod_{w' \mid v} O_{K'_{w'}}]$.

We claim that the natural map $O_{K} \otimes_{O_{K'}} O_{K'_{w'}} \to O_{K_{w}}$ (for $w$ the unique extension of $w'$ to $K$) is an isomorphism. For surjectivity, note that by the proof of Lemma \ref{purelyinsep}, we have $O_{K} = O_{K'}^{(1/p^{m})}$, where $p^{m}$ is $[K \colon K']$. We know that $O_{K'_{w'}}$ spans $O_{K_{w}}$ over $O_{K'}^{(1/p^{m})}$, since the ring $O_{K'_{w'}} \cdot O_{K'}^{(1/p^{m})}$ is finitely-generated over the complete discrete valuation ring $O_{K'_{w'}}$, using that $O_{K'}^{(1/p^{m})}$ is finite over $O_{K'}$ by the finiteness of the relative Frobenius morphism (by \cite{Stacksproj}, OCC6, using that $O_{K'}$ is of finite type over $\mathbb{F}_{q}$, being the coordinate ring of an affine open subscheme of a smooth curve over $\mathbb{F}_{q}$), and hence is complete as a topological ring, contains $O_{K}$, and thus must be the $w$-adic completion $O_{K_{w}}$. Injectivity follows from the linear disjointness given by Lemma \ref{linearlydisjoint}. 
\end{proof}

\begin{cor}\label{appendixBcor}\label{mainappendixBcor}  For $K/F$ a finite extension, we have the following canonical identifications, where all limits are over a cofinal system of finite subsets $\Sigma$ of $V$:
\begin{enumerate}
\item{$\A_{K}^{\bigotimes_{\A}n} = \varinjlim_{\Sigma} [\prod_{v \in \Sigma}  \A_{K,v}^{\bigotimes_{F_{v}}n} \times \prod_{v \notin \Sigma} O_{K,v}^{\bigotimes_{O_{F_{v}}}n}];$}
\item{$G(\A_{K}^{\bigotimes_{\A}n}) = \varinjlim_{\Sigma_{0} \subset \Sigma} [ \prod_{v \in \Sigma}  G(\A_{K,v}^{\bigotimes_{F_{v}}n}) \times \prod_{v \notin \Sigma} \mathcal{G}(O_{K,v}^{\bigotimes_{O_{F_{v}}}n})],$}
\end{enumerate}
and (2) induces a canonical identification (where each restriction is with respect to $\mathcal{G}(O_{K,v}^{\bigotimes_{O_{F_{v}}}n})$):
\[
\begin{tikzcd}
G(\overline{\A}^{\bigotimes_{\A}n}) \arrow["\sim"]{r} & \varinjlim \limits_{K/F \text{ finite}} \prod_{v \in V}'  G(\A_{K,v}^{\bigotimes_{F_{v}}n}).
\end{tikzcd}
\]
\end{cor}



\begin{proof}
(1) is immediate and (2) follows from (1) and \cite[Lem. 2.4]{Cesnavicius}. Note that it makes sense to view $\mathcal{G}(O_{K,v}^{\bigotimes_{O_{F_{v}}}n})$ as a subgroup of $G(\A_{K,v}^{\bigotimes_{F_{v}}n})$ by Proposition \ref{adelictensorprop}.
\end{proof}



We give one more result which will be useful for \v{C}ech-cohomological computations:

\begin{prop}\label{cechrestrictedproduct} For $K/F$ a finite extension, the above restricted product decomposition of $G(\A_{K}^{\bigotimes_{\A}n})$ identifies the subgroup of \v{C}ech $n$-cocycles in $G(\A_{K}^{\bigotimes_{\A}n})$ with the kernel of 
\[
\begin{tikzcd}
\prod_{v \in V}'  G(\A_{K,v}^{\bigotimes_{F_{v}}n}) \arrow{r} & \prod_{v \in V}'  G(\A_{K,v}^{\bigotimes_{F_{v}}n+1})
\end{tikzcd}
\]
given by the \v{C}ech differentials with respect to the cover $\A_{K,v}/F_{v}$ on the $G(\A_{K,v}^{\bigotimes_{F_{v}}n})$-factors and the \v{C}ech differentials with respect to the cover $O_{K,v}/O_{F_{v}}$ on the $\mathcal{G}(O_{K,v}^{\bigotimes_{O_{F_{v}}}n})$-factors (note that these differentials land in the desired restricting subgroups, so this is well-defined). 
\end{prop}

\begin{proof} It's enough to check that the restricted product identifications are compatible with the three inclusion maps $p_{i} \colon \A_{K}^{\bigotimes_{\A}n} \to \A_{K}^{\bigotimes_{\A}n+1}$, $p^{v}_{i} \colon \A_{K,v}^{\bigotimes_{F_{v}}n} \to \A_{K,v}^{\bigotimes_{F_{v}}n+1}$, and $p^{v,\circ}_{i} \colon  O_{K,v}^{\bigotimes_{O_{F_{v}}}n} \to O_{K,v}^{\bigotimes_{O_{F_{v}}}n+1}$ for $1 \leq i \leq n+1$, which is straightforward.
\end{proof}

We develop some cohomological results concerning covers of $\A$, analogous to the results of \S 2.1 for covers of $O_{F,S}$. Set $\overline{\A}_{v} := \overline{F} \otimes_{F} F_{v}$. For notational convenience, we use $H^{i}$ to denote $H^{i}_{\text{fppf}}$.

\begin{lem}\label{vadelicvanishing1} For $M$ a multiplicative $F$-group, we have $H^{n}(\overline{\A}_{v}^{\bigotimes_{F_{v}} k}, M) = 0$ for all $n,k \geq 1$. 
\end{lem}

\begin{proof} For $E'/F$ a finite algebraic extension with maximal separable and purely inseparable subextensions $E,F'$ respectively, note that by Lemma \ref{linearlydisjoint} we have a sequence of isomorphisms $$(E' \otimes_{F} F_{v})^{\bigotimes_{F_{v}}k} \xrightarrow{\sim} [F' \otimes_{F} (E \otimes_{F} F_{v})]^{\bigotimes_{F_{v}}k} \xrightarrow{\sim} [\prod_{w \mid v} E'_{w'}]^{\bigotimes_{F_{v}}k} \xrightarrow{\sim} \prod_{w_{1}, \dots , w_{k} \mid v} \bigotimes_{F_{v}}^{i=1, \dots, k} E'_{w_{i}'},$$ where $E'_{w'}$ is the completion of $E'$ with respect to the unique extension $w'$ of the valuation $w$ on $E$ to the purely inseparable extension $E'$, for all $w \mid v$ in $V_{E}$, and in the third term above, $F_{v}$ is embedded into the direct product diagonally.  and so we obtain an identification $$H^{n}((E' \otimes_{F} F_{v})^{\bigotimes_{F_{v}} k}, M) \xrightarrow{\sim} \prod_{w_{1}, \dots, w_{k} \mid v_{F}} H^{n}((E'_{w_{i}'})^{\bigotimes_{F_{v}} k}, M).$$

Moreover, for $K'/E'$ two such extensions, the inductive map $(E' \otimes_{F} F_{v})^{\bigotimes_{F_{v}}k} \to (K' \otimes_{F} F_{v})^{\bigotimes_{F_{v}}k}$ gets translated to the map on the corresponding products defined by the product over all $k$-tuples $(w_{1}, \dots, w_{k})$ of the maps $$\bigotimes_{F_{v}}^{j=1, \dots, k} E'_{w_{j}'} \to \prod_{\tilde{w}_{1}, \dots \tilde{w}_{k}; \tilde{w}_{j} \mid w_{j} \forall j} \bigotimes_{F_{v}}^{j=1, \dots, k} K'_{\tilde{w}_{j}'}$$ given in the obvious way. The upshot is that it suffices to show that each $\varinjlim_{K'/F} H^{n}(K'_{(w_{K})'}/F_{v}, M)$ vanishes, where $\{w_{K}\}$ is a coherent system of places lifting $v$ (equivalent to fixing a place $\dot{v}$ on $F^{\text{sep}}$ lifting $v$). But each direct limit of this form is isomorphic to $H^{n}(\overline{F_{v}}, M)$, which we know vanishes.
\end{proof}

Fix an embedding $\overline{F} \to \overline{F_{v}}$, which is equivalent to picking a place $\dot{v} \in V_{F^{\text{sep}}}$ lying above $v$. Then $\dot{v}$ and a choice of section $\Gamma_{F}/\Gamma_{F}^{\dot{v}} \xrightarrow{s} \Gamma_{F}$ induce a homomorphism of $F_{v}$-algebras $h \colon \overline{F_{v}} \to \overline{\A}_{v}$ defined as follows: Let $E'/F$ be a finite algebraic extension with $E' = E(x^{1/p^{m}}) = F(x^{1/p^{m}}) \otimes_{F} E$ for $x \in F$ and $E/F$ a finite Galois subextension. There is a homomorphism $$E \cdot F_{v} \to \prod_{w \mid v} E_{w} \xrightarrow{\sim} E \otimes_{F} F_{v},$$ where the first map is the ``diagonal" embedding induced by the fixed embedding $E \cdot F_{v} \to \overline{F_{v}}$ and the section $s$. Applying $F(x^{1/p^{m}}) \otimes_{F} -$ (and Lemma \ref{linearlydisjoint}) extends this to a homomorphism $E' \cdot F_{v} \to E' \otimes_{F} F_{v}$, and taking the direct limit over all finite $E'/F$ gives the map $h$. 


\begin{cor}\label{shapiro1} For any $k \in \mathbb{N}$ and multiplicative $F$-group $M$, the map $h$ induces an isomorphism, called the ``Shapiro isomorphism," $$S_{v}^{k} \colon \check{H}^{k}(\overline{F_{v}}/F_{v}, M) \to \check{H}^{k}(\overline{\A}_{v}/F_{v}, M).$$ 
\end{cor}

\begin{proof} For any finite algebraic field extension $E'/F$, the extension of rings $F_{v} \to F_{v} \otimes_{F} E'$ is fppf. Thus, we get a natural map $$\check{H}^{k}(\overline{\A}_{v}/F_{v}, M) \xrightarrow{\sim} \varinjlim_{E'/F} \check{H}^{k}((E' \otimes_{F} F_{v})/F_{v}, M) \to \check{H}^{k}_{\text{fppf}}(F_{v}, M) \xrightarrow{\sim} H^{k}(F_{v}, M)$$ via the natural comparison homomorphism $\check{H}_{\text{fppf}}^{k}(F_{v}, M) \to H^{k}(F_{v}, M)$ (from \cite[Lem. 03AX]{Stacksproj}). By taking the direct limit of the spectral sequence from \cite[Lem. 03AZ]{Stacksproj}, we deduce that the above map $\check{H}^{k}(\overline{\A}_{v}/F_{v}, M)  \to H^{k}(F_{v}, M)$ is an isomorphism, since the cohomology groups $H^{j}(\overline{\A}_{v}^{\bigotimes_{F_{v}} m}, M)$ vanish for all $j, m \geq 1$ by Lemma \ref{vadelicvanishing1}. Now the commutative diagram 
\[
\begin{tikzcd}
\check{H}^{k}(\overline{F_{v}}/F_{v}, M) \arrow["S_{v}^{k}"]{rr} \arrow["\sim"]{dr} & & \check{H}^{k}(\overline{\A}_{v}/F_{v}, M) \arrow["\sim"]{ld} \\
& \check{H}^{k}_{\text{fppf}}(F_{v}, M)
\end{tikzcd}
\]
implies that $S_{v}^{k}$ is an isomorphism.
\end{proof}

We conclude this subsection by discussing the independence of $S_{v}^{2}$ on the section $\Gamma_{F}/\Gamma_{F}^{\dot{v}} \to \Gamma_{F}$ used to construct $h$.

\begin{lem}\label{differentsections} Let $s_{v}$ and $s'_{v}$ be two choices of sections, $M$ a multiplicative $F$-group, and $\dot{S}_{v}^{2}$, $\dot{S}_{v}^{' 2}$ the corresponding Shapiro homomorphisms $M(\overline{F_{v}}^{\bigotimes_{F_{v}} 3}) \to M(\overline{\A}_{v}^{\bigotimes_{F_{v}} 3})$. Then the induced maps on \v{C}ech cohomology from $\check{H}^{2}(\overline{F_{v}}/F_{v}, M)$ to $\check{H}^{2}(\overline{\A}_{v}/F_{v}, M)$ are the same.
\end{lem}

\begin{proof} 
Since the Shapiro homomorphisms are constructed via the direct limit over finite algebraic extensions, it's enough to prove that, for any fixed $2$-cocycle $x \in M((E'_{\dot{v}_{E'}})^{\bigotimes_{F_{v}}3})$, $E'/F$ a finite extension of fields, there is a 1-cochain $c \in M((E' \otimes_{F} F_{v})^{\bigotimes_{F_{v}}2})$ such that $dc = \dot{S}_{v}^{2}(x) \cdot \dot{S}_{v}^{' 2}(x)^{-1}$, and that if we have a inductive system $\{x_{E'}\}_{E'}$ of such 2-cocycles, as $E'/F$ ranges over an exhaustive tower of finite extensions, then the system $\{c_{E'}\}_{E'}$ is also inductive. We will construct each $c_{E'}$ explicitly using $x$ (it will be useful later to have an explicit cochain to work with).

Assume that $E'/F$ is of the form on p.12, let $E/F$ (resp. $F'/F$) denote a maximal separable (resp. purely inseparable) subextension of $E'/F$, set $E_{v} := E_{\dot{v}_{E}}$, and denote the extension of $\dot{v}_{E}$ to $E'$ by $v'$. For $w \mid v$ in $V_{E}$, denote by $r_{w}, \bar{r}_{w}$ the corresponding isomorphisms $E'_{v'} \xrightarrow{\sim} E'_{w'}$ (induced by applying $F' \otimes_{F} -$ to the isomorphisms $E_{v} \xrightarrow{\sim} E_{w}$ defined by our sections). We define $$c \in \prod_{w_{i_{1}},w_{i_{2}} \mid v_{F}} M(E'_{w'_{i_{1}}} \otimes_{F_{v}} E'_{w'_{i_{2}}})$$ to be given on the $(w_{i_{1}}, w_{i_{2}})$-factor by $$(r_{w_{i_{1},1}} \cdot \bar{r}_{w_{i_{1},3}} \otimes r_{w_{i_{2},2}})(x) \cdot (\bar{r}_{w_{i_{1},2}} \otimes r_{w_{i_{2},1}} \cdot \bar{r}_{w_{i_{2},3}})(x)^{-1},$$ where $r_{w_{i_{j}},k}$ denotes that the source is the $k$th tensor factor of $(E'_{v'})^{\bigotimes_{F_{v}}3}$, $1 \leq k \leq 3$.  It is clear that such a system of 1-cochains $\{c_{E'}\}$ is inductive if the system $\{x_{E'}\}$ is. Recall that $\dot{S}_{v}^{2}$, $\dot{S}_{v}^{' 2}$ are group homomorphisms $$M((E'_{v'})^{\bigotimes_{F_{v}}3}) \to \prod_{w_{i_{1}},w_{i_{2}}, w_{i_{3}} \mid v_{F}} M(E'_{w'_{i_{1}}} \otimes_{F_{v}} E'_{w'_{i_{2}}} \otimes_{F_{v}} E'_{w'_{i_{3}}}).$$ 

To show that $dc = \dot{S}_{v}^{2}(x) \cdot \dot{S}_{v}^{' 2}(x)^{-1}$, we may focus on a fixed $(w_{i_{1}}, w_{i_{2}}, w_{i_{3}})$-factor of the right-hand side. In this factor, the differential of $c$ is given by the six-term product $$(1 \otimes r_{w_{i_{2},1}} \cdot  \bar{r}_{w_{i_{2},3}} \otimes  r_{w_{i_{3},2}})(x) \cdot (r_{w_{i_{1},1}} \cdot  \bar{r}_{w_{i_{1},3}} \otimes 1 \otimes  r_{w_{i_{3},2}})(x)^{-1} \cdot ( r_{w_{i_{1},1}} \cdot  \bar{r}_{w_{i_{1},3}} \otimes \ r_{w_{i_{2},2}} \otimes 1)(x) $$ $$\cdot (1 \otimes \bar{r}_{w_{i_{2},2}} \otimes  r_{w_{i_{3},1}} \cdot  \bar{r}_{w_{i_{3},3}})(x)^{-1} \cdot ( \bar{r}_{w_{i_{1},2}} \otimes 1 \otimes r_{w_{i_{3},1}} \cdot  \bar{r}_{w_{i_{3},3}})(x) \cdot ( \bar{r}_{w_{i_{1},2}} \otimes  r_{w_{i_{2},1}} \cdot  \bar{r}_{w_{i_{2},3}} \otimes 1)(x)^{-1}.$$ 

Since $x$ is a 2-cocycle, we claim that the term $(1 \otimes  r_{w_{i_{2},1}} \cdot \bar{r}_{w_{i_{2},3}}\otimes  r_{w_{i_{3},2}})(x)$ equals 
\begin{equation}\label{cocycle1}
(\bar{r}_{w_{i_{1},1}}\otimes  r_{w_{i_{2},3}} \otimes r_{w_{i_{3},2}})(x) \cdot ( \bar{r}_{w_{i_{1},2}} \otimes  r_{w_{i_{2},1}} \cdot  \bar{r}_{w_{i_{2},3}} \otimes 1)(x) \cdot (\bar{r}_{w_{i_{1},2}} \otimes \bar{r}_{w_{i_{2},1}} \otimes r_{w_{i_{3},3}})(x)^{-1}.
\end{equation}
To see this, note that $$(1 \otimes \id_{1} \otimes \id_{2} \otimes \id_{3})(x) \cdot (\id_{1} \otimes 1 \otimes \id_{2} \otimes \id_{3})(x)^{-1} \cdot  (\id_{1} \otimes \id_{2} \otimes 1 \otimes \id_{3})(x) \cdot  (\id_{1} \otimes  \id_{2} \otimes \id_{3} \otimes 1)(x)^{-1} = 1$$ (inside the group $M((E'_{v'})^{\bigotimes_{F_{v}} 4})$), and now applying $(\id_{2} \otimes \id_{1} \cdot \id_{4} \otimes \id_{3}) \circ (r_{w_{i_{2}}} \otimes \bar{r}_{w_{i_{1}}} \otimes r_{w_{i_{3}}} \otimes \bar{r}_{w_{i_{2}}})$ to the above expression gives the desired equality. We will leave the checking of similar equalities to the reader throughout the proof. The second factor in \eqref{cocycle1} cancels with the last factor in the main six-term equation. Next, we may rewrite the first term of \eqref{cocycle1} as 
\begin{equation}\label{cocycle2} 
( 1 \otimes \bar{r}_{w_{i_{2},2}} \otimes r_{w_{i_{3},1}} \cdot \bar{r}_{w_{i_{3},3}})(x) \cdot (\bar{r}_{w_{i_{1},1}} \otimes \bar{r}_{w_{i_{2},2}} \otimes \bar{r}_{w_{i_{3},3}})(x)^{-1} \cdot (\bar{r}_{w_{i_{1},1}} \otimes 1 \otimes r_{w_{i_{3},2}} \cdot \bar{r}_{w_{i_{3},3}})(x).
\end{equation}

We may also replace $(r_{w_{i_{1},1}} \cdot \bar{r}_{w_{i_{1},3}} \otimes r_{w_{i_{2},2}} \otimes 1)(x)$ from the main equation by the expression $$(\bar{r}_{w_{i_{1},2}} \otimes r_{w_{i_{2},1}} \otimes r_{w_{i_{3},3}})(x) \cdot (r_{w_{i_{1},1}} \cdot \bar{r}_{w_{i_{1},2}} \otimes 1 \otimes r_{w_{i_{3},3}})(x)^{-1} \cdot (r_{w_{i_{1},1}} \otimes r_{w_{i_{2},2}} \otimes r_{w_{i_{3},3}})(x),$$ reducing us to showing the equality \begin{equation}\label{maineq2} \begin{split}
(\bar{r}_{w_{i_{1},1}} \otimes 1 \otimes r_{w_{i_{3},2}} \cdot \bar{r}_{w_{i_{3},3}})(x) \cdot (r_{w_{i_{1},1}} \cdot \bar{r}_{w_{i_{1},2}} \otimes 1 \otimes r_{w_{i_{3},3}})(x)^{-1}  \\
\cdot (\bar{r}_{w_{i_{1},2}} \otimes 1 \otimes r_{w_{i_{3},1}} \cdot \bar{r}_{w_{i_{3},3}})(x) \cdot (r_{w_{i_{1},1}} \cdot \bar{r}_{w_{i_{1},3}} \otimes 1 \otimes r_{w_{i_{3},2}})(x)^{-1} =1.
\end{split}
\end{equation}
Replacing the third factor of \eqref{maineq2} by the expression $$(r_{w_{i_{1},1}} \cdot \bar{r}_{w_{i_{1},3}} \otimes 1 \otimes \bar{r}_{w_{i_{3},2}})(x) \cdot (r_{w_{i_{1},1}} \otimes 1 \otimes r_{w_{i_{3},2}} \cdot \bar{r}_{w_{i_{3},3}})(x)^{-1} \cdot (r_{w_{i_{1},1}} \cdot \bar{r}_{w_{i_{1},2}} \otimes 1 \otimes \bar{r}_{w_{i_{3},3}})(x)$$ reduces \eqref{maineq2} to the equality $$(\bar{r}_{w_{i_{1},1}} \otimes 1 \otimes r_{w_{i_{3},2}} \cdot \bar{r}_{w_{i_{3},3}})(x) \cdot (r_{w_{i_{1},1}} \cdot \bar{r}_{w_{i_{1},2}} \otimes 1 \otimes r_{w_{i_{3},3}})(x)^{-1} $$$$ \cdot (r_{w_{i_{1},1}} \otimes 1 \otimes r_{w_{i_{3},2}} \cdot \bar{r}_{w_{i_{3},3}})(x)^{-1} \cdot (r_{w_{i_{1},1}} \cdot \bar{r}_{w_{i_{1},2}}\otimes 1 \otimes \bar{r}_{w_{i_{3},3}})(x) = 1,$$ which follows easily from the fact that $x$ is a 2-cocycle. 
\end{proof}

\subsection{\v{C}ech cohomology and projective systems} We collect some results concerning the way that \v{C}ech cohomology behaves with respect to projective systems of abelian group schemes. Let $R,S$ be rings with a $\tau$-homomorphism $S \to R$ (where $\tau = \text{\'{e}tale, fppf, fpqc, etc.}$) and let $A$ be a commutative group scheme over $R$. We begin by recalling some gerbe-theoretic constructions:

\begin{Def}[\cite{Dillery}, Definition 2.35] \label{explicitgerbe} Fix a \v{C}ech 2-cocycle $c \in A(S^{\bigotimes_{R}3})$. We may define an $A$-gerbe as follows: take the fibered category $\gerbeE_{c} \to \text{Sch}/R$ whose fiber over $V$ is defined to be the category of pairs $(T,\psi)$, where $T$ is a (right) $A_{V \times_{R} S}$-torsor on $V \times_{R} S$ with $A$-action $m$ (in the $\tau$ topology), along with an isomorphism of $A_{V \times_{R} (S \otimes_{R} S)}$-torsors $\psi \colon p_{2}^{*}T \xrightarrow{\sim} p_{1}^{*}T$, called a \textit{twisted gluing map}, satisfying the following ``twisted gluing condition" on the $A_{V \times_{R} (S^{\bigotimes_{R}3})}$-torsor $q_{1}^{*}T$: $$(p_{12}^{*}\psi) \circ (p_{23}^{*} \psi ) \circ (p_{13}^{*}\psi)^{-1} = m_{c},$$ where $m_{c}$ denotes the automorphism of the torsor $q_{1}^{*}T$ given by right-translation by $c$. A morphism $(T,\psi_{T}) \to (S, \psi_{S})$ in $\gerbeE_{c}$ lifting an $R$-morphism $V \xrightarrow{f} V'$ is a morphism of $A_{V \times S}$-torsors $T \xrightarrow{h} f^{*}S$ satisfying, on $V \times_{F} (S \otimes_{R} S)$, the relation $f^{*}\psi_{S} \circ p_{2}^{*}h = p_{1}^{*}h \circ \psi_{T}$. We will call such a pair $(T, \psi)$ in $\gerbeE_{c}(V)$ a \textit{$c$-twisted torsor over $V$} when $A$ is understood. We call $\gerbeE_{c}$ the \textit{gerbe corresponding to $c$}.
\end{Def}
Recall the following functoriality property of the above gerbes:

\begin{const}[\cite{Dillery}, Construction 2.38] \label{changeofgerbe} Let $A \xrightarrow{f} B$ be an $F$-morphism of commutative group schemes and $a, b \in A(S^{\bigotimes_{R}3}), B(S^{\bigotimes_{R}3})$ \v{C}ech 2-cocycles such that $[f(a)]= [b]$ in $\check{H}^{2}(S/R, B)$. Any $x \in B(S \otimes_{R} S)$ satisfying $d(x)\cdot b = f(a)$ defines a morphism $\gerbeE_{a} \xrightarrow{\phi_{a,b,x}} \gerbeE_{b}$ of fibered categories over $R$ as follows.

For any $a$-twisted torsor $(T,\psi)$ over an $R$-scheme $V$, we define the $b$-twisted torsor $(T',\psi') =: \phi_{a,b,x}(T, \psi)$ over $V$ as follows. Define the $B_{V \times_{R} S}$ torsor $T'$ to be $T \times^{A_{V \times S},f} B_{V \times S}$, and take the gluing map to be $\psi' := \overline{m_{x^{-1}} \circ \psi}$, where $\overline{m_{x^{-1}} \circ \psi}$ denotes the isomorphism of contracted products $$p_{2}^{*}(T \times^{A_{V \times S},f} B_{V \times S}) = $$ 
$$ (p_{2}^{*}T) \times^{A_{V \times (S \otimes_{R}S)},f} B_{V \times (S \otimes_{R}S)} \to (p_{1}^{*}T) \times^{A_{V \times (S \otimes_{R}S)},f} B_{V \times (S \otimes_{R}S)} = p_{1}^{*}(T \times^{A_{V \times S},f} B_{V \times S})$$ induced by $(m_{x^{-1}} \circ \psi) \times \text{id}_{B}$ (we are implicitly identifying $x$ with its image in $B(V \times_{F} (S \otimes_{R}S))$). 

A morphism $(T_{1},\psi_{1}) \xrightarrow{\varphi} (T_{2},\psi_{2})$ of $a$-twisted torsors induces a morphism of the corresponding $b$-twisted torsors via the map on contracted products induced by $\varphi \times \text{id}$, giving the desired functor. \end{const}

As in \cite{Dillery}, we adopt the following convention and also assume that $S/R$ is an fpqc cover:

\begin{conv}\label{conventions} When discussing an abelian $R$-group scheme $A$ and cover $S/R$, we will always assume that $\check{H}^{1}_{\text{fppf}}(S^{\bigotimes_{R}n}, A) = 0$ for all $n \geq 0$. 
\end{conv}

We now recall inverse limits of gerbes (cf. \cite[\S 2.7]{Dillery}). Fix a system $\{A_{n}\}_{n \in \mathbb{N}}$ of commutative affine groups over $R$ with transition epimorphisms $p_{n+1,n} \colon A_{n+1} \to A_{n}$. Assume that we have systems $\{a_{n} \in A_{n}(S^{\bigotimes_{R}3})\}$ and $\{x_{n} \in A_{n}(S \otimes_{R}S) \}$ such that $a_{n}$ are \v{C}ech 2-cocycles and $a_{n} \cdot dx_{n} = p_{n+1,n}(a_{n+1})$, giving rise to a system of gerbes $\{\gerbeE_{n}:= \gerbeE_{a_{n}} \to (\text{Sch}/F)_{\text{fpqc}}\}_{n \in \mathbb{N}}$ (abbreviated as just $\{\gerbeE_{n}\}$) with $R$-morphisms $\pi_{n+1,n} \colon \gerbeE_{n+1} \to \gerbeE_{n}$, where $\pi_{n+1,n} := \phi_{a_{n+1},a_{n},x_{n}}$.

\begin{lem}\label{basicresult} The natural map $\check{H}^{i}(S/R, A) \to \varprojlim_{n} \check{H}^{i}(S/R, A_{n})$ is surjective for all $i$.
\end{lem}

\begin{proof} This is an easy calculation using Convention \ref{conventions}.
\end{proof}

\begin{Def}\label{inverselimitofgerbes} Define the \textit{inverse limit} of the system $\{\gerbeE_{n}\}$, denoted by $\varprojlim_{n} \gerbeE_{n} \to (\text{Sch}/R)_{\text{fpqc}}$, as the category with fiber over $U$ with objects given by systems of pairs $(X_{n}, i_{n})_{n \in \mathbb{N}}$ of an object $X_{n} \in \gerbeE_{n}(U)$ and an isomorphism $i_{n} \colon \pi_{n+1,n}(X_{n+1}) \xrightarrow{\sim} X_{n}$ in $\gerbeE_{n}(U)$, and morphisms $(X_{n},i_{n}) \to (Y_{n},j_{n})$ given by a system of morphisms $\{f_{n} \colon X_{n} \to Y_{n} \}$ such that $j_{n} \circ \pi_{n+1,n}(f_{n+1}) = f_{n} \circ i_{n}$ for all $n$ (we extend this definition to morphisms between objects in different fibers in the obvious way). We call such a system of morphisms \textit{coherent}. We have a compatible system of canonical morphisms of $\text{Sch}/R$-categories $\pi_{m} \colon \varprojlim_{n} \gerbeE_{n} \to \gerbeE_{m}$ for all $m$.
\end{Def}

Recall from \cite[Prop. 2.40]{Dillery} that the map from \v{C}ech $2$-cocycles of $S/R$ valued in $A:= \varprojlim_{n} A_{n}$ to $A$-gerbes split over $S$ induces a bijection between isomorphism classes of such gerbes and $\check{H}^{2}(S/R, A)$. The proof of \cite[ Lemma 2.63]{Dillery} (with the cover $\overline{F}/F$ replaced with $S/R$) shows that if each $A_{n}$ satisfies Convention \ref{conventions} then so does $A$. We thus obtain:

\begin{prop}(\cite[Prop. 2.64]{Dillery})\label{invlim1} With the setup as above, the category $\gerbeE := \varprojlim_{n} \gerbeE_{n} \to (\text{Sch}/R)_{\text{fpqc}}$ can be given the structure of an fpqc $A$-gerbe, split over $S$. Moreover, the natural map $$\check{H}^{i}(S/R,  A) \to \varprojlim_{n} \check{H}^{i}(S/R, A_{n})$$ sends the class in $\check{H}^{2}(S/R, A)$ corresponding to $\gerbeE$ to the element $([a_{n}])  \in \varprojlim_{n} \check{H}^{2}(S/R, A_{n})$. 
\end{prop}

We now give a result which characterizes when the surjection of Lemma \ref{basicresult} is an isomorphism.

\begin{prop}\label{cechidentification} Fix $i \geq 1$; if $\varprojlim_{n}^{(1)} \check{H}^{i-1}(S/R, A_{n}) = 0$  and $\varprojlim_{n}^{(1)} B^{i-1}(n) = 0$, where $B^{i-1}(n) \in C^{i-1}(S/R, A_{n})$ is the subgroup of $(i-2)$-coboundaries (the group of $(-1)$-coboundaries is defined to be trivial), then the natural map $\check{H}^{i}(S/R, A) \to \varprojlim_{n} \check{H}^{i}(S/R, A_{n})$ is injective.
\end{prop}

\begin{proof} We denote the differential $A_{k}(S^{\bigotimes_{R}i}) \to A_{k}(S^{\bigotimes_{R}(i+1)})$ by $d^{(k)}$. First, note that since $\varprojlim_{k}^{(1)} \check{H}^{i-1}(S/R, A_{k}) = 0$, the natural map $$\varprojlim_{k} [A_{k}(S^{\bigotimes_{R}i})/B^{i-1}(k)] \to \varprojlim_{k}[(A_{k}(S^{\bigotimes_{R}i})/B^{i-1}(k))/(\check{H}^{i-1}(S/R, A_{k}))]$$ is surjective. Moreover, the natural map $A(S^{\bigotimes_{R}i}) =  \varprojlim_{k} A_{k}(S^{\bigotimes_{R}i}) \to \varprojlim_{k} [A_{k}(S^{\bigotimes_{R}i})/B^{i-1}(k)]$ is surjective, since we assume that $\varprojlim_{n}^{(1)} B^{i-1}(n) = 0$.

Now by left-exactness of the inverse-limit functor, we have the exact sequence 
\[
\begin{tikzcd}
1 \arrow[r] & \varprojlim_{k} \frac{A_{k}(S^{\bigotimes_{R}i})/B^{i-1}(k)}{\check{H}^{i-1}(S/R, A_{k})} \arrow["\varprojlim d^{(k)}"]{r} & A(S^{\bigotimes_{R}(i+1)}) \arrow[r] & \varprojlim_{k} \frac{A_{k}(S^{\bigotimes_{R}(i+1)})}{d^{(k)}(A_{k}(S^{\bigotimes_{R}i}))}. & 
\end{tikzcd}
\]
In particular, if $x \in A(S^{\bigotimes_{R}(i+1)})$ is such that its image in  $\varprojlim_{k} \frac{A_{k}(S^{\bigotimes_{R}(i+1)})}{d^{(k)}(A_{k}(S^{\bigotimes_{R}i}))}$ is zero (which is the hypothesis of the Proposition), then it lies in the image of $\bar{d}:= \varprojlim d^{(k)}$. But now the diagram
\[
\begin{tikzcd}
A (S^{\bigotimes_{R}i}) \arrow{d} \arrow["d"]{dr} & \\
\varprojlim_{k} [A_{k}(S^{\bigotimes_{R}i})/B^{i-1}(k)] \arrow{d} \arrow{r} & A(S^{\bigotimes_{R}(i+1)}) \\
\varprojlim_{k}  \frac{A_{k}(S^{\bigotimes_{R}i})/B^{i-1}(k)}{\check{H}^{i-1}(S/R, A_{k})} \arrow["\bar{d}"]{ur} &
\end{tikzcd}
\]
commutes, and since the vertical composition is surjective and such an $x$ lies in the image of the lower-diagonal map, it lies in the image of the upper-diagonal map, giving the desired result.
\end{proof}

\section{The profinite group $P_{\dot{V}}$}
This section constructs and studies the pro-algebraic group $P_{\dot{V}}$.  As in the previous section, we use $H^{i}$ as a short-hand for $H^{i}_{\text{fppf}}$. For a fixed a finite Galois extension $E/F$ and $S \subset V$ a finite set of places of $F$, we have two common conditions that we want $S$ to satisfy:
\begin{cond}\label{first2placeconditions} \begin{enumerate} \item{$S$ contains all of the places that ramify in $E$}
 \item{Every ideal class of $E$ contains an ideal with support in $S_{E}$, ie., $\text{Cl}(O_{E,S})$ is trivial.}
 \end{enumerate}
 \end{cond}

\subsection{Tate duality for finite multiplicative $Z$}
The goal of this subsection is to construct an analogue of the global Tate duality isomorphism from \cite{Tate66} for the cohomology group $H^{2}_{\text{fppf}}(F, Z) = \check{H}^{2}(\overline{F}/F, Z)$, where $Z$ is a finite multiplicative group over $F$. Set $A = X^{*}(Z)$, and $A^{\vee} = \Hom(A, \mathbb{Q}/\Z)$. Temporarily fix a finite set of places $S \subset V$ and a multiplicative group $M$ over $O_{F,S}$ split over $E$; denote $X^{*}(M)$ by $X$ and $X_{*}(M) (= X_{*}(M^{\circ})$) by $Y$.

For $v \in S$ a fixed place, we denote by $\text{Res}_{E,v}(M)$ the multiplicative $O_{F,S}$-group split over the finite \'{e}tale extension $O_{E,S}$ determined by the $\Gamma_{E/F}$-module $X\otimes_{\Z}\Z[\{v\}_{E}] =: X[\{v\}_{E}]$. 
We set $\text{Res}_{E,S}(M) := \prod_{v \in S} \text{Res}_{E,v}(M)$, a multiplicative $O_{F,S}$-group split over $O_{E,S}$, with character group $X[S_{E}]$.  There is an embedding $M \hookrightarrow \text{Res}_{E/S}(M)$ via the augmentation map on characters $X[S_{E}] \to X$ and $\frac{\text{Res}_{E/S}(M)}{M}$ has character group $X[S_{E}]_{0}$ (the kernel of the augmentation map).

Global Tate duality for tori (as in \cite{Tate66}) shows that for $M = T$ a torus there exists a class $$\alpha_{3}(E, S) \in H^{2}(\Gamma_{E/F}, \frac{\text{Res}_{E,S}(\mathbb{G}_{m})}{\mathbb{G}_{m}}(O_{E,S}))$$ such that cup product with this class induces for all $i \in \Z$ an isomorphism $$\widehat{H}^{i-2}(\Gamma_{E/F}, Y[S_{E}]_{0}) \xrightarrow{\sim} \widehat{H}^{i}(\Gamma_{E/F}, T(O_{E,S})),$$ where to make sense of the relevant cup product pairing, we are making the identifications 
\begin{equation}\label{Y0pairing} Y \otimes_{\Z} \Z[S_{E}]_{0} = \Hom_{\Z}(X, \Z[S_{E}]_{0}) = \Hom_{O_{E,S}\text{-gp}}( \frac{\text{Res}_{E,S}(\mathbb{G}_{m})}{\mathbb{G}_{m}}, T). \end{equation}

We no longer fix $S$ as above. Our first goal is to construct a functorial isomorphism $$\Theta \colon \varinjlim_{E',S'} \widehat{H}^{-1}(\Gamma_{E'/F}, A^{\vee}[S'_{E'}]_{0}) \xrightarrow{\sim} H^{2}(F, Z),$$ where the limit is over all finite subsets $S' \subset V$ and finite Galois extensions $E'/F$. Choose a finite Galois extension $E/F$ splitting $Z$ and a finite full subset $S \subset V$ such that $S$ satisfies Conditions \ref{first2placeconditions} with respect to $E$ and the following additional condition:

\begin{cond}\label{thirdplacecondition} For each $w \in V_{E}$, there exists $w' \in S_{E}$ such that $\text{Stab}(w, \Gamma_{E/F}) = \text{Stab}(w', \Gamma_{E/F})$. 
\end{cond}
\noindent It is straightforward to check that such a pair $(E, S)$ always exists, and that if $S \subseteq S'$ is finite and full, then it also satisfies Conditions \ref{first2placeconditions} and \ref{thirdplacecondition} (with respect to $E$).

Note that for $n$ a multiple of $\text{exp}(Z)$, we have a functorial isomorphism 
\begin{equation}\label{pairing1}
\Phi_{E,S,n} \colon A^{\vee}[S_{E}]_{0} \xrightarrow{\sim} \Hom_{O_{E,S}}(\frac{\text{Res}_{E,S}(\mu_{n})}{\mu_{n}}, Z) = \Hom_{O_{S}}(\frac{\text{Res}_{E,S}(\mu_{n})}{\mu_{n}}, Z),
\end{equation}
which sends $g \in A^{\vee}[S_{E}]_{0}$ to the homomorphism induced by the map $A \to (\frac{1}{n}\Z/\Z)[S_{E}]_{0}$ given by 
\begin{equation}
a \mapsto \sum_{w \in S_{E}} ng(w)(a) \cdot [w],
\end{equation}
where $g(w)$ denotes the $A^{\vee}$-coefficient of $[w]$ in $g$.

Fix a cofinal sequence $\{n_{i}\}$ in $\mathbb{N}^{\times}$ and denote the associated cofinal prime-to-p sequence by $n_{i}' := n_{i}/p^{m_{i}}$. Identifying $\text{Res}_{E/S}(\mathbb{G}_{m})(O_{S})$ with $\text{Maps}(S_{E}, O_{S}^{\times})$, we may pick functions $$k_{i}' \colon \text{Maps}(S_{E}, O_{S}^{\times}) \to  \text{Maps}(S_{E}, O_{S}^{\times})$$ such that $k_{i}'(x)^{n_{i}'} = x$ and $k_{i+1}'(x)^{n_{i+1}'/n_{i}'} = k_{i}'(x)$. Using the bijection between \v{C}ech cochains in $\text{Res}_{E/S}(\mathbb{G}_{m})(O_{S}^{\bigotimes_{O_{F,S}} n})$ and $C^{n-1}(\Gamma_{S}, \text{Res}_{E/S}(\mathbb{G}_{m})(O_{S}))$ (via Lemma \ref{etalesplitting}) we get a map $$k_{i}' \colon \text{Res}_{E/S}(\mathbb{G}_{m})(O_{S}^{\bigotimes_{O_{F,S}} n}) \to \text{Res}_{E/S}(\mathbb{G}_{m})(O_{S}^{\bigotimes_{O_{F,S}} n})$$ for all $n$, where we are using that $O_{S}^{\times}$ is $n$-divisible for $n$ coprime to $p$ (cf. \cite[Prop. 8.3.4]{NSW}). 

As in \cite{Dillery}, we extend this to $p$-power roots. The quotient map $$\text{Res}_{E,S}(\mathbb{G}_{m})(O_{E,S}^{\bigotimes_{O_{F,S}}n}) \to \frac{\text{Res}_{E,S}(\mathbb{G}_{m})}{\mathbb{G}_{m}}(O_{E,S}^{\bigotimes_{O_{F,S}} n})$$ is surjective, since $H^{1}(O_{E,S}^{\bigotimes_{O_{F,S}}n}, \mathbb{G}_{m}) = 0$, by combining Lemma \ref{etalesplitting} with the fact that $H^{1}(O_{E,S}, \mathbb{G}_{m}) = 0$, since $O_{E,S}$ is a PID. We lift a cocycle representing $\alpha_{3}(E, S) \in \check{H}^{2}(O_{E,S}/O_{F,S}, \frac{\text{Res}_{E,S}(\mathbb{G}_{m})}{\mathbb{G}_{m}})$ to $c_{E,S} \in \text{Res}_{E/S}(\mathbb{G}_{m})(O_{E,S}^{\bigotimes_{O_{F,S}} 3})$. Note that 
$$k_{i}'(c_{E,S}) \in \text{Res}_{E/S}(\mathbb{G}_{m})(O_{S} \otimes_{O_{F,S}} O_{E,S} \otimes_{O_{F,S}} O_{E,S}) \xrightarrow{\sim} \prod_{w \in S_{E}} (O_{S} \otimes_{O_{F,S}} O_{E,S} \otimes_{O_{F,S}} O_{E,S})_{w}^{*}.$$


The same argument in \cite[\S 4.3]{Dillery} shows that for any $x \in O_{S} \otimes_{O_{F,S}} O_{E,S} \otimes_{O_{F,S}} O_{E,S}$ and power $p^{m_{i}}$, we may find a $p^{m_{i}}$th root $x^{(1/p^{m})} \in O_{S}^{\text{perf}} \otimes_{O_{F,S}} O_{E,S} \otimes_{O_{F,S}} O_{E,S}$ with $(x^{(1/p^{m_{i+1}})})^{p^{m_{i+1}}/p^{m_{i}}} = x^{(1/p^{m_{i}})}$ for all $i$. Applying this across all $w \in S_{E}$, we may define an analogous map $$ (-)^{(1/p^{m_{i}})} \colon \text{Res}_{E/S}(\mathbb{G}_{m})(O_{S} \otimes_{O_{F,S}} O_{E,S} \otimes_{O_{F,S}} O_{E,S}) \to \text{Res}_{E/S}(\mathbb{G}_{m})(O_{S}^{\text{perf}} \otimes_{O_{F,S}} O_{E,S} \otimes_{O_{F,S}} O_{E,S}).$$ 

We then set $\alpha_{i}(E,S)$ to be the image of $(k_{i}'(c_{E,S}))^{(1/p^{m_{i}})}$ in $[\text{Res}_{E,S}(\mathbb{G}_{m})/\mathbb{G}_{m}]((O_{S}^{\text{perf}})^{\bigotimes_{O_{F,S}}3})$ and obtain $$d \alpha_{i}(E,S) \in Z^{3,2}(O_{S}^{\text{perf}}/O_{F,S}, O_{E,S}, \frac{\text{Res}_{E,S}(\mu_{n_{i}})}{\mu_{n_{i}}})$$ ($Z^{3,2}(O_{S}^{\text{perf}}/O_{F,S}, O_{E,S}, M)$ denotes $3$-cycles lying in $M(O_{S}^{\text{perf}} \otimes_{O_{F,S}} O_{E,S}^{\bigotimes_{O_{F,S}}3})$). Define the map $$\Theta_{E,S} \colon \widehat{H}^{-1}(\Gamma_{E/F}, A^{\vee}[S_{E}]_{0}) \to \check{H}^{2}(O_{S}^{\text{perf}}/O_{F,S}, Z), \hspace{1mm} g \mapsto d\alpha_{i}(E,S) \underset{O_{E,S}/O_{F,S}}\sqcup g,$$ where the unbalanced cup product is as defined in \cite[\S 4.2]{Dillery} using the pairing $$\underline{A^{\vee}[S_{E}]_{0}} \times [\frac{\text{Res}_{E,S}(\mu_{n_{i}})}{\mu_{n_{i}}}]_{O_{S}^{\text{perf}}} \to Z_{O_{S}^{\text{perf}}}$$ given by \eqref{pairing1}, choosing $n_{i}$ divisible by $\text{exp}(Z)$; this map is independent of the choice of $n_{i}$.

As in \cite{Tasho2}, we have the following important lemma which connects the above map to the global Tate duality pairing for tori discussed above (whose corresponding isomorphisms for various tori and Tate cohomology groups will all be denoted by ``TN", for \textit{Tate-Nakayama}):

\begin{lem}\label{bigdiag1} Let $T$ be a torus defined over $F$ and split over $E$, and let $Z \to T$ be an injection with cokernel $\bar{T}$, all viewed as $O_{F,S}$ groups in the usual way. We write $Y = X_{*}(T)$ and $\bar{Y} = X_{*}(\bar{T})$. Then the following diagram commutes, and its columns are exact.
\[
\begin{tikzcd}
\widehat{H}^{-1}(\Gamma_{E/F}, Y[S_{E}]_{0}) \arrow{d} \arrow["\text{TN}"]{r} &  \check{H}^{1}(O_{E,S}/O_{F,S}, T) \arrow{d} \arrow["\sim"]{r} & \check{H}^{1}(O_{S}^{\text{perf}}/O_{F,S}, T) \arrow{d} \\
\widehat{H}^{-1}(\Gamma_{E/F}, \bar{Y}[S_{E}]_{0}) \arrow{d} \arrow["\text{TN}"]{r} &  \check{H}^{1}(O_{E,S}/O_{F,S}, \bar{T})  \arrow["\sim"]{r} & \check{H}^{1}(O_{S}^{\text{perf}}/O_{F,S}, \bar{T}) \arrow["\delta"]{d} \\
\widehat{H}^{-1}(\Gamma_{E/F}, A^{\vee}[S_{E}]_{0}) \arrow["\Theta_{E,S}"]{rr}  \arrow{d} & &  \check{H}^{2}(O_{S}^{\text{perf}}/O_{F,S}, Z) \arrow{d} \\
\widehat{H}^{0}(\Gamma_{E/F}, Y[S_{E}]_{0}) \arrow{d} \arrow["-\text{TN}"]{r} & \check{H}^{2}(O_{E,S}/O_{F,S}, T) \arrow{r} \arrow{d} & \check{H}^{2}(O_{S}^{\text{perf}}/O_{F,S}, T) \arrow{d} \\
\widehat{H}^{0}(\Gamma_{E/F}, \bar{Y}[S_{E}]_{0})  \arrow["-\text{TN}"]{r} & \check{H}^{2}(O_{E,S}/O_{F,S}, \bar{T}) \arrow{r}  & \check{H}^{2}(O_{S}^{\text{perf}}/O_{F,S}, \bar{T}) 
\end{tikzcd}
\]
\end{lem}

\begin{proof} The right-hand isomorphisms on the first two lines follow from the fact that all $T$-torsors over $O_{F,S}$ are trivial over $O_{E,S}$. The ``connecting homomorphism" $\delta$ is the standard connecting homomorphism in \v{C}ech cohomology (see for example \cite[\S 2.2]{Dillery}), and the right-hand column is exact because, applying the isomorphisms $\check{H}^{i}(O_{S}^{\text{perf}}/O_{F,S}, M) \xrightarrow{\sim} H^{i}(O_{F,S}, M)$ for $i=1,2$ and $M=T, \bar{T}, Z$, the resulting two-column diagram commutes, by functoriality of the \v{C}ech-to-derived comparison maps ( \cite[Prop. E.2.1]{Rosengarten}). From here, the identical argument as in \cite{Tasho2} gives the result, using basic properties of the unbalanced cup product in fppf cohomology discussed in \cite[\S 4.2]{Dillery}.
\end{proof} 

\begin{cor}\label{thetainjectivity} The map $\Theta_{E,S}$ is a functorial injection independent of the choices of $c_{E,S}$, $k_{i}$, and $(-)^{(1/p^{m_{i}})}$.
\end{cor}

\begin{proof} As in the proof of \cite[Prop. 3.2.4]{Tasho2}, we may choose $\bar{Y}$ to be a free $\Z[\Gamma_{E/F}]$-module, so that the connecting homomorphism of the left-hand column is injective, and $\Theta_{E,S}$ is the restriction of ``$-\text{TN}$", which is an isomorphism independent of the choices of $c_{E,S}$, $k_{i}$, or $(-)^{(1/p^{m_{i}})}$.
\end{proof}

Recall the local analogue of $\Theta_{E,S}$ which, if $\dot{v} \in S_{F_{S}}$ with restriction to $F$ (and to $E$, by abuse of notation) denoted by $v$ and $c_{v} \in \mathbb{G}_{m}(E_{v}^{\bigotimes_{F_{v}} 3})$ represents the canonical class of $H^{2}(\Gamma_{E_{v}/F_{v}}, E_{v}^{*})$, is defined (in \cite[\S 4.3]{Dillery}) by 
$$\Theta_{E_{v},n_{i}} \colon \widehat{H}^{-1}(\Gamma_{E_{v}/F_{v}}, A^{\vee}) \to \check{H}^{2}(\overline{F_{v}}/F_{v}, Z_{F_{v}}), \hspace{1mm} g \mapsto d\alpha_{v} \underset{E_{v}/F_{v}} \sqcup \Phi_{n_{i}}(g),$$
where $\alpha_{v} \in \overline{F_{v}} \otimes_{F_{v}} \overline{F_{v}} \otimes_{F_{v}} E_{v}$ is an $n_{i}$th-root of $c_{v}$, chosen in an analogous way to $c_{E,S}$ above. This is also a functorial injection, independent of the choices of $i$, $c_{v}$, and $\alpha_{v}$.

To compare these local and global constructions, first note that we have a homomorphism of $\Gamma_{E_{v}/F_{v}}$-modules $A^{\vee}[S_{E}]_{0} \to A^{\vee}$ given by mapping onto the $v$-factor, as well as an $O_{F,S}$-algebra homomorphism $(O_{S}^{\text{perf}})^{\bigotimes_{O_{F,S}} 3} \to \overline{F_{v}}^{\bigotimes_{F_{v}} 3}$ determined by $\dot{v}$, giving a group homomorphism $Z((O_{S}^{\text{perf}})^{\bigotimes_{O_{F,S}} 3}) \to Z(\overline{F_{v}}^{\bigotimes_{F_{v}} 3})$. Then \cite[Lem. 3.2.6]{Tasho2} shows that the resulting square
\[
\begin{tikzcd}
\widehat{H}^{-1}(\Gamma_{E/F}, A^{\vee}[S_{E}]_{0}) \arrow{d}  \arrow["-\text{TN}"]{r} & \check{H}^{2}(O_{S}^{\text{perf}}/O_{F,S}, Z) \arrow{d} \\
\widehat{H}^{-1}(\Gamma_{E_{v}/F_{v}}, A^{\vee}) \arrow["-\text{TN}"]{r} & \check{H}^{2}(\overline{F_{v}}/F_{v}, Z_{F_{v}})
\end{tikzcd}
\]
commutes, where to obtain the right-hand vertical map we use that the map $Z((O_{S}^{\text{perf}})^{\bigotimes_{O_{F,S}} 3}) \to Z(\overline{F_{v}}^{\bigotimes_{F_{v}} 3})$ preserves \v{C}ech cocycles and cochains, which is straightforward to check.

\begin{lem}\label{injinf} The natural map $\check{H}^{2}(O_{S}^{\text{perf}}/O_{S}, Z) \to \check{H}^{2}(\overline{F}/F, Z)$ is injective.
\end{lem}

\begin{proof} The proof of \cite[Lem. 3.2.7]{Tasho2} works verbatim here, replacing $H^{i}(\Gamma_{S}, M(O_{S}))$ with $\check{H}^{i}(O_{S}^{\text{perf}}/O_{F,S}, M)$ for $M = T, \bar{T}, Z$ and $i=1,2$.
\end{proof}

There are obvious transition maps between the groups $A^{\vee}[S_{E}]_{0}$ for varying $S$ and $E$ defined in \cite[\S 3]{Tasho2} inducing morphisms on $\widehat{H}^{-1}$ which are compatible with the homomorphisms $\Theta_{E,S}$ (cf. \cite[Lem. 3.2.8]{Tasho2}).
We then get the main result of this subsection, which characterizes the cohomology group $H^{2}(F, Z)$:

\begin{prop}\label{mainresult1} The maps $\Theta_{E,S}$ splice to a functorial isomorphism $$\Theta \colon \varinjlim_{E} \widehat{H}^{-1}(\Gamma_{E/F}, A^{\vee}[S^{(E)}_{E}]_{0}) \to H^{2}(F, Z),$$ where the limit is over all finite Galois extensions $E/F$ splitting $Z$ and $S^{(E)}$ denotes an arbitrary choice of places of $V$ satisfying Conditions \ref{first2placeconditions} and \ref{thirdplacecondition} for $E/F$ such that if $K/E/F$, we have $S^{(E)} \subset S^{(K)}$ (a mentioned in the preceding discussion, the map in this result does not depend on the choices of the $S^{(E)}$'s by \cite[Lem. 3.2.8]{Tasho2}).
\end{prop}

\begin{proof} This proof closely follows the proof of \cite[Corollary 3.2.9]{Tasho2}. It is enough to prove the result with $H^{2}(F, Z)$ replaced by $\check{H}^{2}(\overline{F}/F, Z)$. By Corollary \ref{thetainjectivity}, Lemma \ref{injinf}, and \cite[Lem. 3.2.8]{Tasho2}, we obtain a functorial injective homomorphism $\Theta$ as claimed, which is independent of the choices of (appropriately chosen) $S^{(E)}$, so all that remains to prove is surjectivity. 

For any $h \in \check{H}^{2}(\overline{F}/F, Z)$, we may find $E'/F$ finite such that $h \in \check{H}^{2}(E'/F, Z)$; denote the separable closure of $F$ in $E'$ by $E$, so that $E' = E \cdot F_{m}$ for some unique $m \in \mathbb{N}$. Moreover, since $E'^{\bigotimes_{F} 3} = \varinjlim_{S^{(E)}} O_{E',S^{(E)}}^{\bigotimes_{O_{F,S^{(E)}}} 3}$, where the direct limit is over all finite $S^{(E)} \subset V$ satisfying the required conditions with respect to $E/F$, there is some finite $S^{(E)}$ satisfying the required conditions with respect to $E/F$ such that we can find $h_{E',S^{(E)}} \in Z(O_{E',S^{(E)}}^{\bigotimes_{O_{F,S^{(E)}}} 3})$ with image in $Z(E'^{\bigotimes_{F} 3}) \to \check{H}^{2}(E'/F, Z)$ equal to $h$. We may enlarge $S^{(E)}$ even further to assume that $h_{E',S^{(E)}} \in Z^{2}(O_{E',S^{(E)}}/O_{F,S^{(E)}}, Z)$, since the \v{C}ech differential on  $Z(O_{E',S^{(E)}}^{\bigotimes_{O_{F,S^{(E)}}} 3})$ is the same as that of $Z(E'^{\bigotimes_{F} 3})$, and we may use finitely many elements of $F$ and $E'$ to encode the fact that $dh_{E',S^{(E)}} = 1$ in $Z(E'^{\bigotimes_{F} 4})$. Denote by $\bar{h}_{E',S^{(E)}}$ the image of $h_{E',S^{(E)}}$ in $\check{H}^{2}(O_{E',S^{(E)}}/O_{F,S^{(E)}}, Z)$. 

Once we have such an $\bar{h}_{E',S^{(E)}}$, choose an $O_{F,S}$-torus $Z \hookrightarrow T$ with $\bar{T} :=T/Z$ such that  $\bar{Y} = X_{*}(\bar{T})$ is free over $\Gamma_{E/F}$, and denote the image of $\bar{h}_{E',S^{(E)}}$ in $\check{H}^{2}(O_{E',S^{(E)}}/O_{F,S^{(E)}}, T)$ by $\bar{h}_{E',S^{(E)},T}$. Note that we have a commutative diagram of isomorphisms from Corollary \ref{EtoEprime}:
\[
\begin{tikzcd}
\check{H}^{2}(O_{E,S^{(E)}}/O_{F,S^{(E)}}, T) \arrow{d} \arrow["\sim"]{r} & \check{H}^{2}(O_{E',S^{(E)}}/O_{F,S^{(E)}}, T) \arrow{d} \\
\check{H}^{2}(O_{E,S^{(E)}}/O_{F,S^{(E)}}, \bar{T}) \arrow["\sim"]{r} & \check{H}^{2}(O_{E',S^{(E)}}/O_{F,S^{(E)}}, \bar{T}),
\end{tikzcd}
\]
and so we may pick a (unique) preimage, denoted by $\bar{h}_{E, S^{(E)},T}$, of $\bar{h}_{E',S^{(E)},T}$ in $\check{H}^{2}(O_{E,S}/O_{F,S}, T)$, and by the commutativity of the diagram, the image of $\bar{h}_{E,S^{(E)},T}$ in $\check{H}^{2}(O_{E,S^{(E)}}/O_{F,S^{(E)}}, \bar{T})$ is zero. We may thus lift $\text{-TN}^{-1}(\bar{h}_{E,S^{(E)},T}) \in \widehat{H}^{0}(\Gamma_{E/F}, Y[S^{(E)}_{E}]_{0})$ to some $g \in \widehat{H}^{-1}(\Gamma_{E/F}, A^{\vee}[S_{E}^{(E)}]_{0})$, and then the same argument as in \cite[Corollary 3.2.9]{Tasho2} shows that $\Theta_{E,S^{(E)}}(g) \in \check{H}^{2}(O_{S^{(E)}}^{\text{perf}}/O_{F,S^{(E)}}, Z)$ has image in $\check{H}^{2}(\overline{F}/F, Z)$ equal to $h$, as desired (even though we need to take the image of $\bar{h}_{E',S^{(E)},T}$ in $\check{H}^{2}(O_{E,S}/O_{F,S}, T)$, the argument of \cite{Tasho2} uses that the image of their $\Theta_{E,S^{(E)}}(g)$ in $\check{H}^{2}(\overline{F}/F, T)$ is the same as that of $h$, which is still true for our $g$ obtained via the above adjustment for non-separability).
\end{proof}

\subsection{The groups $P_{E, \dot{S}_{E}, n}$}
Let $E/F$ be a finite Galois extension, $S \subset V$ a finite full set of places, and $\dot{S}_{E} \subseteq S_{E}$ a set of lifts for the places in $S$. When dealing with a multiplicative $O_{F,S}$-group $M$, we will frequently work with $\check{H}^{2}(O_{S}^{\text{perf}}/O_{F,S}, M)$ rather than $H^{2}_{\text{fppf}}(O_{F,S}, Z)$; these two groups are canonically isomorphic by Corollary \ref{ComparisonIso2}. We assume that the pair $(S, \dot{S}_{E})$ satisfies the following:

\begin{cond}\label{placeconditions} \begin{enumerate} \item{$S$ contains all places that ramify in $E$.}
\item{Every ideal class of $E$ contains an ideal with support in $S_{E}$ (i.e., $\text{Cl}(O_{E,S}) = 0$).}
\item{For every $w \in V_{E}$, there exists $w' \in S_{E}$ with $\text{Stab}(w, \Gamma_{E/F}) = \text{Stab}(w', \Gamma_{E/F})$.}
\item{For every $\sigma \in \Gamma_{E/F}$, there exists $\dot{v} \in \dot{S}_{E}$ such that $\sigma \dot{v} = \dot{v}$.}
\end{enumerate}
\end{cond}

Pairs $(S, \dot{S}_{E})$ satisfying these conditions always exist, and if $(S', \dot{S}'_{E})$ contains $(S, \dot{S}_{E})$ (in the obvious sense) and the latter satisfies these conditions, then so does the former. For notational ease, denote the group $\frac{\text{Res}_{E,S}(\mu_{n})}{\mu_{n}}$ introduced in the previous subsection by $\bar{R}_{E,S}[n]$. For a fixed $n \in \mathbb{N}$, we first set $P_{E,S,n}$ to be the multiplicative $O_{F,S}$-group split over $O_{E,S}$ corresponding to the $\Gamma_{E/F}$-module $M_{E,S,n} = \frac{1}{n}\Z/\Z[\Gamma_{E/F} \times S_{E}]_{0,0}$ consisting of elements of $\frac{1}{n}\Z/\Z[\Gamma_{E/F} \times S_{E}]$ killed by both augmentation maps. We define the multiplicative group $P_{E, \dot{S}_{E}, n}$ to correspond to the $\Gamma_{E/F}$-submodule $M_{E, \dot{S}_{E}, n} \subset M_{E,S,n}$ of elements $x$ such that $x[(\sigma, w)] = 0$ if $w \notin \sigma(\dot{S}_{E})$. We have the following linear-algebraic result \cite[Lem. 3.3.2]{Tasho2}:

\begin{lem}\label{threepartlemma} Let $A$ be a $\Z[\Gamma_{E/F}]$-module which is finite as an abelian group.
\begin{enumerate}
\item{If $\text{exp}(A)$ divides $n$, then we may define a functorial (in $A$) isomorphism $$\Psi_{E,S,n} \colon \Hom(A, M_{E,S,n})^{\Gamma} \to \widehat{Z}^{-1}(\Gamma_{E/F}, A^{\vee}[S_{E}]_{0}), H \mapsto h := \sum_{w \in S_{E}}  h_{w} [w],$$ where $h_{w} \colon A \to \mathbb{Q}/\Z, a \mapsto H(a)[(e,w)]$ which restricts to an isomorphism $ \Hom(A, M_{E,\dot{S}_{E}, n})^{\Gamma} \xrightarrow{\sim} A^{\vee}[\dot{S}_{E}]_{0} \cap \widehat{Z}^{-1}(\Gamma_{E/F}, A^{\vee}[S_{E}]_{0}).$}

\item{For $n \mid m$, the isomorphisms $\Psi_{E,S,n}$ and $\Psi_{E,S,m}$ are compatible with the natural inclusion $M_{E,\dot{S}_{E},n} \to M_{E, \dot{S}_{E}, m}$. Setting $M_{E,S} := \varinjlim_{n} M_{E,S,n}$, we thus obtain an isomorphism $$\Psi_{E,S} \colon \Hom(A, M_{E,S})^{\Gamma} \to \widehat{Z}^{-1}(\Gamma_{E/F}, A^{\vee}[S_{E}]_{0}).$$}

\item{We have an induced surjective map $$A^{\vee}[\dot{S}_{E}]_{0} \cap \widehat{Z}^{-1}(\Gamma_{E/F}, A^{\vee}[S_{E}]_{0}) \to \widehat{H}^{-1}(\Gamma_{E/F}, A^{\vee}[S_{E}]_{0}).$$}
\end{enumerate}
\end{lem}


Now for fixed $n \in \mathbb{N}$ and $A$ a $\Z[\Gamma_{E/F}]$-module which is finite as an abelian group with corresponding $O_{F,S}$-group $Z$ such that $\text{exp}(A)$ divides $n$, we obtain a map 
\begin{equation}\label{thetaP}
\Theta^{P}_{E,\dot{S}_{E}, n} \colon \Hom(P_{E, \dot{S}_{E}, n}, Z)^{\Gamma} \xrightarrow{\Theta_{E,S} \circ \Psi_{E,S,n}} \check{H}^{2}(O_{S}^{\text{perf}}/O_{F,S}, Z);
\end{equation}
note that this map is functorial in the group $Z$. For $A = M_{E,\dot{S}_{E},n}$, we have the canonical element $\text{id}$ of the left-hand side of \eqref{thetaP}, and we define $\xi_{E,\dot{S}_{E}, n} \in \check{H}^{2}(\OSp/O_{F,S}, P_{E,\dot{S}_{E},n})$ as its image. 

We will now see how the groups $P_{E,\dot{S}_{E}, n}$ behave for varying $E/F$. For $(S', \dot{S}'_{K})$ satisfying Conditions \ref{placeconditions} with respect to the finite Galois extension $K/F$ and $m \in \mathbb{N}$, we write $(E, \dot{S}_{E}, n) < (K, \dot{S}'_{K}, m) $ when $n \mid m$, $K$ contains $E$, $S \subseteq S'$, and $\dot{S}_{E} \subseteq (\dot{S}'_{K})_{E}$. Given $E$, $(S, \dot{S}_{E})$, and $K$, one can always find $(S', \dot{S}'_{K})$ with $(E, \dot{S}_{E}, n) < (K, \dot{S}'_{K}, m)$. For $(E, \dot{S}_{E}, n) < (K, \dot{S}'_{K}, m)$, we define a map of $\Gamma_{K/F}$-modules from $M_{E, \dot{S}_{E}, n}$ to $M_{K, \dot{S}'_{K}, m}$ (with inflated action on the former) given by $$\sum_{(\sigma, w) \in \Gamma_{E/F} \times S_{E}} a_{\sigma,w} [(\sigma, w)] \mapsto \sum_{(\gamma, u)} a_{\bar{\gamma}, u_{E}} [(\gamma, u)],$$ where the right-hand sum is over all pairs $(\gamma, u)$ in $\Gamma_{K/F} \times S'_{K}$ such that $\gamma^{-1}u \in \dot{S}'_{K} \cap S_{K}$, and $\bar{\gamma}$ denotes the image of $\gamma$ in $\Gamma_{E/F}$. It follows from \cite[Lem. 3.3.4]{Tasho2} that these transition maps are compatible with the homomorphisms $\Theta^{P}$. According to \cite[Lems. 3.3.3, 3.3.5]{Tasho2}, we get the hoped-for coherence between the canonical classes $\xi_{E, \dot{S}_{E}, n}$ discussed above:

\begin{lem}\label{Ptheta2}\label{coherence1} The homomorphism $\check{H}^{2}(O_{S'}^{\text{perf}}/O_{F,S'}, P_{K, \dot{S}'_{K}, m}) \to \check{H}^{2}(O_{S'}^{\text{perf}}/O_{F,S'}, (P_{E, \dot{S}_{E}, n})_{O_{F,S'}})$ maps $\xi_{K, \dot{S}'_{K}, m}$ to the image of $\xi_{E, \dot{S}_{E}, n}$ under the inflation map $$\check{H}^{2}(O_{S}^{\text{perf}}/O_{S}, P_{E, \dot{S}_{E}, n}) \to \check{H}^{2}(O_{S'}^{\text{perf}}/O_{S'}, (P_{E, \dot{S}_{E,n}})_{O_{F,S'}}).$$
\end{lem}

Fix a system of quadruples $(E_{i}, S_{i}, \dot{S}_{i},n_{i})_{i \in \mathbb{N}}$ such that $(S_{i}, \dot{S}_{i})$ satisfies Conditions \ref{placeconditions} with respect to the finite Galois extension $E_{i}/F$, the $E_{i}$ form an exhaustive tower of finite Galois extensions of $F$, the $S_{i}$ form an exhaustive tower of finite subsets of $V$, the $n_{i}$ form a cofinal system in $\mathbb{N}^{\times}$, we have the containment $\dot{S}_{i} \subseteq (\dot{S}_{i+1})_{E_{i}}$ for all $i$, and $n_{i} \mid n_{i+1}$ for all $i$. Such a system evidently exists. Note that $\dot{V} := \varprojlim_{i} \dot{S}_{i}$ is a subset of $V_{F^{\text{sep}}}$ of lifts of $V$, and the group $$P_{\dot{V}} := \varprojlim_{i} P_{E_{i}, \dot{S}_{i}, n_{i}}$$ is a pro-algebraic group over $F$. For a finite multiplicative $Z$ over $F$, the maps $\{\Theta^{P}_{E_{i}, \dot{S}_{i}, n_{i}}\}_{i}$ induce $$\Theta^{P}_{\dot{V}} \colon \Hom(P_{\dot{V}}, Z)^{\Gamma} \to \check{H}^{2}(\overline{F}/F, Z) (= H^{2}(F,Z)),$$ which factors through the homomorphisms 
\begin{equation}\label{ThetaPcomp} \Hom(P_{E_{i}, \dot{S}_{i}, n_{i}}, Z)^{\Gamma} \xrightarrow{\Theta^{P}_{E_{i}, \dot{S}_{i}, n_{i}}} \check{H}^{2}(O_{S_{i}}^{\text{perf}}/O_{F,S_{i}}, Z) \to \check{H}^{2}(\overline{F}/F, Z)
\end{equation}
for all sufficiently large $i$, and hence is surjective, since we may choose $i$ with $\text{exp}(Z) \mid n_{i}$ and invoke Lemma \ref{threepartlemma} and Proposition \ref{mainresult1}  to deduce the surjectivity of \eqref{ThetaPcomp} for all  $j\gg i$.

From \cite[Lem. 3.3.6]{Tasho2}, we have the following alternative characterization of $\Hom_{F}(P_{\dot{V}}, Z)$:
\begin{lem}\label{Aveedot} Let $Z$ be a finite multiplicative $F$-group, $A = X^{*}(Z)$, and $A^{\vee}[\dot{V}]_{0}$ the kernel of the augmentation map $A^{\vee}[\dot{V}] \to A^{\vee}$. Then we have a natural isomorphism $\Hom_{F}(P_{\dot{V}}, Z) \xrightarrow{\sim} A^{\vee}[\dot{V}]_{0}$.
\end{lem}

We conclude this subsection by discussing some local-global compatibility regarding $P_{\dot{V}}$ and its local analogues $u_{v}$ from \cite[\S 3.1]{Dillery}. For a fixed place $v \in \dot{V}$, recall the multiplicative $F_{v}$-groups $$u_{E_{v}/F_{v}, n} := \frac{\text{Res}_{E_{v}/F_{v}}(\mu_{n})}{\mu_{n}}, u_{v} := \varprojlim_{E_{v}/F_{v}, n} u_{E_{v}/F_{v}, n}.$$ For $Z$ a finite multiplicative $F_{v}$-group with $\text{exp}(Z) \mid n$, there is an isomorphism (defined similarly to the map $\Psi_{E,S,n}$ above, cf. \cite[\S 3.1]{Dillery}) $$\Hom_{F_{v}}(u_{E_{v}/F_{v}, n}, Z) \xrightarrow{\Psi_{E_{v},n}} \widehat{Z}^{-1}(\Gamma_{E_{v}/F_{v}}, A^{\vee}).$$

We now define a localization map $\text{loc}_{v}^{P} \colon u_{v} \to (P_{\dot{V}})_{F_{v}}$ for a fixed $v \in \dot{V}$. Fix $E/F$ a finite Galois extension along with a triple $(S, \dot{S}_{E}, n)$ such that $(S, \dot{S}_{E})$ satisfies Conditions \ref{placeconditions} with respect to $E/F$. 
There is a morphism of $\Gamma_{E_{v}/F_{v}}$-modules  $$\text{loc}_{v}^{M_{E,\dot{S}_{E},n}} \colon M_{E, \dot{S}_{E}, n} \to X^{*}(u_{E_{v}/F_{v}, n}) = \frac{1}{n}\Z/\Z[\Gamma_{E_{v}/F_{v}}]_{0},$$ given by $$H = \sum_{(\sigma, w) \in \Gamma_{E/F} \times S_{E}} c_{\sigma,w}[(\sigma,w)] \mapsto \sum_{(\sigma,v), \sigma \in \Gamma_{E_{v}/F_{v}}} c_{\sigma, v}[\sigma] := H_{v}. $$ 

Denote by $u_{E_{v}/F_{v}, n} \xrightarrow{\text{loc}_{v}^{P_{E, \dot{S}_{E}, n}}} (P_{E, \dot{S}_{E}, n})_{F_{v}}$ the corresponding morphism of group schemes, which glues as we range over all $4$-tuples $(E_{i}, S_{i}, \dot{S}_{i}, n_{i})$, inducing a homomorphism of profinite $F_{v}$-groups $\text{loc}_{v}^{P} \colon u_{v} \to (P_{\dot{V}})_{F_{v}},$ as desired. 

There is a local analogue of the map $\Theta_{\dot{V}}^{P}$ constructed above, which we denote by $\Hom_{F_{v}}(u_{v}, Z) \xrightarrow{\Theta_{v}} \check{H}^{2}(\overline{F_{v}}/F_{v}, Z)$. These local maps agree with the global analogue after localization:

\begin{lem}(\cite[Lem. 3.3.7]{Tasho2} \label{localtoglobal1} For $E/F$ finite Galois splitting $Z$, $(S, \dot{S}_{E})$ satisfying Conditions \ref{placeconditions} with respect to $E$, $n \in \mathbb{N}$ a multiple of $\text{exp}(Z)$, and $\dot{v} \in \dot{V}$ (with $\dot{v}_{F}, \dot{v}_{E} =:v$, by abuse of notation), the following diagram commutes
\[
\begin{tikzcd}
\Hom_{F}(P_{E, \dot{S}_{E},n}, Z) \arrow["\text{loc}_{v}^{P_{E,\dot{S}_{E}, n}}"]{d} \arrow["\Theta^{P}_{E,\dot{S}_{E}, n}"]{r} & \check{H}^{2}(\overline{F}/F, Z) \arrow{d} \\
\Hom_{F_{v}}(u_{E_{v}/F_{v}, n}, Z_{F_{v}}) \arrow["\Theta^{u}_{E_{v},n}"]{r} & \check{H}^{2}(\overline{F_{v}}/F_{v}, Z_{F_{v}}),
\end{tikzcd}
\]
where the right vertical map is induced by the inclusion $\overline{F} \to \overline{F_{v}}$ determined by $\dot{v}$.
\end{lem}

Recall from Lemma \ref{coherence1} that elements $\xi_{i} := \xi_{E_{i}, \dot{S}_{i}, n_{i}}$ form a coherent system in the projective system of groups $\{\check{H}^{2}(O_{S_{i}}^{\text{perf}}/O_{F,S_{i}}, P_{E_{i}, \dot{S}_{i}, n_{i}})\}_{i}$. We also have (by Lemma \ref{injinf}), for all $i$, injective homomorphisms $\check{H}^{2}(O_{S_{i}}^{\text{perf}}/O_{F,S_{i}}, P_{E_{i}, \dot{S}_{i}, n_{i}}) \to \check{H}^{2}(\overline{F}/F, P_{E_{i}, \dot{S}_{i}, n_{i}})$, and hence the element $(\xi_{i})_{i}$ may be viewed as an element of $\varprojlim_{i} \check{H}^{2}(\overline{F}/F, P_{E_{i}, \dot{S}_{i}, n_{i}})$. Let $\xi_{v} \in \check{H}^{2}(\overline{F_{v}}/F_{v}, u_{v}) \xrightarrow{\sim} \widehat{\Z}$ (\cite[Thm. 3.4]{Dillery}) denote the canonical class obtained by taking the preimage of $-1 \in \widehat{\Z}$. We may now deduce the final result of this subsection:

\begin{cor}\label{tasho3.3.8} For $\dot{v} \in \dot{V}$, consider the maps $$\check{H}^{2}(\overline{F}/F, P_{\dot{V}}) \to \check{H}^{2}(\overline{F_{v}}/F_{v}, (P_{\dot{V}})_{F_{v}}) \leftarrow \check{H}^{2}(\overline{F_{v}}/F_{v}, u_{v}), $$ where the left map is induced by the inclusion $\overline{F} \to \overline{F_{v}}$ determined by $\dot{v}$ and the right map is $\text{loc}_{v}^{P}$.  If $\tilde{\xi} \in H^{2}(\overline{F}/F, P_{\dot{V}})$ is any preimage of $(\xi_{i})$ (which exists by Lemma \ref{basicresult}) then the images of $\tilde{\xi}$ and $\xi_{v}$ in the middle term are equal.
\end{cor}

\begin{proof} We claim first that the natural map $\check{H}^{2}(\overline{F_{v}}/F_{v}, (P_{\dot{V}})_{F_{v}}) \to \varprojlim_{i} \check{H}^{2}(\overline{F_{v}}/F_{v}, P_{E_{i}, \dot{S}_{i}, n_{i}})$ is an isomorphism. To simplify notation, set $P_{i} := (P_{E_{i}, \dot{S}_{i}, n_{i}})_{F_{v}}$. By Proposition \ref{cechidentification}, it suffices to show that $\varprojlim_{i}^{(1)} \check{H}^{1}(\overline{F_{v}}/F_{v}, P_{i}) = 0$ and $\varprojlim_{i} d(P_{i}(\overline{F_{v}})) = 0$; the latter trivially follows because it's a system of finite groups. We now explain the vanishing of $\varprojlim_{i}^{(1)} \check{H}^{1}(\overline{F_{v}}/F_{v}, P_{i})$ (all group schemes are of finite type so we replace the notation ``$\check{H}^{1}(\overline{F_{v}}/F_{v}, -)$'' with ``$H^{1}(F_{v}, -)$'' for notational convenience):

As in \cite[\S 3.4]{Tasho2} we may fit any $M_{E, \dot{S}_{E}, n}$ into a sequence (where the middle $0$-subscript denotes the kernel of the $\Gamma_{E/F}$-augmentation map)
\begin{equation}\label{JIMJeq1}
0 \to M_{E, \dot{S}_{E}, n} \to \frac{1}{n}\Z/\Z[\Gamma_{E/F} \times S]_{0} \to \frac{1}{n}\Z/\Z[S_{E}]_{0} \to 0.
\end{equation}
These identifications are compatible with the transition maps, which on the right-most terms are $$\sum_{(\sigma, w) \in \Gamma_{E/F} \times S_{E}} a_{\sigma, w}[(\sigma, w)] \mapsto \sum_{(\gamma, u) \in \Gamma_{K/F} \times S_{K}'} (\# \Gamma_{K/E}^{u})a_{\bar{\gamma}, u_{E}}[(\gamma, u)].$$
We have the exact sequence
 \begin{equation*} H^{1}(F_{v}, A_{i}) \to H^{1}(F_{v}, P_{i}) \to C_{i},
 \end{equation*}
 where $A_{i}$ denotes the multiplicative group scheme Cartier dual to $\frac{1}{n_{i}}\Z/\Z[\Gamma_{E_{i}/F} \times S_{i}]_{0} $ and $C_{i}$ is the image of $H^{1}(F_{v}, P_{i})$ in $H^{2}(F_{v}, B_{i})$, where $B_{i}$ is the kernel of $A_{i} \to P_{i}$. Note that $\varprojlim^{(1)} C_{i} = 0$, since these are all finite groups (because $H^{2}(F_{v}, B_{i})$ is, using local Poitou-Tate duality), so it suffices to prove that $\varprojlim^{(1)} H^{1}(F_{v}, A_{i}) = 0$. 

Denoting by $\tilde{A}_{i}$ the multiplicative group scheme Cartier dual to  $\frac{1}{n_{i}}\Z/\Z[\Gamma_{E_{i}/F} \times S_{i}]$, we have an exact sequence
\begin{equation*}
H^{1}(F_{v}, \tilde{A}_{i}) \to H^{1}(F_{v}, A_{i}) \to \tilde{C}_{i},
\end{equation*}
where $\tilde{C}_{i}$ is the image of $H^{1}(F_{v}, A_{i}) \to H^{2}(F, \tilde{B}_{i})$ and $\tilde{B}_{i}$ is the kernel of $\tilde{A}_{i} \to A_{i}$ (it's Cartier dual to $\frac{1}{n_{i}}\Z/\Z[\Gamma_{E_{i}/F}]$). Since $H^{2}(F_{v}, B_{i})$ is finite, we have $\varprojlim^{(1)} \tilde{C}_{i} =0$, and we thus reduce further to showing that $\varprojlim^{(1)} H^{1}(F_{v}, \tilde{A}_{i}) = 0$.

Note that $\tilde{A}_{i} = \prod_{S_{i}} [\mathrm{Res}_{E_{i}/F}(\mu_{n_{i}})]_{F_{v}}$. We first compute $H^{1}(F_{v}, \mathrm{Res}_{E_{i}/F}(\mu_{n_{i}}))$; the Mackey formula and Shapiro's lemma tell us that 
\begin{equation}\label{JIMJeq2}
H^{1}(\Gamma_{v}, \frac{1}{n_{i}}\Z/\Z[\Gamma_{E_{i}/F}]) = \bigoplus_{w \in V_{E_{i}}, w \mid v} H^{1}(\Gamma_{E_{i,w}}, \frac{1}{n_{i}}\Z/\Z),
\end{equation}
and so by local Poitou-Tate duality we have $H^{1}(F_{v}, \mathrm{Res}_{E_{i}/F}(\mu_{n_{i}})) = \bigoplus_{w \in V_{E_{i}}, w \mid v} H^{1}(E_{i,w}, \mu_{n_{i}})$. It follows that 
\begin{equation*}
H^{1}(F_{v}, \tilde{A}_{i})  = \bigoplus_{S_{i}} \bigoplus_{w \in V_{E_{i}}, w \mid v} H^{1}(E_{i,w}, \mu_{n_{i}}),
\end{equation*}
and we want to show that $\varprojlim^{(1)}$ of this sequence of groups vanishes. From here, one can use the verbatim argument in the proof of \cite[Proposition 3.1]{Dillery} (which uses a computation involving the explicit resolution computing $\varprojlim^{(1)}$ as in \cite[\S 3.5]{Weibel}) to obtain the desired result, finally proving the claimed vanishing.


The isomorphism we just proved implies that the map $\check{H}^{2}(\overline{F}/F, P_{\dot{V}}) \to \check{H}^{2}(\overline{F_{v}}/F_{v}, (P_{\dot{V}})_{F_{v}})$ factors as the composition $$\check{H}^{2}(\overline{F}/F, P_{\dot{V}}) \to \varprojlim_{i} \check{H}^{2}(\overline{F}/F, P_{E_{i}, \dot{S}_{i}, n_{i}}) \to \varprojlim_{i} \check{H}^{2}(\overline{F_{v}}/F_{v}, (P_{E_{i}, \dot{S}_{i}, n_{i}})_{F_{v}}),$$ where the second map is the inverse limit of the localizations for each $i$. It is thus enough to show that, for each $i$, the map $\text{loc}_{v}^{P_{E_{i}, \dot{S}_{i}, n_{i}}}$ sends $\xi_{(E_{i})_{v}, n_{i}} \in \check{H}^{2}(\overline{F_{v}}/F_{v}, u_{(E_{i})_{v}/F_{v}, n_{i}})$ to the image of $\xi_{E_{i}, \dot{S}_{i}, n_{i}}$ under the map $\check{H}^{2}(\overline{F}/F, P_{E_{i}, \dot{S}_{i}, n_{i}}) \to \check{H}^{2}(\overline{F_{v}}/F_{v}, (P_{E_{i}, \dot{S}_{i}, n_{i}})_{F_{v}})$.  Once we have reached this step, we get the result from the proof of \cite[Corollary 3.8]{Tasho2}.
\end{proof}

\subsection{The vanishing of $H^{1}(F, P_{\dot{V}})$ and $H^{1}(F_{v}, (P_{\dot{V}})_{F_{v}})$ }
In the local case, an instrumental property of the groups $u_{v}$ was that $H^{1}(F, u_{v}) = 0$; our goal in this subsection is to prove the analogue for $P_{\dot{V}}$ and its localizations. 

Recall the equation \eqref{JIMJeq1} used in the proof of Corollary \ref{tasho3.3.8} above. Observe (cf. \cite[Lem. 3.4.3]{Tasho2}) that, for fixed $(E, \dot{S}_{E}, n)$ the right-most transition maps in \eqref{JIMJeq1} vanish for all sufficiently large $(K, \dot{S}'_{K}, m) > (E, \dot{S}_{E}, n)$. 
We may now deduce some preliminary cohomological vanishing:

\begin{lem}\label{vanishingcolimit2} The following colimits over $(E, \dot{S}_{E}, n)$ vanish:
\begin{enumerate}
\item{$\varinjlim H^{1}(\Gamma,  \frac{1}{n}\Z/\Z[\Gamma_{E/F} \times S]_{0})$;}
\item{$\varinjlim H^{1}(\Gamma_{v},  \frac{1}{n}\Z/\Z[\Gamma_{E/F} \times S]_{0})$ for any $v \in V$.}
\end{enumerate}
\end{lem}

\begin{proof} The first vanishing follows from the proof of \cite[Lem. 3.4.4]{Tasho2}. For the second vanishing, we first use the identity \eqref{JIMJeq2}; identifying each $H^{1}(\Gamma_{E_{w}}, \frac{1}{n}\Z/\Z)$ with $\text{Hom}(\Gamma_{E_{w}}, \frac{1}{n}\Z/\Z)$, the transition map $$\bigoplus_{w \in V_{E}, w \mid v} \text{Hom}(\Gamma_{E_{w}}, \frac{1}{n}\Z/\Z) \to \bigoplus_{u \in V_{K}, u \mid v} \text{Hom}(\Gamma_{K_{u}}, \frac{1}{m}\Z/\Z)$$ is given by the maps $\text{Hom}(\Gamma_{E_{w}}, \frac{1}{n}\Z/\Z)  \to \bigoplus_{u \mid w} \Hom(\Gamma_{K_{u}}, \frac{1}{m}\Z/\Z)$ induced by the inclusions $\Gamma_{K_{u}} \hookrightarrow \Gamma_{E_{w}}$. For a fixed $f_{w} \in \text{Hom}(\Gamma_{E_{w}}, \frac{1}{n}\Z/\Z)$, the kernel $H_{f_{w}}$ of $f_{w}$ is an open normal subgroup of $\Gamma_{E_{w}}$, and so if $K/E$ is a large enough finite Galois extension, we have $\Gamma_{K_{u}} \subseteq H_{f_{w}}$ for all $u \mid w$ places of $K$. Note that, given such a $K$, this property also holds for any $K'/K/F$ finite Galois and $\tilde{u} \mid w$ a place of $K'$. Now for any $(f_{w}) \in \bigoplus_{w \mid v} \text{Hom}(\Gamma_{E_{w}}, \frac{1}{n}\Z/\Z)$ we can apply the previous sentence to each $w$ and take a compositum to find a finite Galois $K/F$ such that for any $w$ and $u \in V_{K}$ with $u \mid w$, we have $\Gamma_{K_{u}} \subseteq H_{f_{w}}$. Thus, the image of $(f_{w})_{w}$ in $\bigoplus_{u \mid v} \text{Hom}(\Gamma_{K_{u}}, \frac{1}{n}\Z/\Z)$ is trivial, proving (2).
\end{proof}

\begin{prop}\label{mainvanishing1} For any $v \in \dot{V}$, we have $H^{1}(F_{v}, (P_{\dot{V}})_{F_{v}}) = 0$.
\end{prop}

\begin{proof} Note that $H^{1}(F_{v}, (P_{\dot{V}})_{F_{v}}) = \varprojlim_{i} H^{1}(F_{v}, (P_{E_{i}, \dot{S}_{i}, n_{i}})_{F_{v}})$, since $\varprojlim_{i}^{(1)} H^{0}(F_{v}, (P_{E_{i}, \dot{S}_{i}, n_{i}})_{F_{v}})$ is trivial (the same argument works for $F$ instead of $F_{v}$). Thus, local Poitou-Tate duality gives $$H^{1}(F_{v}, (P_{\dot{V}})_{F_{v}}) = \varprojlim_{i} (H^{1}(\Gamma_{v}, M_{E_{i}, \dot{S}_{i}, n_{i}})^{*}) =  (\varinjlim_{i} H^{1}(\Gamma_{v}, M_{E_{i}, \dot{S}_{i}, n_{i}}))^{*}.$$ Now we have (using \eqref{JIMJeq1}) the exact sequence $$0 \to C_{i} \to H^{1}(\Gamma_{v}, M_{E_{i}, \dot{S}_{i}, n_{i}}) \to H^{1}(\Gamma_{v}, \frac{1}{n_{i}}\Z/\Z[\Gamma_{E_{i}/F} \times S_{i}]_{0}),$$ where $C_{i}$ is a subquotient of $\frac{1}{n_{i}}\Z/\Z[(S_{i})_{E_{i}}]_{0}$, and the colimits of the outer two terms are zero, by Lemma \ref{vanishingcolimit2} and the remark preceding it, giving the result.
\end{proof}

We need to recall a result from global class field theory. Let $\overline{C} := \varprojlim_{K/F} \overline{C_{K}},$ where $\overline{C_{K}}$ is the profinite completion of the id\'{e}le class group $C_{K}$ of the finite Galois extension $K/F$, and the limit is over all such extensions. For fixed $K/F$ finite Galois and $n \in \mathbb{N}$, note that $\overline{C}[n]^{\Gamma_{K}} = \overline{C_{K}}[n]$. 

\begin{lem}\label{profinite1} The completed universal norm group $\overline{N}:=  \varprojlim_{K/F} N_{K/F}(\overline{C_{K}})$ is trivial.
\end{lem}

\begin{proof}  For any finite Galois $K/F$, we have the exact sequence 
$$0 \to N_{K/F}(C_{K}) \to C_{F} \xrightarrow{(-,K/F)} \Gamma_{K/F}^{\text{ab}} \to 0.$$ Since the group $N_{K/F}(C_{K})$ is open of finite index in $C_{F}$, the inverse limit over all open subgroups of $C_{F}$ of finite index may be taken over all open subgroups of finite index which lie in $N_{K/F}(C_{K})$, and for any such subgroup $U$, we get the exact sequence $$0 \to \frac{N_{K/F}(C_{K})}{U} \to \frac{C_{F}}{U} \to \Gamma_{K/F}^{\text{ab}} \to 0,$$ which after applying the (left-exact) functor $\varprojlim(-)$ yields the exact sequence $$0 \to N_{K/F}(C_{K})^{\wedge} \to \overline{C_{F}} \to \Gamma_{K/F}^{\text{ab}} \to 0;$$ note that surjectivity is preserved because the kernels are all finite groups. Now since $C_{K}$ is dense in $\overline{C_{K}}$, we have that $N_{K/F}(\overline{C_{K}}) = N_{K/F}(C_{K})^{\wedge}$ inside $\overline{C_{F}}$, by continuity of the norm map. Applying the inverse limit over all finite Galois $K/F$ then yields the exact sequence $$0 \to \overline{N} \to \overline{C_{F}} \to \Gamma^{\text{ab}}$$ so it's enough to show that the completed universal residue map $\overline{C_{F}} \to \Gamma^{\text{ab}}$ is injective, which is a basic fact of global class field theory (see e.g. \cite[Prop. 8.1.26]{NSW}).
\end{proof}

We move on to a slightly more involved vanishing result:
\begin{lem}\label{vanishingcolimit3} The following colimit over $(E, \dot{S}_{E}, n)$ vanishes:
$$\varinjlim H^{2}(\Gamma,  \frac{1}{n}\Z/\Z[\Gamma_{E/F} \times S]_{0}) = 0.$$
\end{lem}

\begin{proof}

As in the proof of Lemma \ref{vanishingcolimit2},  it is enough to show that the colimit $\varinjlim H^{2}(\Gamma_{E}, \frac{1}{n}\Z/\Z)$ vanishes, with the transition maps given by the restriction homomorphism. For $(E,n)$ fixed, by \cite[Thm. 8.4.4]{NSW} (with $S=V_{E}$), we have an isomorphism $$H^{2}(\Gamma_{E}, \frac{1}{n}\Z/\Z) \xrightarrow{\sim} (\widehat{H}^{0}(\Gamma_{E}, \overline{C}[n]))^{\vee},$$ where $\widehat{H}^{0}(\Gamma_{E}, \overline{C}[n]) := \varprojlim_{K/E} \widehat{H}^{0} (\Gamma_{K/E}, \overline{C}[n]^{\Gamma_{K}})$, with transition maps given by the projections $$ \frac{\overline{C_{E}}[n]}{N_{K'/E}( \overline{C_{K'}}[n])} \to  \frac{\overline{C_{E}}[n]}{N_{K/E}( \overline{C_{K}}[n])};$$ recall that for $M$ a locally-compact Hausdorff topological group, $M^{\vee}$ denotes $\Hom_{\text{cts}}(M, \mathbb{R}/\Z)$. 

We claim that the natural map $\overline{C_{E}}[n] \to \widehat{H}^{0} (\Gamma_{E}, \overline{C}[n])$ is an isomorphism, which follows from the fact that both $\varprojlim_{K/E} N_{K/E}( \overline{C_{K}}[n])$ and $\varprojlim^{(1)} N_{K/E}( \overline{C_{K}}[n])$ vanish. For the former, there is an inclusion $N_{K/E}( \overline{C_{K}}[n]) \hookrightarrow N_{K/E}(\overline{C_{K}})[n]$, and hence an inclusion $$\varprojlim N_{K/E}( \overline{C_{K}}[n])  \hookrightarrow (\varprojlim N_{K/E}(\overline{C_{K}})))[n] = 0,$$ where the last term equals zero by Lemma \ref{profinite1}. Moreover, one checks using the explicit resolution computing $\varprojlim^{(1)}$ along with Lemma \ref{profinite1} and the compactness of $\overline{C_{E}}$ that $\varprojlim^{(1)} N_{K/E}( \overline{C_{K}}[n])= 0$. In conclusion, we obtain an isomorphism $H^{2}(\Gamma_{E}, \frac{1}{n}\Z/\Z) \xrightarrow{\sim} \overline{C_{E}}[n]^{\vee}$. 

It thus suffices to show that $\varinjlim \overline{C_{E}}[n]^{\vee}$ vanishes, where the transition maps are induced by the maps $\overline{C_{K}}[m] \to \overline{C_{E}}[n]$ given by $N_{K/E}$ composed with the $m/n$-power map. Given $f \in \overline{C_{E}}[n]^{\vee}$, 
since it's continuous and $\frac{1}{n}\Z/\Z$ is finite, $\text{ker}(f)$ is an open subgroup of $\overline{C_{E}}[n]$. The norm groups $N_{K/E}(\overline{C_{K}}) \subseteq \overline{C_{E}}$ shrink to the identity, so there is some $(K,m) > (E, n)$ with $N_{K/E}(\overline{C_{K}}) \subseteq \text{ker}(f)$ (using the finite intersection property) and hence the image of $f$ in $\overline{C_{K}}[m]^{\vee}$ is zero.
\end{proof}

The surjectivity of the maps $P_{i+1} \to P_{i}$ immediately implies (cf. \cite[Corollary 3.3]{Dillery}): 



\begin{lem}\label{H1isom} The spectral sequence comparing \v{C}ech and derived cohomology from \cite[03AV]{Stacksproj}  gives an isomorphism $\check{H}^{p}(\overline{?}/?, P_{\dot{V}}) \to H^{p}(?, P_{\dot{V}})$ for $?=F, F_{v}$ and all $p\geq 0$.
\end{lem}


\begin{prop}\label{vanishingH1} We have $H^{1}(F, P_{\dot{V}}) = \varprojlim_{i} H^{1}(F, P_{E_{i}, \dot{S}_{i}, n_{i}}) = 0$.
\end{prop}

\begin{proof} The first equality is from the proof of Proposition \ref{mainvanishing1}. Localization $\check{H}^{1}(\overline{F}/F, P_{\dot{V}}) \to \check{H}^{1}(\overline{F_{v}}/F_{v}, P_{\dot{V}})$ gives (via Lemma \ref{H1isom}) a localization map $H^{1}(F, P_{\dot{V}}) \to H^{1}(F_{v}, P_{\dot{V}})$ for all $v \in V$. Setting $\text{ker}^{1}(F, P_{\dot{V}}) := \text{ker}[H^{1}(F, P_{\dot{V}}) \to \prod_{\dot{v} \in \dot{V}} H^{1}(F_{v}, P_{\dot{V}})]$, it suffices by Lemma \ref{vanishingcolimit2} to show $\text{ker}^{1}(F, P_{\dot{V}}) \xrightarrow{\sim} \varprojlim_{i} \text{ker}^{1}(F, P_{E_{i}, \dot{S}_{i}, n_{i}}) = 0$.

For $i$ fixed, \cite[Lem. 4.4]{Cesnavicius} tells us that we have a perfect pairing of finite abelian groups $$\text{ker}^{1}(F, P_{E_{i}, \dot{S}_{i}, n_{i}}) \times \text{ker}^{2}(F, \underline{M_{E_{i}, \dot{S}_{i}, n_{i}}}) \to \mathbb{Q}/\Z,$$ where $\underline{M_{E_{i}, \dot{S}_{i}, n_{i}}}$ is the \'{e}tale $F$-group scheme associated to the $\Gamma$-module $M_{E_{i}, \dot{S}_{i}, n_{i}}$ (and is Cartier dual to the finite flat $F$-group scheme $P_{E_{i}, \dot{S}_{i}, n_{i}}$). Thus, it's enough to show that $$\varprojlim_{i} (\text{ker}^{2}(\Gamma, M_{E_{i}, \dot{S}_{i}, n_{i}}))^{*} =  (\varinjlim_{i} \text{ker}^{2}(\Gamma, M_{E_{i}, \dot{S}_{i}, n_{i}}))^{*} = 0,$$ which we will do by showing that the direct limit $\varinjlim_{i} \text{ker}^{2}(\Gamma, M_{E_{i}, \dot{S}_{i}, n_{i}})$ vanishes, which follows from an easier version of the analogous argument in the proof of \cite[Prop. 3.4.6]{Tasho2} that we leave to the reader (using our Lemmas \ref{vanishingcolimit2}, \ref{vanishingcolimit3} in place of Lemma 3.4.4 loc. cit.).
\end{proof}

\subsection{The canonical class} We now show that there is a canonical element $\xi \in \check{H}^{2}(\overline{F}/F, P_{\dot{V}})$ lifting the element $(\xi_{i}) \in \varprojlim_{i} \check{H}^{2}(\overline{F}/F, P_{E_{i}, \dot{S}_{i}, n_{i}})$ constructed above. Set $P:= P_{\dot{V}}$, $P_{i} := P_{E_{i}, \dot{S}_{i}, n_{i}}$ and $M_{i} := M_{E_{i}, \dot{S}_{i}, n_{i}}$, and denote the projection $P \to P_{i}$ by $p_{i}$. Whenever we work with an embedding $\overline{F} \to \overline{F_{v}}$ for $v \in V$, we assume it is the one induced by $\dot{v} \in \dot{V}$ unless otherwise specified. We begin by proving some results about \v{C}ech cohomology groups associated to $P_{\dot{V}}$.

\begin{lem}\label{inverselimit1} The natural projective system maps $\check{H}^{k}(\overline{\A}_{v}/F_{v}, P) \to \varprojlim_{i} \check{H}^{k}(\overline{\A}_{v}/F_{v}, P_{i})$ are isomorphisms for $k=0,1,2$ (recall that $\overline{\A}_{v} := \overline{F} \otimes_{F} F_{v}$).
\end{lem}

\begin{proof} The case $k=0$ is trivial, so we only need to focus on $k=1,2$. By Lemma \ref{cechidentification}, it's enough to show that $\varprojlim_{i}^{(1)} \check{H}^{k}(\overline{\A}_{v}/F_{v}, P_{i}) = 0$ for $k=0,1$ and that $\varprojlim_{i}^{(1)} B^{1}(i) = 0$ (notation as in Lemma \ref{cechidentification}). The vanishing of $\varprojlim_{i}^{(1)} \check{H}^{0}(\overline{\A}_{v}/F_{v}, P_{i})$ follows from the fact that $\check{H}^{0}(\overline{\A}_{v}/F_{v}, P_{i}) = P_{i}(F_{v})$, and the system $\{P_{i}(F_{v})\}$ consists of finite groups. The vanishing of $\varprojlim_{i}^{(1)} B^{1}(i)$ comes from the fact that the system $\{ B^{1}(i)\}$ has surjective transition maps: On \v{C}ech 0-cochains the transition maps $P_{i+1}(\overline{\A}_{v}) \to P_{i}(\overline{\A}_{v})$ are all surjective by Lemma \ref{vadelicvanishing1}, and since the \v{C}ech differentials are compatible with $F$-homomorphisms (in our case, the transition maps $P_{i+1} \to P_{i}$), this surjectivity carries over to the group of $1$-coboundaries. 

The proof of Corollary \ref{shapiro1} shows that the groups $\check{H}^{1}(\overline{\A}_{v}/F_{v}, P_{i})$ are (compatibly) isomorphic to $H^{1}(F_{v}, P_{i})$, and, as argued in the proof in Corollary \ref{tasho3.3.8}, $\varprojlim_{i}^{(1)} H^{1}(F_{v}, P_{i}) = 0$.
\end{proof}

Combining Lemma \ref{inverselimit1} with the proof of Corollary \ref{shapiro1} gives an isomorphism $$\check{H}^{1}(\overline{\A}_{v}/F_{v}, P) \xrightarrow{\sim} \varprojlim \check{H}^{1}(\overline{\A}_{v}/F_{v}, P_{i}) \xrightarrow{\sim}  \varprojlim \check{H}^{1}(\overline{F_{v}}/F_{v}, P_{i}) \xrightarrow{\sim} \check{H}^{1}(\overline{F_{v}}/F_{v}, P),$$ and so Lemma \ref{H1isom} lets us identify $\check{H}^{1}(\overline{\A}_{v}/F_{v}, P)$ with $H^{1}(F_{v}, P)$ as well. The local canonical class $\xi_{v} \in \check{H}^{2}(\overline{F_{v}}/F_{v}, u_{v}) = H^{2}(F_{v}, u_{v})$ maps via $S^{2}_{v} \circ \text{loc}_{v}$ to $\check{H}^{2}(\overline{\A}_{v}/F_{v}, P)$ (notation as in \S 2.2).

We now construct a canonical $x \in \check{H}^{2}(\overline{\A}/\A, P)$ such that for each $\dot{v} \in \dot{V}$ , its image in $\check{H}^{2}(\overline{\A}_{v}/F_{v}, P)$ (via the ring homomorphism $\overline{F} \otimes_{F} \A \xrightarrow{\iota \otimes \pi_{v}} \overline{F_{v}} \otimes_{F} F_{v}$ , where $\iota \colon \overline{F} \to \overline{F_{v}}$ is our fixed inclusion, and $\pi_{v}$ is projection onto the $v$th-factor) equals $S_{v}^{2}(\text{loc}_{v}(\xi_{v}))$ via a \v{C}ech-theoretic analogue of the construction in \cite[\S 3.5]{Tasho2}. Fix $\dot{\xi}_{v} \in u_{v}(\overline{F_{v}}^{\bigotimes_{F_{v}} 3})$ a 2-cocycle representing $\xi_{v}$, and let $\Gamma_{\dot{v}} \subseteq \Gamma$ denote the decomposition group of $\dot{v} \in \dot{V}$; choose a section $\Gamma/\Gamma_{\dot{v}} \to \Gamma$---recall from \S 2.2 that this is equivalent to fixing a compatible system of diagonal embeddings $$E \cdot F_{v} \to \prod_{w \in V_{E}, w \mid v} E_{w}$$ ranging over all finite Galois extensions $E/F$ (which are the identity $E \cdot F_{v} \to E_{\dot{v}_{E}}$ on the $\dot{v}_{E}$-factor), and thus (as explained in \S 2.2) an explicit realization of the Shapiro map $h \colon G(\overline{F_{v}}^{\bigotimes_{F_{v}}3})  \to G(\overline{\A_{v}}^{\bigotimes_{F_{v}}3})$ at the level of 2-cochains for any multiplicative $F$-group scheme $G$, which is functorial in $G$ (with respect to $F$-homomorphisms) and compatible with the \v{C}ech differentials on both sides. 

The maps $S_{v,i}^{2} \colon P_{i}(\overline{F_{v}}^{\bigotimes_{F_{v}}3})  \to P_{i}(\overline{\A_{v}}^{\bigotimes_{F_{v}}3})$ splice to give a homomorphism $\dot{S}^{2}_{v} \colon P(\overline{F_{v}}^{\bigotimes_{F_{v}}3}) \to P(\overline{\A_{v}}^{\bigotimes_{F_{v}}3}),$ and we set $\dot{x}_{v} := \dot{S}^{2}_{v}(\text{loc}_{v}(\dot{\xi}_{v})) \in Z^{2}(\overline{\A}_{v}/F_{v}, P)$. Note that for $i$ fixed, $p_{i}(\dot{x}_{v} ) = 1 \in P_{i}(\overline{\A_{v}}^{\bigotimes_{F_{v}}3})$ for all $\dot{v} \in \dot{V}$ with $v \notin S_{i}$. Indeed, the functoriality of the Shapiro maps implies that $p_{i} \circ \dot{S}_{v}^{2} \circ \text{loc}_{v} = S_{v,i}^{2} \circ p_{i} \circ \text{loc}_{v}$ on $P_{i}(\overline{F_{v}}^{\bigotimes_{F_{v}}3})$, and $p_{i} \circ \text{loc}_{v} \colon u_{v} \to P_{F_{v}} \to (P_{i})_{F_{v}}$ is trivial for $v \notin S_{i}$, since it is induced by the direct limit over $j \in \mathbb{N}$ (with $\dot{v}_{E_{j}}\in \dot{S}_{j}$) of $\Gamma_{\dot{v}}$-module homomorphisms $$ \frac{1}{n_{i}}\Z/\Z[\Gamma_{E_{i}/F} \times (S_{i})_{E_{i}}]_{0,0} \to \frac{1}{n_{j}}\Z/\Z[\Gamma_{E_{j}/F} \times (S_{j})_{E_{j}}]_{0,0} \to X^{*}(u_{n_{j},E_{j}/F_{v}}),$$ where the kernel of the second map contains all elements whose $(\sigma, \dot{v}_{E_{j}})$-coefficients $c_{\sigma, \dot{v}_{E_{j}}}$ are zero for all $\sigma \in \Gamma_{E_{j}/F}$, and the image of the first map lands in the subgroup of elements whose coefficients $c_{\sigma, w}$ are zero for all $w \in (S_{j})_{E_{j}}$ not lying above an element of $(S_{i})_{E_{i}}$, which is the case for $\dot{v}_{E_{j}}$, since $\dot{v}_{E_{i}} \in (S_{i})_{E_{i}}$ would mean that $v \in S_{i}$; this gives our desired triviality.

The above paragraph implies that the element $(p_{i}(\dot{x}_{v}))_{v \in V} \in \prod_{v \in V} P_{i}(\overline{\A}_{v}^{\bigotimes_{F_{v}}3})$ is trivial in all but finitely-many $v$-coordinates, so we may view it as an element of $\bigoplus_{v \in S_{i}} Z^{2}(\overline{\A}_{v}/F_{v}, P_{i}).$ We may further view it as an element of the restricted product $\prod_{v}' P_{i}(\A_{L,v}^{\bigotimes_{F_{v}}3})$ with respect to the subgroups $P_{i}(O_{L,v}^{\bigotimes_{O_{F_{v}}}3})$, (cf. \S 2.2) for some finite extension $L/F$ (because $\overline{\A}_{v} = \varinjlim \A_{K,v}$ over all finite extensions $K/F$ and each $P_{i}/F$ is of finite type). We thus obtain by Proposition \ref{cechrestrictedproduct} an element of $Z^{2}(\overline{\A}/\A, P_{i})$, and as we vary $i$ they splice to give $\dot{x} \in \varprojlim_{i} Z^{2}(\overline{\A}/\A, P_{i}) = Z^{2}(\overline{\A}/\A, P)$.

The argument in \cite[\S 3.5]{Tasho2} (using \v{C}ech differentials) shows that $[\dot{x}] \in \check{H}^{2}(\overline{\A}/\A, P)$ is independent of the choice of local representatives $\dot{\xi}_{v}$. We now explain why it is also independent of the choice of section $\Gamma/\Gamma_{\dot{v}} \to\Gamma$. Denote by $s, s'$ two sections with Shapiro maps $\dot{S}_{v}$, $\dot{S}'_{v}$ (cf. \S 2.2).

Lemma \ref{differentsections} implies that for any $v \in V$ we may find $c_{v} \in P(\overline{\A}_{v} \otimes_{F_{v}} \overline{\A}_{v})$ such that $dc_{v} = \dot{S}_{v}(\text{loc}_{v}(\dot{\xi}_{v})) \cdot \dot{S}'_{v}(\text{loc}_{v}(\dot{\xi}_{v}))^{-1}$; we claim that we may choose $c_{v}$ such that for any $i$ with $v \notin S_{i}$ we have $p_{i}(c_{v}) = 1$ in $P_{i}(\overline{\A}_{v} \otimes_{F_{v}} \overline{\A}_{v})$. Indeed, by the proof of Lemma \ref{differentsections} we can set $$c_{v} = (r_{1} \cdot \bar{r}_{3} \otimes r_{2})(\text{loc}_{v}(\dot{\xi}_{v})) \cdot (\bar{r}_{2} \otimes r_{1} \cdot \bar{r}_{3})(\text{loc}_{v}(\dot{\xi}_{v}))^{-1},$$ where $r \colon \overline{F_{v}} \to \overline{\A}_{v}$ is the direct limit of the maps $E'_{(\dot{v}_{E})'} \to \prod_{w \mid v} E'_{w'}$ defined on the $w$-coordinate by the isomorphism $r_{w} \colon E'_{(\dot{v}_{E})'} \xrightarrow{\sim} E'_{w'}$ determined by the section $s$, similarly with $\bar{r}$, where as in the proof of Lemma \ref{differentsections}, the subscript $i$ in $r_{i}$ denotes that its source is the $i$th tensor factor of $(E'_{v'})^{\bigotimes_{F_{v}}3}$. Since the maps $r, \bar{r}$ are $F_{v}$-homomorphisms they commute with $p_{i}$ and hence $$p_{i}(c_{v}) = (r_{1} \cdot \bar{r}_{3} \otimes r_{2})(p_{i}\text{loc}_{v}(\dot{\xi}_{v})) \cdot (\bar{r}_{2} \otimes r_{1} \cdot \bar{r}_{3})(p_{i}\text{loc}_{v}(\dot{\xi}_{v}))^{-1} = 1.$$

As a result, the element $\tilde{c}:= (c_{v} )_{v} \in \prod_{v \in V} P(\overline{\A}_{v} \otimes_{F_{v}} \overline{\A}_{v})$ has projection $p_{i}(\tilde{c})$ with all but finitely-many trivial coordinates, and hence has well-defined image in $P_{i}(\overline{\A} \otimes_{\A} \overline{\A})$ (using Corollary \ref{mainappendixBcor}). Setting $c:= \varprojlim_{i} p_{i}(\tilde{c})$ gives an element of $P(\overline{\A} \otimes_{\A} \overline{\A}) $ which satisfies $\dot{S}_{v}^{2}(\text{loc}_{v}(\dot{\xi}_{v})) \cdot \dot{S}_{v}^{' 2}(\text{loc}_{v}(\dot{\xi}_{v}))^{-1} = dc$, concluding the argument for why the class $[\dot{x}] \in \check{H}^{2}(\overline{\A}/\A, P)$ is canonical. 

The final key step in constructing a canonical class in $\check{H}^{2}(\overline{F}/F, P)$ is showing that there is a unique element of $\check{H}^{2}(\overline{F}/F, P)$ whose image in $\check{H}^{2}(\overline{\A}/\A, P)$ is the class $x:= [\dot{x}]$, and whose image in $\varprojlim_{i} \check{H}^{2}(\overline{F}/F, P_{i})$ is $(\xi_{i})$. The argument will use the following projective system of complexes of tori, following \cite[\S 3.5]{Tasho2}:

\begin{lem}\label{toritower}(\cite[Lem. 3.5.1]{Tasho2} The projective system $\{P_{i}\}$ fits into a short exact sequence of projective systems of tori $\{T_{i}\}$, $\{U_{i}\}$
$$1 \to \{P_{i}\} \to \{T_{i}\} \to \{U_{i}\} \to 1$$
such that each $T_{i+1} \to T_{i}$ is surjective with kernel a torus (this also holds for each $U_{i+1} \to U_{i}$).
\end{lem}
For any $i$, consider the double complex of abelian groups $K^{p,q} = $
\[ 
\begin{tikzcd}
T_{i}(\overline{\A}) \arrow{r} \arrow{d} & T_{i}(\overline{\A} \otimes_{\A} \overline{\A}) \arrow{r} \arrow{d} & T_{i}(\overline{\A} \otimes_{\A} \overline{\A} \otimes_{\A} \overline{\A}) \arrow{r} \arrow{d} & \dots \\
U_{i}(\overline{\A}) \arrow{r} & U_{i}(\overline{\A} \otimes_{\A} \overline{\A}) \arrow{r}& U_{i}(\overline{\A} \otimes_{\A} \overline{\A} \otimes_{\A} \overline{\A}) \arrow{r} & \dots.
\end{tikzcd}
\]

Note that the complex with $j$th term $C^{j} := \mathcal{H}^{0}(K^{j, \bullet})$ ($j \geq 0$) is exactly the \v{C}ech complex of $P_{i}$ with respect to the fpqc cover $\overline{\A}/\A$, and so the low-degree exact sequence for the spectral sequence associated to a double complex gives an injective map $\check{H}^{1}(\overline{\A}/\A, P_{i}) \hookrightarrow H^{1}(\overline{\A}/\A, T_{i} \to U_{i}).$

Since the kernels of $T_{i+1} \to T_{i}$ and $U_{i+1} \to U_{i}$ are tori, combining Corollary \ref{mainappendixBcor} with Lemma \ref{vadelicvanishing1} tells us that the maps $T_{i+1}(\overline{\A}^{\bigotimes_{\A}n}) \to T_{i}(\overline{\A}^{\bigotimes_{\A}n})$ and $U_{i+1}(\overline{\A}^{\bigotimes_{\A}n}) \to U_{i}(\overline{\A}^{\bigotimes_{\A}n})$ are surjective for all $n$ (this also holds for $\overline{\A}$ replaced by $\A^{\text{sep}}$, by smoothness). It follows that the induced map $$C^{j}(\overline{\A}/\A, T_{i+1} \to U_{i+1}) \to C^{j}(\overline{\A}/\A, T_{i} \to U_{i})$$ (where $C^{j}(\overline{\A}/\A, T \to U)$ is the group of $j$-cochains for the corresponding total complex) is surjective for any $j$, and so the system $\{C^{j}(\overline{\A}/\A, T_{i} \to U_{i})\}_{i \geq 0}$ satisfies the Mittag-Lefler condition. Replacing $\overline{\A}$ by $\A^{\text{sep}}$ in order to use group cohomology (Lemma \ref{adelicgroupcohomology}), it follows from \cite[Thm. 3.5.8]{NSW} that we obtain the exact sequence 
\begin{equation}\label{lim1SES}
1 \to \varprojlim^{(1)} H^{1}(\overline{\A}/\A, T_{i} \to U_{i}) \to H^{2}(\overline{\A}/\A, T \to U) \to \varprojlim H^{2}(\overline{\A}/\A, T_{i} \to U_{i}) \to 1,
\end{equation}
where the middle term denotes the cohomology of the complex with $j$th term $$C^{j}(\overline{\A}/\A, T \to U) := \varprojlim C^{j}(\overline{\A}/\A, T_{i} \to U_{i}) = C^{j}(\overline{\A}/\A, T) \oplus C^{j-1}(\overline{\A}/\A, U),$$ where $T = \varprojlim_{i} T_{i}$ and $U := \varprojlim_{i} U_{i}$ are pro-tori over $F$ (by construction, the kernel of $T \to U$ is $P$). The low-degree exact sequence for double complexes again gives a map $\check{H}^{2}(\overline{\A}/\A, P) \to H^{2}(\overline{\A}/\A, T \to U)$, which need not be injective, and we have the natural map $\check{H}^{2}(\overline{F}/F, P) \to \check{H}^{2}(\overline{\A}/\A, P)$. We then have the following analogue of \cite[Prop. 3.5.2]{Tasho2}:

\begin{prop}\label{canonicalclass} There is a unique $\xi \in \check{H}^{2}(\overline{F}/F, P)$ whose image in $\varprojlim \check{H}^{2}(\overline{F}/F, P_{i})$ equals the system $(\xi_{i})$, and whose image in $H^{2}(\overline{\A}/\A, T \to U)$ equals the image of $x \in \check{H}^{2}(\overline{\A}/\A, P)$ there.
\end{prop}

\begin{proof} If $\tilde{\xi} \in \check{H}^{2}(\overline{F}/F, P)$ is any preimage of $(\xi_{i}) \in \varprojlim \check{H}^{2}(\overline{F}/F, P_{i})$ and $\tilde{\xi}_{\A}$ denotes its image in $\check{H}^{2}(\overline{\A}/\A, P)$, the images of $x$ and $\tilde{\xi}_{\A}$ in $\varprojlim H^{2}(\overline{\A}/\A, T_{i} \to U_{i})$ via the composition 
$$ \check{H}^{2}(\overline{\A}/\A, P) \to H^{2}(\overline{\A}/\A, T \to U) \to \varprojlim H^{2}(\overline{\A}/\A, T_{i} \to U_{i})$$ coincide by the argument in \cite[Prop. 3.5.2]{Tasho2}, replacing the use of \cite[Thm. C.1.B]{KS1} loc. cit. with Proposition \ref{restrictedproduct} and the use of Corollary 3.3.8 loc. cit. with Corollary \ref{tasho3.3.8}. 

The same argument as in \cite[Lem. 3.5.3]{Tasho} (replacing the results from \cite[\S C]{KS1} with their analogues from \S A.3, A.4) shows that the map 
$\varprojlim^{(1)} \check{H}^{1}(\overline{F}/F, P_{i}) \to \varprojlim^{(1)} H^{1}(\overline{\A}/\A, T_{i} \to U_{i})$ is an isomorphism which, along with \eqref{lim1SES}, imply that we may modify $\tilde{\xi}$ by an element of $\varprojlim^{(1)} \check{H}^{1}(\overline{F}/F, P_{i})$ so that the images of $\tilde{\xi}_{\A}$ and $x$ in $H^{2}(\overline{\A}/\A, T \to U)$ are equal, proving existence. Uniqueness follows from the injectivity of the composition $$ \varprojlim^{(1)} \check{H}^{1}(\overline{F}/F, P_{i}) \to \varprojlim^{(1)} H^{1}(\overline{\A}/\A, T_{i} \to U_{i}) \to H^{2}(\overline{\A}/\A, T \to U).$$ 
\end{proof}

\begin{Def}\label{canonicalclassdef} The \textit{canonical class} $\xi \in \check{H}^{2}(\overline{F}/F, P)$ is the element whose existence and uniqueness is asserted in Proposition \ref{canonicalclass}. As explained in \cite[p. 41]{Tasho2}, $\xi$ is independent of the systems $\{T_{i}\}$, $\{U_{i}\}$ chosen above.
\end{Def}


\section{Cohomology of the gerbe $\gerbeE_{\dot{V}}$}

\subsection{Basic definitions}

As in previous sections, we write $H^{i}$ for $H^{i}_{\text{fppf}}$. Let $\xi \in \check{H}^{2}(\overline{F}/F, P_{\dot{V}})$ be the canonical class of Definition \ref{canonicalclassdef}. By \cite[\S 2]{Dillery}, $\xi$ corresponds to an isomorphism class of (fpqc) $P_{\dot{V}}$-gerbes split over $\overline{F}$. Let $\gerbeE_{\dot{V}} \to (\text{Sch}/F)_{\text{fpqc}}$ be such a gerbe; we can equip $\gerbeE_{\dot{V}}$ with the structure of a site, $(\gerbeE_{\dot{V}})_{\text{fpqc}}$, by giving it the fpqc topology inherited from $(\text{Sch}/F)_{\text{fpqc}}$. For $G$ an affine algebraic group over $F$ we write $H^{1}(\gerbeE_{\dot{V}}, G_{\gerbeE_{\dot{V}}})$ to denote all isomorphism classes of $G_{\gerbeE_{\dot{V}}}$-torsors on $(\gerbeE_{\dot{V}})_{\text{fpqc}}$. Any such torsor $\mathscr{T}$ carries an action $\iota \colon (P_{\dot{V}})_{\gerbeE_{\dot{V}}} \times \mathscr{T} \to \mathscr{T}$.

For $Z \subset G$ a finite central $F$-subgroup, define $H^{1}(\gerbeE_{\dot{V}}, Z \to G)$ as all isomorphism classes $[\mathscr{T}]$ such that $\iota$ is induced by a homomorphism of $F$-groups $P_{\dot{V}} \xrightarrow{\phi} Z \hookrightarrow G$---we refer to such a torsor $\mathscr{T}$ as \textit{$Z$-twisted}. Note that such $\phi$ always factors through the projection $P_{\dot{V}} \to P_{E_{i}, \dot{S}_{i}, n_{i}}$ for some $i$. For any other choice of $P_{\dot{V}}$-gerbe $\gerbeE'_{\dot{V}}$ representing $\xi$, we have an isomorphism of $P_{\dot{V}}$-gerbes $h \colon {\gerbeE_{\dot{V}}} \to {\gerbeE'_{\dot{V}}}$, inducing an isomorphism $H^{1}(\gerbeE_{\dot{V}}, G_{\gerbeE_{\dot{V}}}) \to H^{1}(\gerbeE_{\dot{V}}, G_{\gerbeE'_{\dot{V}}})$ which, since $\check{H}^{1}(\overline{F}/F, P_{\dot{V}})$ vanishes by Proposition \ref{vanishingH1}, is independent of the choice of $h$, by \cite[Lem. 2.56]{Dillery}. 

Let $\mathcal{A}$ denote the category whose objects are pairs $(Z, G)$ as above and whose morphisms from $(Z_{1}, G_{1})$ to $(Z_{2}, G_{2})$ are $F$-morphisms from $G_{1}$ to $G_{2}$ that map $Z_{1}$ to $Z_{2}$. As shown in \cite[\S 2.6]{Dillery}, we have the ``inflation-restriction" exact sequence (of pointed sets or abelian groups) $$1 \to H^{1}(F, G) \to H^{1}(\gerbeE_{\dot{V}}, Z \to G) \to \Hom_{F}(P_{\dot{V}}, Z) \to H^{2}(F, G),$$ where the $H^{2}$-term is to be ignored if $G$ is non-abelian. The above map from $\Hom_{F}(P_{\dot{V}}, Z)$ to $H^{2}(F, G)$ can be described as the composition of the map $\Theta^{P}_{\dot{V}} \colon \Hom_{F}(P_{\dot{V}}, Z) \to H^{2}(F, Z)$ defined in \S 3.2 with the map $H^{2}(F, Z) \to H^{2}(F, G)$. We also have a commutative diagram
\[
\begin{tikzcd}
H^{1}(F, G) \arrow{r} \arrow[equals]{d} & H^{1}(\gerbeE_{\dot{V}}, Z \to G) \arrow["\dag"]{d} \arrow["\text{Res}"]{r} & \Hom_{F}(P_{\dot{V}}, Z) \arrow{r} \arrow["\Theta^{P}_{\dot{V}}"]{d} & H^{2}(F, G) \arrow[equals]{d} \\ 
H^{1}(F, G) \arrow{r} & H^{1}(F, G/Z) \arrow{r} & H^{2}(F, Z) \arrow{r} & H^{2}(F,G),
\end{tikzcd}
\]
where the map $\dag$ sends $[\mathscr{T}]$ to (the class of) the descent of the $(G/Z)_{\gerbeE_{\dot{V}}}$-torsor $\mathscr{T} \times^{G_{\gerbeE_{\dot{V}}}} (G/Z)_{\gerbeE_{\dot{V}}}$ to a $G/Z$-torsor $T$ over $F$. It has such a descent (unique up to isomorphism) because, by construction, its $(P_{\dot{V}})_{\gerbeE_{\dot{V}}}$-action is trivial. We have the following result(s) from \cite{Tasho2}:

\begin{lem}\label{tasho3.6.1}\label{3.6.2}( \cite[Lems. 3.6.1, 3.6.2]{Tasho2}) \begin{enumerate}
\item{If $G$ is either abelian or connected and reductive, then the map $\dag$ defined above is surjective.}
\item{ If $G$ is connected and reductive, then for each $x \in H^{1}(\gerbeE_{\dot{V}}, Z \to G)$, there exists a maximal torus $T \subset G$ such that $x$ is in the image of $H^{1}(\gerbeE_{\dot{V}}, Z \to T)$.}
\end{enumerate}
\end{lem}

\begin{proof}  The identical proofs loc. cit. work here, using the five-lemma and our discussion following Lemma \ref{Ptheta2} for abelian $G$ and replacing the use of Lemma A.1 loc. cit. with \cite[Corollary 1.10]{Thang3} for connected reductive $G$.
\end{proof}




The next goal is to construct a localization map $\text{loc}_{v} \colon H^{1}(\gerbeE_{\dot{V}}, Z \to G) \to H^{1}(\gerbeE_{v}, Z \to G)$ for any $\dot{v} \in \dot{V}$, where $\gerbeE_{v}$ denotes the local gerbe defined in \cite[\S 3]{Dillery}, such that the diagram
\begin{equation}\label{gerbeloc}
\begin{tikzcd}
H^{1}(F, G) \arrow{r} \arrow{d} & H^{1}(\gerbeE_{\dot{V}}, Z \to G) \arrow{r} \arrow{d} & \Hom_{F}(P_{\dot{V}}, Z) \arrow{d} \\
H^{1}(F_{v}, G) \arrow{r} & H^{1}(\gerbeE_{v}, Z \to G) \arrow{r} & \Hom_{F_{v}}(u_{v}, Z)
\end{tikzcd}
\end{equation}
commutes, where (using \v{C}ech cohomology) the left vertical map is induced by the inclusion $\overline{F} \to \overline{F_{v}}$ and the right vertical map is induced by the $F$-homomorphism $\text{loc}^{P}_{v} \colon u_{v} \to (P_{\dot{V}})_{F_{v}}$ (cf. \S 3.2). 

We have the stack $(\gerbeE_{\dot{V}})_{F_{v}} := \gerbeE \times_{\text{Sch}/F} (\text{Sch}/F_{v})$, which is an fpqc $(P_{\dot{V}})_{F_{v}}$-gerbe split over $\overline{F_{v}}$; 
Fixing an isomorphism of $P_{\dot{V}}$-gerbes $(\gerbeE_{\dot{V}})_{F_{v}} \xrightarrow{\sim} \gerbeE_{x_{v}}$, where $x_{v}$ denotes a \v{C}ech 2-cocycle representing the image of $\xi \in \check{H}^{2}(\overline{F}/F, P_{\dot{V}})$ in $\check{H}^{2}(\overline{F_{v}}/F_{v}, P_{\dot{V}})$, and an isomorphism of $u_{v}$-gerbes $\gerbeE_{v} \xrightarrow{\sim} \gerbeE_{\dot{\xi}_{v}}$, where $\dot{\xi}_{v}$ is a \v{C}ech 2-cocycle representing the local canonical class $\xi_{v}$, the fact that $\text{loc}_{v}^{P}(\xi_{v}) = [x_{v}]$ (by Corollary \ref{tasho3.3.8}) implies, by the functoriality of gerbes given by Construction \ref{changeofgerbe}, that we have a (non-canonical) morphism of fibered categories over $F_{v}$ from $\gerbeE_{\dot{\xi}_{v}}$ to $\gerbeE_{x_{v}}$ which is the morphism $\text{loc}_{v}^{P}$ on bands, and thus we obtain a functor $\gerbeE_{v} \xrightarrow{\text{loc}_{v}^{\mathcal{E}}} \gerbeE_{\dot{V}}$ via the composition $$\gerbeE_{v} \xrightarrow{\sim} \gerbeE_{\dot{\xi}_{v}} \to \gerbeE_{x_{v}} \xrightarrow{\sim} (\gerbeE_{\dot{V}})_{F_{v}} \to \gerbeE_{\dot{V}}.$$ 

Although $\text{loc}_{v}^{\mathcal{E}}$ is highly non-canonical, the morphism $\gerbeE_{\dot{\xi}_{v}} \to \gerbeE_{x_{v}}$ is unique up to post-composing by an automorphism of $\gerbeE_{x_{v}}$ determined by a \v{C}ech 1-cocycle of $P_{\dot{V}}$ valued in $\overline{F_{v}}$ (see \cite[\S 2.3]{Dillery}), and since, by Proposition \ref{mainvanishing1}, the group $\check{H}^{1}(\overline{F_{v}}/F_{v}, P_{\dot{V}})$ is trivial, such a 1-cocycle is in fact a 1-coboundary. The same is true for the isomorphisms $\gerbeE_{v} \to \gerbeE_{\dot{\xi}_{v}}$ and $(\gerbeE_{\dot{V}})_{F_{v}} \to \gerbeE_{x_{v}}$ (which we call ``normalizing isomorphisms"), using Proposition \ref{mainvanishing1} and \cite[Corollary 3.5]{Dillery}. We are now ready to define our main localization map. 

We can use $\text{loc}_{v}^{\mathcal{E}}$ to pull back any $Z$-twisted $G_{\gerbeE_{\dot{V}}}$-torsor to a $Z_{F_{v}}$-twisted $G_{\gerbeE_{v}}$-torsor, defining $$\text{loc}_{v} \colon H^{1}(\gerbeE_{\dot{V}}, Z \to G) \to H^{1}(\gerbeE_{v}, Z \to G);$$ combining the previous paragraph with \cite[Prop. 2.55]{Dillery} show that $\text{loc}_{v}$ is canonical. It is straightforward to check that this localization map makes the diagram \eqref{gerbeloc} commute.
This map is canonical up to finer equivalence classes of $G_{\gerbeE_{v}}$-torsors: We can replace isomorphism classes with sets of torsors whose elements are related via isomorphisms $\mathscr{T} \xrightarrow{\sim} \eta^{*}\mathscr{T}$ of $G_{\gerbeE_{v}}$-torsors induced by translation by $z^{-1} \in Z(\overline{F_{v}})$, where $\eta \colon \gerbeE_{v} \to \gerbeE_{v}$ is the automorphism of gerbes induced by the 1-coboundary $d(z)$ (cf. \cite[Lemma 2.56]{Dillery}), and different choices of $\text{loc}_{v}$ preserve these classes.

\subsection{Tate-Nakayama duality for tori}
As in \cite{Tasho2}, we define $\mathcal{T} \subset \mathcal{A}$ to be the full subcategory consisting of objects $[Z \to T]$ for which $T$ is a torus, and for $v \in V$ we define $\mathcal{T}_{v}$ analogously. Recall from \cite[\S 4.1]{Dillery} that associated to such a pair $[Z \to T] \in \mathcal{T}_{v}$ we have the group $$\overline{Y}_{+v,\text{tor}}[Z \to T] := (X_{*}(T/Z)/[I_{v}X_{*}(T)])_{\text{tor}} = (X_{*}(T/Z)/[I_{v}X_{*}(T)])^{N_{E/F_{v}}},$$ where $I_{v} \subset \Z[\Gamma_{v}]$ denotes the augmentation ideal, $E/F_{v}$ denotes a finite Galois extension splitting $T$, and the superscript $N_{E/F_{v}}$ denotes the kernel of the norm map. Moreover, by \cite[Thm. 4.10]{Dillery}, we have a canonical functorial isomorphism $$\iota_{v} \colon \overline{Y}_{+v,\text{tor}}[Z \to T] \xrightarrow{\sim} H^{1}(\gerbeE_{v}, Z \to T)$$ which commutes with the maps of both groups to $\Hom_{F_{v}}(u_{v}, Z)$. 

Following \cite{Tasho2}, the first step is to construct the global analogue of the groups $\overline{Y}_{+v,\text{tor}}[Z \to T]$. For fixed $[Z \to T] \in \mathcal{T}$ we set $Y:= X_{*}(T)$, $\overline{Y}:= X_{*}(T/Z)$, and $A^{\vee} :=  \Hom_{\Z}(X^{*}(Z), \mathbb{Q}/\Z)$. We have a short exact sequence $$0 \to Y \to \overline{Y} \to A^{\vee} \to 0$$ and for any $i$, tensoring with the $\Gamma_{E_{i}/F}$-module $\Z[(S_{i})_{E_{i}}]_{0}$ gives the exact sequence
\begin{equation}\label{Y0SES}
0 \to Y[(S_{i})_{E_{i}}]_{0} \to \overline{Y}[(S_{i})_{E_{i}}]_{0} \to A^{\vee}[(S_{i})_{E_{i}}]_{0} \to 0;
\end{equation} denote by $\overline{Y}[(S_{i})_{E_{i}}, \dot{S}_{i}]_{0} \subseteq  \overline{Y}[(S_{i})_{E_{i}}]_{0}$ the preimage of the subgroup $A^{\vee}[\dot{S}_{i}]$ under the above surjection; note that, by construction $\overline{Y}[(S_{i})_{E_{i}}, \dot{S}_{i}]_{0}$ contains the image of $Y[(S_{i})_{E_{i}}]_{0}$.

Choosing any section $s \colon (S_{i})_{E_{i}} \to (S_{i+1})_{E_{i+1}}$ such that $s(\dot{S}_{i}) \subset \dot{S}_{i+1}$, we may define a map 
$$ \overline{Y}[(S_{i})_{E_{i}}, \dot{S}_{i}]_{0} \xrightarrow{s_{!}} \overline{Y}[(S_{i+1})_{E_{i+1}}, \dot{S}_{i+1}]_{0}, \hspace{1mm} s_{!}(\sum_{w \in (S_{i})_{E_{i}}} c_{w}[w]) = \sum_{w' \in (S_{i+1})_{E_{i+1}}, s((w')_{E_{i}}) = w'} c_{(w')_{E_{i}}} [w'].$$ 

\begin{lem}(\cite[Lem. 3.7.1]{Tasho2} The map $f \mapsto s_{!}f$ induces a well-defined homomorphism $$! \colon \frac{\overline{Y}[(S_{i})_{E_{i}}, \dot{S}_{i}]_{0}}{I_{E_{i}/F}Y[(S_{i})_{E_{i}}]_{0}} \to \frac{\overline{Y}[(S_{i+1})_{E_{i+1}}, \dot{S}_{i+1}]_{0}}{I_{E_{i+1}/F}Y[(S_{i+1})_{E_{i+1}}]_{0}}$$ which is independent of the choice of $s$.
\end{lem}

\begin{Def} We define $$\overline{Y}[V_{\overline{F}}, \dot{V}]_{0,+,\text{tor}} := \varinjlim_{i} \frac{\overline{Y}[(S_{i})_{E_{i}}, \dot{S}_{i}]_{0}}{I_{E_{i}/F}Y[(S_{i})_{E_{i}}]_{0}}[\text{tor}], \hspace{1mm} Y[V_{\overline{F}}]_{0, \Gamma, \text{tor}} := \varinjlim_{i} \frac{Y[(S_{i})_{E_{i}}]_{0}}{I_{E_{i}/F}Y[(S_{i})_{E_{i}}]_{0}}[\text{tor}]$$ with transition maps given by $!$. 
\end{Def} 

The above two groups fit into the short exact sequence 
$$0 \to Y[V_{\overline{F}}]_{0, \Gamma, \text{tor}} \to \overline{Y}[V_{\overline{F}}, \dot{V}]_{0,+,\text{tor}} \to A^{\vee}[\dot{V}]_{0} \to 0.$$

For any $v \in V$ we can define a localization morphism $$l_{v} \colon \overline{Y}[V_{\overline{F}}, \dot{V}]_{0,+,\text{tor}} \to \overline{Y}_{+v, \text{tor}}$$ as follows. For a fixed index $i$, choose a representative $\dot{\tau} \in \Gamma_{E_{i}/F}$ for each right coset $\tau \in \Gamma_{E_{i}/F}^{\dot{v}} \backslash \Gamma_{E_{i}/F}$ such that $\dot{\tau} = 1$ for the trivial coset; for $f= \sum_{w \in (S_{i})_{E_{i}}} c_{w}[w] \in \overline{Y}[(S_{i})_{E_{i}}, \dot{S}_{i}]_{0}$, set $$l^{i}_{v}(f) = \sum_{\tau \in \Gamma_{E_{i}/F}^{\dot{v}} \backslash \Gamma_{E_{i}/F}} \prescript{\dot{\tau}}{}c_{\prescript{\tau^{-1}}{}(\dot{v})} \in \overline{Y}.$$

It is shown in \cite[Lem. 3.7.2]{Tasho2} that  $l_{v}^{i}$  descends to a group homomorphism 
\begin{equation}\label{Yloc} l_{v}^{i} \colon \frac{\overline{Y}[(S_{i})_{E_{i}}, \dot{S}_{i}]_{0}}{I_{E_{i}/F}Y[(S_{i})_{E_{i}}]_{0}} \to \frac{\overline{Y}}{I_{v}Y}
\end{equation}
that is independent of the choices of representatives $\dot{\tau}$ and is compatible with the transition maps $!$ defined above. We may thus define the localization map $l_{v}$ as the direct limit of the maps $l_{v}^{i}$.




 We can now give the statement of the global Tate-Nakayama isomorphism:

\begin{thm}\label{toritatenak} There exists a unique isomorphism 
$$\iota_{\dot{V}} \colon \overline{Y}[V_{\overline{F}}, \dot{V}]_{0,+,\text{tor}} \to H^{1}(\gerbeE_{\dot{V}}, Z \to T)$$ of functors $\mathcal{T} \to \text{AbGrp}$ that fits into the commutative diagram 
\[
\begin{tikzcd}
Y[V_{\overline{F}}]_{0, \Gamma, \text{tor}} \arrow["\text{TN}"]{d} \arrow{r} &  \overline{Y}[V_{\overline{F}}, \dot{V}]_{0,+,\text{tor}} \arrow["\iota_{\dot{V}}"]{d} \arrow{r} & A^{\vee}[\dot{V}]_{0} \arrow{d} \\
H^{1}(F, T) \arrow{r} & H^{1}(\gerbeE_{\dot{V}}, Z \to T) \arrow{r} & \Hom_{F}(P_{\dot{V}}, Z),
\end{tikzcd}
\]
where $\text{TN}$ denotes the colimit over $i$ of the finite global Tate-Nakayama isomorphisms $$H^{-1}(\Gamma_{E_{i}/F}, Y[(S_{i})_{E_{i}}]_{0}) \to H^{1}(\Gamma_{E_{i}/F}, T(O_{E_{i}, S_{i}}))$$ as in Lemma \ref{bigdiag1} (these splice to give a well-defined map, by Lemma 3.1.2 and Corollary 3.1.8 from \cite{Tasho2}), and the right vertical arrow is the one from Lemma \ref{Aveedot}. 

Moreover, for each $v \in \dot{V}$, the following diagram commutes 
\[
\begin{tikzcd}
\overline{Y}[V_{\overline{F}}, \dot{V}]_{0,+,\text{tor}} \arrow["\iota_{\dot{V}}"]{d} \arrow["l_{v}"]{r} & \overline{Y}_{+v, \text{tor}} \arrow["\iota_{v}"]{d} \\
H^{1}(\gerbeE_{\dot{V}}, Z \to T) \arrow["\text{loc}_{v}"]{r} & H^{1}(\gerbeE_{v}, Z \to T).
\end{tikzcd}
\]
\end{thm}

As with its analogue \cite[Thm. 3.7.3]{Tasho2}, this theorem takes some work to prove. 
Although $\overline{Y}[(S_{i})_{E_{i}}, \dot{S}_{i}]_{0}$ is not $\Gamma_{E_{i}/F}$-stable, it still makes sense to define the group $\overline{Y}[(S_{i})_{E_{i}}, \dot{S}_{i}]_{0}^{N_{E_{i}/F}}$ as the intersection $\overline{Y}[(S_{i})_{E_{i}}, \dot{S}_{i}]_{0} \cap \overline{Y}[(S_{i})_{E_{i}}]_{0}^{N_{E_{i}/F}}.$ We have the linear-algebraic result:

\begin{lem}(\cite[Lems. 3.7.6, 3.7.7]{Tasho2} \label{tasho3.7.6} \begin{enumerate}
\item{
$\frac{\overline{Y}[(S_{i})_{E_{i}}, \dot{S}_{i}]_{0}^{N_{E_{i}/F}}}{I_{E_{i}/F}Y[(S_{i})_{E_{i}}]_{0}} = \frac{\overline{Y}[(S_{i})_{E_{i}}, \dot{S}_{i}]_{0}}{I_{E_{i}/F}Y[(S_{i})_{E_{i}}]_{0}}[\text{tor}]$.}
\item{
Every element of $\overline{Y}[(S_{i})_{E_{i}}, \dot{S}_{i}]_{0}/I_{E_{i}/F}Y[(S_{i})_{E_{i}}]_{0}$ has a representative supported on $\dot{S}_{i}$.}
\end{enumerate}
\end{lem}

Following the outline of \cite[\S 3.7]{Tasho2}, the first step is proving an analogous Tate-Nakayama isomorphism for the groups $P_{E_{i}, \dot{S}_{i}} := \varprojlim_{n \in \mathbb{N}} P_{E_{i}, \dot{S}_{i}, n}$, which is useful because $P_{\dot{V}} = \varprojlim_{i} P_{E_{i}, \dot{S}_{i}}$ (cf. \cite[\S 3.3]{Tasho2}). Fix a triple $(E, S, \dot{S}_{E})$ satisfying Conditions \ref{placeconditions} and denote by $\mathcal{T}_{E}$ the full subcategory of objects $[Z \to T]$ of $\mathcal{T}$ such that $T$ splits over $E$. 

Note that $\check{H}^{2}(O_{S}^{\text{perf}}/O_{S}, P_{E, \dot{S}_{E}}) = H^{2}(O_{F,S}, P_{E, \dot{S}_{E}}) = \varprojlim_{n} H^{2}(O_{F,S}, P_{E, \dot{S}_{E},n})$; the first equality follows from \cite[03AV]{Stacksproj} and the second one from the vanishing of $\varprojlim^{(1)} H^{1}(O_{F,S}, P_{E, \dot{S}_{E},n})$ due to:

\begin{lem} The groups $H^{1}(O_{F,S}, P_{E, \dot{S}_{E},n})$ are finite for all finite $n$.

\end{lem}

\begin{proof} Set $P = P_{E, \dot{S}_{E},n}$. By \cite[Prop. 4.12]{Cesnavicius} (and its proof), the natural map $$H^{1}(O_{F,S}, P) \to H^{1}(\A_{F,S}, P) = \prod_{v \in S} H^{1}(F_{v}, P) \times \prod_{v \notin S} H^{1}(O_{v}, P) $$ has closed, discrete image and finite kernel, so it suffices to show that the image is finite. We claim that the right-hand side is compact---for this claim, it's enough by Tychonoff's theorem to prove that each $H^{1}(F_{v}, P)$ and $H^{1}(O_{v}, P)$ is compact. We showed this result for the former groups in Corollary \ref{tasho3.3.8}; for the latter, note that \cite{Cesnavicius} (3.1.1) says that each subset $H^{1}(O_{v}, P) \subseteq H^{1}(F_{v}, P)$ may be canonically topologized so that this inclusion is open, and since $H^{1}(F_{v}, P)$ is profinite (and hence totally disconnected), it is also closed, and therefore compact. Now the result follows, since closed, discrete subspaces of compact spaces are finite.
\end{proof}

We thus have a canonical class $\xi_{E, \dot{S}_{E}} \in \check{H}^{2}(O_{S}^{\text{perf}}/O_{S}, P_{E, \dot{S}_{E}}) =  \varprojlim_{n} \check{H}^{2}(O_{S}^{\text{perf}}/O_{F,S}, P_{E, \dot{S}_{E},n})$ given by the limit of the classes $\xi_{E, \dot{S}_{E}, n}$ as in Lemma \ref{coherence1}. By \cite[\S 2.3]{Dillery}, $\check{H}^{2}(O_{S}^{\text{perf}}/O_{S}, P_{E, \dot{S}_{E}})$ is in bijective correspondence with isomorphism classes of $P_{E, \dot{S}_{E}}$-gerbes (over $\text{Sch}/O_{F,S}$) split over $O_{S}^{\text{perf}}$; fix such a gerbe $\gerbeE_{E, \dot{S}_{E}}$. For any $[Z \to T] \in \mathcal{T}_{E}$, the group $H^{1}(\gerbeE_{E, \dot{S}_{E}}, Z \to T)$ is defined identically as above. We have the usual inflation-restriction exact sequence
$$1 \to H^{1}(O_{S}, T) \to H^{1}(\gerbeE_{E, \dot{S}_{E}}, Z \to T) \to \Hom_{O_{F,S}}(P_{E_{i}, \dot{S}_{i}}, Z) \to H^{2}(O_{S}, T),$$ where the last map is the composition of the direct limit of the maps $\Theta^{P}_{E, \dot{S}_{E}, n}$ defined by equation \eqref{thetaP} with the natural map $H^{2}(O_{F,S}, Z) \to H^{2}(O_{F,S}, T)$. 

Fix a 2-cochain $c_{E,S} \in [\text{Res}_{E/S}(\mathbb{G}_{m})](O_{E,S}^{\bigotimes_{O_{F,S}}3})$, a cofinal system $\{n_{i}\}_{i \in \mathbb{N}}$ in $\mathbb{N}^{\times}$, and $n_{i}$-root maps $k_{i}$ as in \S 3.1. 
Recall that if $\Psi_{E,S,n_{i}}(\text{id}) := \beta_{i} \in \text{Maps}(S_{E}, M_{E, \dot{S}_{E}, n_{i}}^{\vee})_{0}$, we showed in \S 3.1 that $\xi_{E, \dot{S}_{E}, n_{i}}$ is represented by the \textit{explicit} 2-cocycle $$\dot{\xi}_{E, \dot{S}_{E}, n_{i}} := d(\overline{k_{i}(c_{E,S})}) \underset{O_{E,S}/O_{F,S}} \sqcup \beta_{i},$$ where for $x \in  [\text{Res}_{E/S}(\mathbb{G}_{m})](R)$, $\bar{x}$ is its image in $[\text{Res}_{E/S}(\mathbb{G}_{m})/\mathbb{G}_{m}](R)$. Set $\dot{\xi}_{E, \dot{S}_{E}} := \varprojlim \dot{\xi}_{E, \dot{S}_{E}, n_{i}}$.

Using  \cite[Fact 3.2.3]{Tasho2} and \cite[Lems. 4.5, 4.6]{Dillery} for \v{C}ech cup product computations, replacing the extension $E/F$ loc. cit. with $O_{E,S}/O_{F,S}$, we see that $p_{i+1,i}(\dot{\xi}_{E, \dot{S}_{E}, n_{i+1}}) = \dot{\xi}_{E, \dot{S}_{E}, n_{i}}$. Giving an element of $Z^{1}(\mathcal{E}_{\dot{\xi}_{E, \dot{S}_{E}, n_{i}}}, Z \to T)$ is equivalent to giving a \textit{$\dot{\xi}_{E, \dot{S}_{E}, n_{i}}$-twisted $Z$ -cocycle in $T$} (cf. \cite[\S 2.5]{Dillery}), which is a pair $(x, \psi)$, where $P_{E_{i}, \dot{S}_{i}} \xrightarrow{\psi} Z$ and $x \in T(O_{S}^{\text{perf}} \otimes_{O_{F,S}} O_{S}^{\text{perf}})$ such that $dx = \psi(\dot{\xi}_{E, \dot{S}_{E}, n_{i}})$. This description makes it clear that $\varinjlim_{i} H^{1}(\gerbeE_{\dot{\xi}_{E, \dot{S}_{E},n_{i}}}, Z \to T) \to H^{1}(\gerbeE_{\dot{\xi}_{E, \dot{S}_{E}}}, Z \to T)$ is bijective.


For a fixed $[Z \to T]$ in $\mathcal{T}_{E}$, we set $\overline{T} := T/Z$, and recall our notation with cocharacter groups from earlier in this subsection. Note that for $i$ large enough so that $\text{exp}(Z)$ divides $n_{i}$, for $g \in \overline{Y}[S_{E}]_{0}$, we have $n_{i} \cdot g \in Y[S_{E}]_{0}$, the restriction of $n_{i} \cdot g$ to the subgroup $\text{Res}_{E/S}(\mu_{n_{i}})/\mu_{n_{i}}$ factors through the subgroup $Z$, and equals the map $$\frac{\text{Res}_{E/S}(\mu_{n_{i}})}{\mu_{n_{i}}} \xrightarrow{[g] \times -} Z$$ induced by \eqref{pairing1}, where $[g]$ denotes the image of $g$ in $A^{\vee}[S_{E}]_{0}$ via \eqref{Y0SES}.

Define $A^{\vee}[\dot{S}_{E}]^{N_{E/F}}$ to be $A^{\vee}[S_{E}]_{0}^{N_{E/F}} \cap A^{\vee}[\dot{S}_{E}]_{0}$, which is in bijection with $\Hom(P_{E, \dot{S}_{E}}, Z)^{\Gamma}$ via the map $\Psi_{E,S}$ defined in Lemma \ref{threepartlemma}. Following the linear algebraic situation for the group $P_{\dot{V}}$, define $\overline{Y}[S_{E}, \dot{S}_{E}]_{0}$ as the preimage of $A^{\vee}[\dot{S}_{E}]_{0}$ in $\overline{Y}[S_{E}]_{0}$, and set $\overline{Y}[S_{E}, \dot{S}_{E}]_{0}^{N_{E/F}} := \overline{Y}[S_{E}, \dot{S}_{E}]_{0} \cap \overline{Y}[S_{E}]_{0}^{N_{E/F}}$. We are now ready to give the first version of the extended Tate-Nakayama isomorphism, which is the analogue of \cite[Prop. 3.7.8]{Tasho2}: 

\begin{prop}\label{tatenak1} \begin{enumerate}
\item{Given $\bar{\Lambda} \in \overline{Y}[S_{E}, \dot{S}_{E}]_{0}^{N_{E/F}}$ and $i$ large enough so that $\text{exp}(Z)$ divides $n_{i}$, we may define a $\dot{\xi}_{E, \dot{S}_{E}, n_{i}}$-twisted 1-cocycle valued in $T$ by the pair $$z_{\bar{\Lambda},i} := (\overline{k_{i}(c_{E,S})} \underset{O_{E,S}/O_{F,S}} \sqcup n_{i} \bar{\Lambda}, \Psi^{-1}_{E, S, n_{i}}([\bar{\Lambda}])),$$ where the unbalanced cup product is with respect to the pairing induced by \eqref{Y0pairing}. }

\item{The pullback $p_{i+1,i}^{*}(z_{\bar{\Lambda},i})$ coincides with the $\dot{\xi}_{E, \dot{S}_{E}, n_{i+1}}$-twisted cocycle $z_{\bar{\Lambda},i+1}$. Thus, pulling back any $z_{\Lambda, i}$ to $\gerbeE_{\dot{\xi}_{E, \dot{S}_{E}}}$ defines the same $\dot{\xi}_{E, \dot{S}_{E}}$-twisted cocycle, denoted by $z_{\bar{\Lambda}}$.}

\item{The assignment $\bar{\Lambda} \mapsto z_{\bar{\lambda}}$ defines an isomorphism $$\dot{\iota}_{E, \dot{S}_{E}} \colon \frac{\overline{Y}[S_{E}, \dot{S}_{E}]_{0}^{N_{E/F}}}{I_{E/F}Y[S_{E}]_{0}} \to H^{1}(\gerbeE_{\dot{\xi}_{E, \dot{S}_{E}}}, Z \to T)$$ which is functorial in $[Z \to T] \in \mathcal{T}_{E}$ and makes the following diagram commute:
\[
\begin{tikzcd}
1 \arrow{r} & \widehat{H}^{-1}(\Gamma_{E/F}, Y[S_{E}]_{0})  \arrow{r} \arrow["\text{TN}"]{d} & \frac{\overline{Y}[S_{E}, \dot{S}_{E}]_{0}^{N_{E/F}}}{I_{E/F}Y[S_{E}]_{0}} \arrow["\dot{\iota}_{E, \dot{S}_{E}} "]{d} \arrow{r} & A^{\vee}[\dot{S}_{E}]^{N_{E/F}} \arrow["\Psi_{E,S}^{-1}"]{d} \arrow{r} & \widehat{H}^{0}(\Gamma_{E/F}, Y[S_{E}]_{0}) \arrow["-\text{TN}"]{d}  \\
1 \arrow{r} & H^{1}(O_{F,S}, T) \arrow{r} & H^{1}(\gerbeE_{\dot{\xi}_{E, \dot{S}_{E}}}, Z \to T) \arrow{r} & \Hom(P_{E, \dot{S}_{E}}, Z)^{\Gamma} \arrow{r} & H^{2}(O_{F,S}, T). 
\end{tikzcd}
\]
 }
\end{enumerate}
\end{prop}

\begin{proof} 
Proving the first claim just means showing, for fixed large enough $i$, the equality $$d(\overline{k_{i}(c_{E,S})} \underset{O_{E,S}/O_{F,S}} \sqcup n_{i} \bar{\Lambda}) = \Psi^{-1}_{E, S, n_{i}}([\bar{\Lambda}])(\dot{\xi}_{E, \dot{S}_{E}, n_{i}}).$$ 

Viewing $n_{i} \bar{\Lambda}$ as a $-1$-cochain, we see that $d(n_{i} \bar{\Lambda}) = 0$, since by construction $\bar{\Lambda}$ is killed by $N_{E/F}$. Hence, it follows from \cite[Prop. 4.4]{Dillery},that 
\begin{equation}\label{tatenak1.1} d(\overline{k_{i}(c_{E,S})} \underset{O_{E,S}/O_{F,S}} \sqcup n_{i} \bar{\Lambda})  = d[\overline{k_{i}(c_{E,S})}] \underset{O_{E,S}/O_{F,S}} \sqcup n_{i} \bar{\Lambda},\end{equation}
and now since $d[\overline{k_{i}(c_{E,S})}]$ lies in the subgroup $[\text{Res}_{E/S}(\mu_{n_{i}})/\mu_{n_{i}}]((O_{S}^{\text{perf}})^{\bigotimes_{O_{F,S}}3})$ and we know that the restriction of $n_{i} \bar{\Lambda}$ to $\text{Res}_{E/S}(\mu_{n_{i}})/\mu_{n_{i}}$ is equal to $\Phi_{E,S,n_{i}}([\bar{\Lambda}])$, we can rewrite the right-hand term of \eqref{tatenak1.1} as $d[\overline{k_{i}(c_{E,S})}] \underset{O_{E,S}/O_{F,S}} \sqcup \Phi_{E,S,n_{i}}([\bar{\Lambda}])$. By functoriality, this term can be rewritten as 
$$d[\overline{k_{i}(c_{E,S})}] \underset{O_{E,S}/O_{F,S}} \sqcup \Phi_{E,S,n_{i}}( \Psi_{E,S,n_{i}}^{-1}([\bar{\Lambda}])^{\vee}(\beta_{i})),$$ which again by functoriality may be expressed as $$d[\overline{k_{i}(c_{E,S})}] \underset{O_{E,S}/O_{F,S}} \sqcup \Psi_{E,S,n_{i}}^{-1}([\bar{\Lambda}]) \circ \Phi_{E,S,n_{i}}(\beta_{i}) =  \Psi_{E,S,n_{i}}^{-1}([\bar{\Lambda}])(d[\overline{k_{i}(c_{E,S})}] \underset{O_{E,S}/O_{F,S}} \sqcup \Phi_{E,S,n_{i}}(\beta_{i})),$$ where to obtain the above equality we are using the fact $\Psi_{E,S,n_{i}}^{-1}([\bar{\Lambda}])$ is $\Gamma_{E/F}$-fixed to apply \cite[Lem. 4.6]{Dillery}. By definition this last term equals $\Psi_{E,S,n_{i}}^{-1}([\bar{\Lambda}])(\dot{\xi}_{E, \dot{S}_{E}, n_{i}})$, as desired.

We now move to the second claim of the proposition. The first step is noting that $p_{i+1,i} \circ \Psi^{-1}_{E,S,n_{i+1}}([\bar{\Lambda}]) = \Psi^{-1}_{E,S,n_{i}}([\bar{\Lambda}])$, since, as discussed in Lemma \ref{threepartlemma}, the maps $\Psi_{E,S,n}$ are compatible with the projection maps for the system $\{P_{E,\dot{S}_{E}, n_{i}}\}_{i}$. Moreover, we have by the $\Z$-bilinearity of the unbalanced cup product and coherence of the system of maps $\{k_{i}\}_{i}$ that $$ \overline{k_{i+1}(c_{E,S})} \underset{O_{E,S}/O_{F,S}} \sqcup n_{i+1} \bar{\Lambda} =  \overline{k_{i+1}(c_{E,S})} \underset{O_{E,S}/O_{F,S}} \sqcup (\frac{n_{i+1}}{n_{i}}) [n_{i} \bar{\Lambda}] = \overline{k_{i}(c_{E,S})} \underset{O_{E,S}/O_{F,S}} \sqcup n_{i} \bar{\Lambda},$$ concluding the proof of the second claim.

It is clear that the map $\bar{\Lambda} \mapsto z_{\bar{\Lambda}}$ defines a functorial homomorphism from $\overline{Y}[S_{E}, \dot{S}_{E}]_{0}^{N_{E/F}}$ to $H^{1}(\gerbeE_{\dot{\xi}_{E, \dot{S}_{E}}}, Z \to T)$. Moreover, if $\bar{\Lambda}$ lies in the subgroup $Y[S_{E}]_{0}$, we first note that $[\bar{\Lambda}]$ vanishes in $A^{\vee}[S_{E}]_{0}$, so that the homomorphism associated to $z_{\bar{\Lambda}}$ is trivial. By $\Z$-bilinearity and the fact that already $\bar{\Lambda} \in Y[S_{E}]_{0}$, the associated twisted cocycle (which is, by the previous line, an actual cocycle) is given by $\overline{c_{E,S}} \underset{O_{E,S}/O_{F,S}} \sqcup \bar{\Lambda}$, which, since $\overline{c_{E,S}}$ is valued in the finite \'{e}tale extension $O_{E,S}/O_{F,S}$, \cite[Prop. 4.2]{Dillery} and \cite[\S 4.3]{Tasho} tell us that (after applying the appropriate comparison isomorphisms) this cup product may be computed as the usual Galois-cohomological cup product $\overline{c_{E,S}} \cup \bar{\Lambda}$, which sends all of $I_{E/F}Y[S_{E}]_{0}$ to 1-coboundaries, showing that the above map induces a functorial homomorphism $$\frac{\overline{Y}[S_{E}, \dot{S}_{E}]_{0}^{N_{E/F}}}{I_{E/F}Y[S_{E}]_{0}}  \to H^{1}(\gerbeE_{\dot{\xi}_{E, \dot{S}_{E}}}, Z \to T),$$ as asserted. This argument also shows that the top square in the diagram of the proposition commutes. The commutativity of the middle square is by construction, and the final square commutes by the diagram in Lemma \ref{bigdiag1}. Since all horizontal maps in the diagram apart from $\dot{\iota}_{E, \dot{S}_{E}}$ are isomorphisms, it is an isomorphism as well by the five-lemma.
\end{proof}

The issue now is that, given our non-canonical explicit gerbe $\gerbeE_{\dot{\xi}_{E, \dot{S}_{E}}}$, it is not clear that such an isomorphism will be canonical, or even that the groups $H^{1}(\gerbeE_{E, \dot{S}_{E}}, Z \to T)$ are canonical. The following result addresses these concerns:

\begin{prop}\label{tatenak1bis} The group $H^{1}(\gerbeE_{E, \dot{S}_{E}}, Z \to T)$ is independent of the choice of gerbe $\gerbeE_{E, \dot{S}_{E}}$ up to unique isomorphism, and is equipped with a canonical functorial isomorphism $\iota_{E, \dot{S}_{E}}$ to $\frac{\overline{Y}[S_{E}, \dot{S}_{E}]_{0}^{N_{E/F}}}{I_{E/F}Y[S_{E}]_{0}}$ that fits into the commutative diagram of Proposition \ref{tatenak1}.
\end{prop}

\begin{proof} The map $\iota_{E, \dot{S}_{E}}$ is obtained by composing an isomorphism (which the proposition asserts is unique) $H^{1}(\gerbeE_{E, \dot{S}_{E}}, Z \to T) \to H^{1}(\gerbeE_{\dot{\xi}_{E, \dot{S}_{E}}}, Z \to T)$ induced by any isomorphism of $P_{E,\dot{S}_{E}}$-gerbes $\gerbeE_{E, \dot{S}_{E}} \xrightarrow{\sim} \gerbeE_{\dot{\xi}_{E, \dot{S}_{E}}}$ and then applying $\dot{\iota}_{E, \dot{S}_{E}}$ from Proposition \ref{tatenak1}.

This proposition requires work to show, but all the necessary arguments are done in \cite[\S 3.7]{Tasho2}. The main ingredient is Lemma 3.7.10 loc. cit., which is purely group-theoretic and carries over to our setting unchanged (in the statement of that Lemma, eliminate the use of $S$ and replace $\mathbb{N}_{S}$ by $\mathbb{N}$). Once this result is known,  \cite[Corollary 3.7.11]{Tasho2} proves the proposition. The proof of this corollary relies on Lemma 3.7.9 loc. cit., which holds in our setting with $\mathbb{N}_{S}$ replaced by $\mathbb{N}$, Proposition 3.7.8 loc. cit., which is our Proposition \ref{tatenak1}, and the finiteness of $H^{1}(O_{F,S}, T)$, which is true in our setting as well.
\end{proof}

Note that, in particular, the isomorphism $\iota_{E, \dot{S}_{E}}$ does not depend on the choice of cochain $c_{E,S}$ lifting a representative of the canonical Tate class in $H^{2}(O_{F,S}, \text{Res}_{E,S}(\mathbb{G}_{m})/\mathbb{G}_{m})$ which was used to construct the explicit gerbes $\gerbeE_{\dot{\xi}_{E, \dot{S}_{E},n_{i}}}$ and the isomorphism $\dot{\iota}_{E, \dot{S}_{E}}$ in Proposition \ref{tatenak1}.

In order to extend the isomorphism of Proposition \ref{tatenak1} to $\gerbeE_{\dot{V}}$, we need to vary the extension $E/F$. As such, let $K/F$ be a finite Galois extension containing $E$, and $(S' \dot{S}'_{K})$ be a pair satisfying Conditions \ref{placeconditions}. We may assume that $S \subset S'$ and $\dot{S}_{E} \subset (\dot{S}'_{K})_{E}$. Let $\gerbeE_{K, \dot{S}'_{K}}$ and $\gerbeE_{E, \dot{S}_{E}}$ be gerbes corresponding to the canonical classes $\xi_{K, \dot{S}'_{K}} \in \check{H}^{2}(O_{S'}^{\text{perf}}/O_{F,S'}, P_{K, \dot{S}'_{K}})$ and $\xi_{E, \dot{S}_{E}} \in \check{H}^{2}(O_{S}^{\text{perf}}/O_{F,S}, P_{E, \dot{S}_{E}})$, respectively. The first step is to construct an inflation map 
$$\text{Inf} \colon H^{1}(\gerbeE_{E, \dot{S}_{E}}, Z \to T) \to H^{1}(\gerbeE_{K, \dot{S}'_{K}}, Z \to T).$$

We begin by taking the pullback $(P_{E, \dot{S}_{E}})_{O_{F,S'}}$-gerbe $(\gerbeE_{E, \dot{S}_{E}})_{O_{F,S'}}$, which splits over $O_{F,S'} \cdot O_{S} \subset O_{S'}^{\text{perf}} \subset \overline{F}$ and represents the image of $\xi_{E, \dot{S}_{E}}$ in $\check{H}^{2}(O_{S'}^{\text{perf}}/O_{F,S'}, (P_{E, \dot{S}_{E}})_{O_{F,S'}})$.
We have a projection map $P_{K, \dot{S}'_{K}} \to (P_{E, \dot{S}_{E}})_{O_{F,S'}}$ given by the inverse limit of the finite-level projection maps, which on $\check{H}^{2}$ sends $\xi_{K, \dot{S}'_{K}}$ to the image of $\xi_{E, \dot{S}_{E}}$, by Lemma \ref{Ptheta2}. Using this equality of cocycles, picking normalizations of $\gerbeE_{E, \dot{S}_{E}}$ and $\gerbeE_{K, \dot{S}'_{K}}$ and using Construction \ref{changeofgerbe} allows us to construct a (non-canonical) morphism of stacks over $O_{F,S'}$ from $\gerbeE_{K, \dot{S}'_{K}}$ to $(\gerbeE_{E, \dot{S}_{E}})_{O_{F,S'}}$. By pulling back torsors via the composition of functors $$\gerbeE_{K, \dot{S}'_{K}} \to (\gerbeE_{E, \dot{S}_{E}})_{O_{F,S'}} \to \gerbeE_{E, \dot{S}_{E}},$$ we get the desired inflation map. 

The map we just constructed from $H^{1}(\gerbeE_{E, \dot{S}_{E}}, Z \to T)$ to $H^{1}(\gerbeE_{K, \dot{S}'_{K}}, Z \to T)$ is evidently functorial in $[Z \to T] \in \mathcal{T}_{E}$, but it is not a priori clear that it is canonical. 

Recall the system $(E_{i}, S_{i}, \dot{S}_{i})$ constructed in \S 3.2. For any $P_{E_{i}, \dot{S}_{i}}$-gerbe $\gerbeE_{i}$ over $O_{F,S_{i}}$ split over $O_{S_{i}}^{\text{perf}}$ representing the \v{C}ech 2-cocycle $\xi_{E_{i}, \dot{S}_{i}}$, we first take the $(P_{E_{i}, \dot{S}_{i}})_{F}$-gerbe $(\gerbeE_{i})_{F}$, which corresponds to the image of $\xi_{E_{i}, \dot{S}_{i}}$ in $\check{H}^{2}(\overline{F}/F, (P_{E_{i}, \dot{S}_{i}})_{F})$. By construction of the canonical class $\xi$, the image of $\xi$ in $\check{H}^{2}(\overline{F}/F, (P_{E_{i}, \dot{S}_{i}})_{F})$ equals this image of $\xi_{E_{i}, \dot{S}_{i}}$. Thus, after normalizing the gerbes $\gerbeE_{\dot{V}}$ and $\gerbeE_{i}$ and choosing a coboundary, we get a functor $\gerbeE_{\dot{V}} \to \gerbeE_{i}$, and thus by pullback a group homomorphism $$\text{Inf} \colon H^{1}(\gerbeE_{i}, Z \to T) \to H^{1}(\gerbeE_{\dot{V}}, Z \to T).$$

The following result describes how these inflation maps behave, and its proof is the same as its number-field analogues (\cite[Props. 3.7.12, 3.7.13]{Tasho2}):
\begin{prop}\label{inflation1}\label{directlimit} \begin{enumerate}
\item{The inflation map constructed above is independent of the choice of functor $\gerbeE_{K, \dot{S}'_{K}} \to \gerbeE_{E, \dot{S}_{E}}$, injective, functorial in $[Z \to T] \in \mathcal{T}_{E}$, and fits into the commutative diagrams:
\[
\begin{tikzcd}
H^{1}(\gerbeE_{E, \dot{S}_{E}}, Z \to T) \arrow["\text{Inf}"]{r} & H^{1}(\gerbeE_{K, \dot{S}'_{K}}, Z \to T) \\
\frac{\overline{Y}[S_{E}, \dot{S}_{E}]_{0}^{N_{E/F}}}{I_{E/F}Y[S_{E}]_{0}} \arrow["\iota_{E, \dot{S}_{E}}"]{u} \arrow["!"]{r} & \frac{\overline{Y}[S'_{K}, \dot{S}'_{K}]_{0}^{N_{K/F}}}{I_{K/F}Y[S'_{K}]_{0}} \arrow["\iota_{K, \dot{S}'_{K}}"]{u},
\end{tikzcd}
\] 
\[
\begin{tikzcd}
1 \arrow{r} & H^{1}(O_{F,S}, T) \arrow["\text{Inf}"]{d} \arrow{r} & H^{1}(\gerbeE_{E, \dot{S}_{E}}, Z \to T) \arrow["\text{Inf}"]{d} \arrow{r} & \Hom(P_{E, \dot{S}_{E}}, Z)^{\Gamma} \arrow{d} \\
1 \arrow{r} & H^{1}(O_{F,S'}, T) \arrow{r} & H^{1}(\gerbeE_{K, \dot{S}'_{K}}, Z \to T) \arrow{r} & \Hom(P_{K, \dot{S}'_{K}}, Z)^{\Gamma}. \\
\end{tikzcd}
\]}
\item{The above inflation maps splice together to give a canonical isomorphism of functors $\mathcal{T} \to \text{AbGrp}$:
$$\varinjlim_{i} H^{1}(\gerbeE_{i}, Z \to T) \to H^{1}(\gerbeE_{\dot{V}}, Z \to T).$$}
\end{enumerate}
\end{prop}

We are now in a position to prove Theorem \ref{toritatenak}. We obtain the functorial isomorphism $\iota_{\dot{V}}$ by the composition of 
functorial isomorphisms
$$\overline{Y}[V_{\overline{F}}, \dot{V}]_{0,+,\text{tor}} \xrightarrow{\text{Lem. \ref{tasho3.7.6}}}\varinjlim_{i} \frac{\overline{Y}[(S_{i})_{E_{i}}, \dot{S}_{i}]_{0}^{N_{E_{i}/F}}}{I_{E_{i}/F}Y[(S_{i})_{E}]_{0}} \xrightarrow{\varinjlim_{i} \iota_{E_{i}, \dot{S}_{i}}} \varinjlim_{i} H^{1}(\gerbeE_{i}, Z \to T) \xrightarrow{\text{Prop. \ref{inflation1}}} H^{1}(\gerbeE_{\dot{V}}, Z \to T),$$ which is canonical and well-defined by Proposition \ref{inflation1}. 

We conclude this subsection by collecting local-to-global consequences of Theorem \ref{toritatenak}, whose proof is the same as \cite[Corollaries 3.7.4, 3.7.5]{Tasho2} (using our Lemma \ref{3.6.2} for the second part):

\begin{cor}\label{localtoglobal2}\label{Tasho3.7.4}  \begin{enumerate}
\item{We have the following commutative diagram with exact bottom row
\[
\begin{tikzcd}
H^{1}(\gerbeE_{\dot{V}}, Z \to T) \arrow["(\text{loc}_{v})_{v}"]{r} & \bigoplus_{v \in \dot{V}} H^{1}(\gerbeE_{v}, Z \to T) & \\
\overline{Y}[V_{\overline{F}}, \dot{V}]_{0,+,\text{tor}} \arrow["\iota_{\dot{V}}"]{u} \arrow["(l_{v})_{v}"]{r} & \bigoplus_{v \in \dot{V}} \overline{Y}_{+v, \text{tor}} \arrow["(\iota_{v})_{v}"]{u} \arrow["\Sigma"]{r} & \frac{\overline{Y}}{IY}[\text{tor}].
\end{tikzcd}
\]}
\item{For $G$ a connected reductive group and $x \in H^{1}(\gerbeE_{\dot{V}}, Z \to G)$, $\text{loc}_{v}(x)$ is the neutral element in $H^{1}(\gerbeE_{v}, Z \to G)$ for almost all $v \in \dot{V}$.}
\end{enumerate}
\end{cor}




\subsection{Extending to reductive groups} Let $\mathcal{R}$ denote the full subcategory of $\mathcal{A}$ consisting of objects $[Z \to G]$ where $G$ is a connected reductive group over $F$. In \cite[\S 3.8]{Tasho2}, it is necessary for duality to replace the sets $H^{1}(\gerbeE_{\dot{V}}, Z \to G)$ with a quotient, denoted by $H^{1}_{\text{ab}}(\gerbeE_{\dot{V}}, Z \to G)$. However, in our case $H^{1}(\gerbeE_{\dot{V}}, Z \to G)$ suffices due to the vanishing of $H^{1}(F, G)$ for all simply-connected (semi-simple) connected $G$ (which is an immediate consequence of \cite[Thm. 2.4]{Thang1}).

The first step in extending Theorem \ref{toritatenak} to $\mathcal{R}$ is defining an analogue of the linear algebraic data $\overline{Y}[V_{\overline{F}}, \dot{V}]_{0,+,\text{tor}}([Z \to T])$ for $[Z \to T] \in \mathcal{T}$. For a maximal $F$-torus $T$ of $G$, we let $T_{\text{sc}}$ denote $T \cap G_{\text{sc}}$, a maximal torus of $G_{\text{sc}}$. We then can define the abelian group $$\varinjlim_{(E, S_{E}, \dot{S}_{E})} \frac{[X_{*}(T/Z)/X_{*}(T_{\text{sc}})][S_{E}, \dot{S}_{E}]_{0}^{N_{E/F}}}{I_{E/F}( [X_{*}(T)/X_{*}(T_{\text{sc}})][S_{E}]_{0})},$$ where the colimit is over any cofinal system of triples $(E, S_{E}, \dot{S}_{E})$, where $E/F$ is a finite Galois extension splitting $T$ and the pair $(S_{E}, \dot{S}_{E})$ satisfies Conditions \ref{placeconditions}; the transition maps are given by the map $!$ as in \S 4.2, and we define $[X_{*}(T/Z)/X_{*}(T_{\text{sc}})][S_{E}, \dot{S}_{E}]_{0}$ to be all $\sum c_{w}[w] \in [X_{*}(T/Z)/X_{*}(T_{\text{sc}})][S_{E}]_{0}$ such that if $w \notin \dot{S}_{E}$, then $c_{w} \in X_{*}(T)/X_{*}(T_{\text{sc}})$ (the superscript $N_{E/F}$ denotes the elements killed by the $E/F$-norm).

Now for two such tori $T_{1}, T_{2}$, we can define a map 
\begin{equation}\label{limitmap} \varinjlim  \frac{[X_{*}(T_{1}/Z)/X_{*}(T_{1,\text{sc}})][S_{E}, \dot{S}_{E}]_{0}^{N_{E/F}}}{I_{E/F}( [X_{*}(T_{1})/X_{*}(T_{1,\text{sc}})][S_{E}]_{0})} \to  \varinjlim  \frac{[X_{*}(T_{2}/Z)/X_{*}(T_{2,\text{sc}})][S_{E}, \dot{S}_{E}]_{0}^{N_{E/F}}}{I_{E/F}( [X_{*}(T_{2})/X_{*}(T_{2,\text{sc}})][S_{E}]_{0})}
\end{equation}
as follows. By \cite[Lem. 4.2]{Tasho}, for any $g \in G(F^{\text{sep}})$ such that $\text{Ad}(g)(T_{1}) = T_{2}$, we get an isomorphism $X_{*}(T_{1}/Z)/X_{*}(T_{1,\text{sc}}) \to X_{*}(T_{2}/Z)/X_{*}(T_{2,\text{sc}})$ which is independent of the choice of $g$, and is thus $\Gamma$-equivariant. It follows that $\text{Ad}(g)$ also induces the desired homomorphism \eqref{limitmap} on direct limits. We then define a functor $\mathcal{R} \to \text{AbGrp}$ given by $$\overline{Y}[V_{\overline{F}}, \dot{V}]_{0,+,\text{tor}}([Z \to G]) := \varinjlim [\varinjlim_{(E, S_{E}, \dot{S}_{E})} \frac{[X_{*}(T/Z)/X_{*}(T_{\text{sc}})][S_{E}, \dot{S}_{E}]_{0}^{N_{E/F}}}{I_{E/F}( [X_{*}(T)/X_{*}(T_{\text{sc}})][S_{E}]_{0})}],$$ where the outer colimit is over all maximal $F$-tori $T$ of $G$ via the maps constructed above. It is clear that this extends the functor $\overline{Y}[V_{\overline{F}}, \dot{V}]_{0,+,\text{tor}}$ constructed in the previous section for $\mathcal{T} \subset \mathcal{R}$, so our notation is justified. In what follows, we will always take our colimits over the fixed cofinal system $(E_{i}, S_{i}, \dot{S}_{i})$ constructed above (such a system eventually splits any $F$-torus $T$).

\begin{thm}\label{reductivetatenak} The isomorphism of functors $\iota_{\dot{V}}$ from Theorem \ref{toritatenak} extends to a unique isomorphism of functors (valued in pointed sets) on $\mathcal{R}$, also denoted by $\iota_{\dot{V}}$, from $\overline{Y}[V_{\overline{F}}, \dot{V}]_{0,+,\text{tor}}$ to $H^{1}(\gerbeE_{\dot{V}}, -)$. 
\end{thm}

\begin{proof} Fix $[Z \to G]$ in $\mathcal{R}$ with maximal $F$-torus $T$. We claim that the fibers of the composition $$\overline{Y}[V_{\overline{F}}, \dot{V}]_{0,+,\text{tor}}([Z \to T]) \xrightarrow{\iota_{\dot{V}}} H^{1}(\gerbeE_{\dot{V}}, Z \to T) \to H^{1}(\gerbeE_{\dot{V}}, Z \to G)$$ are torsors under the image of $Y[V_{\overline{F}}]_{0,\Gamma,\text{tor}}(T_{\text{sc}})$. By twisting, it's enough to prove this for the fiber over the trivial class in $H^{1}(\gerbeE_{\dot{V}}, Z \to G)$. If $x$ is in this fiber then $\iota_{\dot{V}}(x) \in H^{1}(F, T)$ [since $H^{1}(\gerbeE_{\dot{V}}, Z \to T) \to \Hom_{F}(P_{\dot{V}}, Z)$ factors through $H^{1}(\gerbeE_{\dot{V}}, Z \to G)$], so that $x \in Y[V_{\overline{F}}]_{0,\Gamma,\text{tor}}(T)$. Moreover, $\iota_{\dot{V}}(x)$ maps to the neutral class via the map $H^{1}(F, T) \to H^{1}(F,G)$. We have the commutative diagram of pointed sets with exact rows
\[
\begin{tikzcd}
(G/T)(F) \arrow{r} & H^{1}(F, T) \arrow{r} & H^{1}(F, G) \\
(G_{\text{sc}}/T_{\text{sc}})(F) \arrow{u} \arrow{r} & H^{1}(F, T_{\text{sc}}) \arrow{r} \arrow{u} & H^{1}(F, G_{\text{sc}}) \arrow{u},
\end{tikzcd}
\]
and since the natural map $G_{\text{sc}}/T_{\text{sc}} \to G/T$ is an isomorphism (of $F$-schemes, not groups), we may lift the image of $x$ in $H^{1}(F,T)$ to an element $x_{\text{sc}} \in H^{1}(F, T_{\text{sc}})$. Now the claim is clear by the functoriality of Tate-Nakayama duality for tori.

The above claim immediately implies that we have an injective map 
$$\frac{\overline{Y}[V_{\overline{F}}, \dot{V}]_{0,+,\text{tor}}([Z \to T])}{\text{Im}[Y[V_{\overline{F}}]_{0,\Gamma,\text{tor}}(T_{\text{sc}})]} \to H^{1}(\gerbeE_{\dot{V}}, Z \to G).$$ Arguments involving cocharacter modules (see \cite{Tasho2}, proof of Theorem 3.8.1) show that the image $\text{Im}[Y[V_{\overline{F}}]_{0,\Gamma,\text{tor}}(T_{\text{sc}})]$ is exactly the kernel of the natural map $$\overline{Y}[V_{\overline{F}}, \dot{V}]_{0,+,\text{tor}}([Z \to T]) \to \overline{Y}[V_{\overline{F}}, \dot{V}]_{0,+,\text{tor}}([Z \to G]),$$ and so, putting the above two observations together, we have a natural inclusion \begin{equation}\label{imageinclusion}
\text{Im}[\overline{Y}[V_{\overline{F}}, \dot{V}]_{0,+,\text{tor}}([Z \to T]) \to  \overline{Y}[V_{\overline{F}}, \dot{V}]_{0,+,\text{tor}}([Z \to G])] \hookrightarrow H^{1}(\gerbeE_{\dot{V}}, Z \to G).
\end{equation}

The next key observation is that any two elements of $\overline{Y}[V_{\overline{F}}, \dot{V}]_{0,+,\text{tor}}([Z \to G])$ lie in the image of the group $\overline{Y}[V_{\overline{F}}, \dot{V}]_{0,+,\text{tor}}([Z \to T])$ for some maximal $F$-torus $T \subset G$. The analogous argument using elliptic maximal tori (over the local fields $F_{v}$) in the proof of \cite[Thm. 3.8.1]{Tasho2} works for us, once we replace \cite[Corollary 7.3]{PR} with \cite[Lem. 3.6.1]{Thang4}, using that $H^{2}(F_{v}, T'_{\text{sc}})$ vanishes for any $F_{v}$-anisotropic maximal torus $T'_{\text{sc}}$ (by Tate-Nakayama duality), and the fact that the map $H^{2}(F, T'_{\text{sc}}) \to \prod_{v \in V_{F}} H^{2}(F_{v}, T'_{\text{sc}})$ is injective whenever there exists a place $v \in S$ such that $(T'_{\text{sc}})_{F_{v}}$ is an $F_{v}$-anisotropic maximal torus in a connected semisimple group $G_{\text{sc}}$ (see \cite[Prop. 6.12]{PR}, the proof of which works for function fields). 

We now claim that if $x_{i} \in \overline{Y}[V_{\overline{F}}, \dot{V}]_{0,+,\text{tor}}([Z \to T_{i}])$ for $i=1,2$ map to the same element in $\overline{Y}[V_{\overline{F}}, \dot{V}]_{0,+,\text{tor}}([Z \to G])$, then their images $\iota_{\dot{V}}(x_{i}) \in H^{1}(\gerbeE_{\dot{V}}, Z \to T_{i})$ map to the same element of $H^{1}(\gerbeE_{\dot{V}}, Z \to G)$. Choose ${j}$ large enough so that $E_{j}$ splits $T_{i}$, $\text{exp}(Z) \mid n_{j}$, and $x_{i}$ comes from $\bar{\Lambda}_{i} \in \overline{Y}_{i}[(S_{j})_{E_{j}}, \dot{S}_{j}]_{0}^{N_{E_{j}/F}}$. Denote by $\dot{\gerbeE}_{j}$ the explicit gerbe $\gerbeE_{\dot{\xi}_{E_{j}, \dot{S}_{j}}}$ defined in \S 4.2, similarly with $\dot{\gerbeE}_{j, n_{j}}$ (recall that these depend on choices, cf. \S 4.2). By construction, $\dot{\iota}_{\dot{V}}(x_{i})$ is the inflation of a class in $H^{1}(\dot{\gerbeE}_{j,n_{j}}, Z \to T_{i})$ represented by the twisted 1-cocycle $z_{\bar{\Lambda}_{i},j}$ (defined in Proposition \ref{tatenak1}). 
Since the functor $\gerbeE_{\dot{V}} \to \dot{\gerbeE}_{j,n_{j}}$ used to define inflation factors through $(\dot{\gerbeE}_{j,n_{j}})_{\overline{F}}$, it is enough to show that the $z_{\bar{\Lambda}_{i},j}$ are equivalent as twisted cocycles in $H^{1}((\dot{\gerbeE}_{j,n_{j}})_{\overline{F}}, Z \to G)$, which follows from an identical argument as in  \cite[Lem. 5.9]{Dillery}, replacing the tori loc. cit. with their global analogues defined above (as well as the projective system of finite multiplicative groups).

To summarize, the maps \eqref{imageinclusion} glue across all $T$ to give the desired map $\iota_{\dot{V}}$, which is surjective by Lemma \ref{3.6.2} and injective because any two elements of $\overline{Y}[V_{\overline{F}}, \dot{V}]_{0,+,\text{tor}}([Z \to G])$ both lie in the image of $\overline{Y}[V_{\overline{F}}, \dot{V}]_{0,+,\text{tor}}([Z \to T])$ for some $T$.  By construction, these isomorphisms extend the isomorphism of functors $\iota_{\dot{V}}$ defined on the full subcategory $\mathcal{T}$, and are functorial with respect to morphisms $[Z \to T] \to [Z \to G]$ in $\mathcal{R}$ given by inclusions of maximal tori defined over $F$. Since every $x \in H^{1}(\gerbeE_{\dot{V}}, Z \to G)$ lies in the image of some $H^{1}(\gerbeE_{\dot{V}}, Z \to T)$, it follows that the extension of $\iota_{\dot{V}}$ to $\mathcal{R}$ also defines an isomorphism of functors on $\mathcal{R}$.
\end{proof}

To conclude this subsection, we state some local-global compatibilities that arise from Theorem \ref{reductivetatenak}. Note that, for a fixed $v$, the natural transformation (given by the direct limit of \eqref{Yloc}) $$\overline{Y}[V_{\overline{F}}, \dot{V}]_{0,+,\text{tor}} \xrightarrow{l_{v}} \overline{Y}_{+v,\text{tor}}$$ may be extended to a morphism of functors on $\mathcal{R}$ induced by mapping $f \in [X_{*}(T/Z)/X_{*}(T_{\text{sc}})][S_{E}, \dot{S}_{E}]_{0}$ to an element of $X_{*}(T/Z)/X_{*}(T_{\text{sc}})$ via the same formula as in \eqref{Yloc}. We also define a new functor $\mathcal{R} \to \text{AbGrp}$, denoted by $\overline{Y}_{+,\text{tor}}$, by $$[Z \to G] \mapsto \varinjlim \frac{X_{*}(T/Z)/X_{*}(T_{\text{sc}})}{I(X_{*}(T)/X_{*}(T_{\text{sc}}))}[\text{tor}], $$ where $I$ is the augmentation ideal of $\Gamma$, the colimit is taken over all maximal $F$-tori $T$ of $G$, and the transition maps are induced by $\text{Ad}(g)$. 

\begin{cor}\label{localtoglobal3} We have a commutative diagram with exact bottom row:
\[
\begin{tikzcd}
H^{1}(\gerbeE_{\dot{V}}, Z \to G) \arrow["(\text{loc}_{v})_{v}"]{r} & \bigsqcup_{v \in V} H^{1}(\gerbeE_{v}, Z \to G) & \\
\overline{Y}[V_{\overline{F}}, \dot{V}]_{0,+,\text{tor}}([Z \to G]) \arrow["(l_{v})_{v}"]{r} \arrow["\iota_{\dot{V}}"]{u} & \bigoplus_{v \in V} \overline{Y}_{+v,\text{tor}}([Z \to G]) \arrow["\Sigma"]{r} \arrow["(\iota_{v})_{v}"]{u} & \overline{Y}_{+,\text{tor}}([Z \to G]),
\end{tikzcd}
\]
where the symbol $\bigsqcup$ denotes the subset of the direct product of pointed sets in which all but finitely many coordinates equal the neutral element, and the map $\Sigma$ makes sense since any maximal $F_{v}$-torus of $G_{F_{v}}$ is $G(\overline{F_{v}})$-conjugate to the base-change $T_{F_{v}}$ of a maximal $F$-torus $T$ in $G$.
\end{cor}

\begin{proof} The commutativity is an immediate consequence of Corollary \ref{localtoglobal2}, the functoriality of $\iota_{\dot{V}}$, and the fact that every $x \in H^{1}(\gerbeE_{\dot{V}}, Z \to G)$ lies in the image of some $H^{1}(\gerbeE_{\dot{V}}, Z \to T)$. The exactness of the bottom row is a straightforward character-theoretic argument.
\end{proof}

We also have the following analogue of \cite[Corollary 3.8.2]{Tasho2}:

\begin{cor}\label{localtoglobal4} The image of $$H^{1}(\gerbeE_{\dot{V}}, Z \to G) \xrightarrow{(\text{loc}_{v})_{v}}  \bigsqcup_{v \in V} H^{1}(\gerbeE_{v}, Z \to G)$$ consists precisely of those elements which map trivially under the composition $$\bigsqcup_{v \in \dot{V}} H^{1}(\gerbeE_{v}, Z \to G) \to \bigoplus_{v \in V} \overline{Y}_{+v,\text{tor}}([Z \to G]) \to  \overline{Y}_{+,\text{tor}}([Z \to G]).$$
\end{cor}

\begin{proof} Unlike in \cite{Tasho2}, where work is needed, this is a trivial consequence of Corollary \ref{localtoglobal3}.  
\end{proof}

\subsection{Unramified localizations}
For $[Z \to G] \in \mathcal{R}$, we can pull any $Z$-twisted $G_{\gerbeE_{\dot{V}}}$-torsor $\mathscr{T} \in Z^{1}(\gerbeE_{\dot{V}}, Z \to G)$ back to the $G_{\gerbeE_{\dot{V}, \overline{F_{v}}}}$-torsor $\mathscr{T}_{\overline{F_{v}}}$, and then via picking gerbe normalizations and a $1$-coboundary to get a functor $\Phi \colon \gerbeE_{v} \to \gerbeE_{\dot{V}}$, we set $\text{loc}_{v}(\mathscr{T}):= \Phi^{*}(\mathscr{T}_{\overline{F_{v}}})$, a $Z$-twisted $G_{\gerbeE_{v}}$-torsor. Note that  $\text{loc}_{v}(\mathscr{T})$ depends on our choice of $\Phi$ up to replacing $\text{loc}_{v}(\mathscr{T})$ by the canonically-isomorphic (via translation by $a^{-1}$) torsor $\eta^{*}(\text{loc}_{v}(\mathscr{T}))$, where $\eta \colon \gerbeE_{v} \to \gerbeE_{v}$ is the automorphism induced by a 1-coboundary $d(a)$, for $a \in u_{v}(\overline{F_{v}})$ (cf. the last paragraph of \S 4.1 and \cite[\S 2.5]{Dillery}).  

Note that since $\text{Res}[\mathscr{T}] \in \Hom_{F}(P_{\dot{V}}, Z)$ factors through $P_{E_{i}, S_{i}, n_{i}}$ for some $i$, for all $v \notin S_{i}$ we have that $\text{Res}[\text{loc}_{v}(\mathscr{T})]$ is trivial, and hence $\text{loc}_{v}(\mathscr{T})$ is the pullback of some $G$-torsor over $F_{v}$ via the projection $\gerbeE_{v} \xrightarrow{\pi} \text{Sch}/F_{v}$. The canonical inclusion $Z(O_{F_{v}^{\text{nr}}}) \to Z(\overline{F_{v}})$ is an equality for all but finitely many $v$ (because $Z$ is split over $F_{v}^{\text{nr}}$ for all but finitely many $v$, and $O_{F_{v}^{\text{nr}}}$ contains all roots of unity in $\overline{F_{v}}$). Choose an $O_{F,S}$-model $\mathcal{G}$ of $G$ for a some finite subset $S \subset V$; note that, for almost all $v$, the subgroups $\mathcal{G}(O_{F_{v}^{\text{nr}}})$ and $\mathcal{G}(O_{F_{v}^{\text{nr}}}^{\text{perf}})$ inside $G(F_{v}^{\text{sep}})$ and $G(\overline{F_{v}})$ (respectively) do not depend on the choice of model $\mathcal{G}$. Our goal in this subsection is to prove the following function field analogue of \cite[Prop. 6.1.1]{Taibi}, which will be needed for representation theory:

\begin{prop}\label{taibiprop} Let $\mathscr{T} \in Z^{1}(\gerbeE_{\dot{V}}, Z \to G)$. For all but finitely many $v \in V$, the torsor $\text{loc}_{v}(\mathscr{T}) \in Z^{1}(\gerbeE_{v}, Z \to G)/d(Z)$ is inflated from a $\mathcal{G}$-torsor $\mathcal{T}_{v}$ over $O_{F_{v}}$. Here, we are using $Z^{1}(\gerbeE_{v}, Z \to G)/d(Z)$ to denote equivalence classes of $G_{\gerbeE_{v}}$-torsors with the equivalence relation given by $\mathscr{T} \sim \eta^{*}\mathscr{T}$ for $\eta \colon \gerbeE_{v} \to \gerbeE_{v}$ induced by $d(a)$ for $a \in u_{v}(\overline{F_{v}})  \mapsto z \in Z(\overline{F_{v}})$ (we can always assume that $z \in Z(O_{F_{v}}^{\text{nr}})$, by the above discussion).  

Moreover, choosing normalizations $\gerbeE_{\dot{\xi}}$ and $\gerbeE_{\dot{\xi}_{v}}$ of the gerbes $\gerbeE_{\dot{V}}$ and $\gerbeE_{v}$ to view $\mathscr{T}$ as a torsor on $\gerbeE_{\dot{\xi}}$ (the choice of normalization and class $\dot{\xi}$ does not affect the class of $\mathscr{T}$ in  $Z^{1}(\gerbeE_{\dot{\xi}}, Z \to G)/d(Z)$), we may identify $\mathscr{T}$ with a $\dot{\xi}$-twisted $G$-torsor $(\mathcal{S'}, \text{Res}(\mathscr{T}), \psi')$ (cf. \cite[Def. 2.46]{Dillery}), where $\mathcal{S'}$ is a $G$-torsor over $\overline{F}$. Fix a $Z(\overline{F})$-orbit of trivializations $\mathcal{O} = \{\mathcal{S}' \xrightarrow{h'} \underline{G}\}$; then we may choose the $\mathcal{G}$-torsors $\mathcal{T}_{v}$ over $O_{F_{v}}$ such that for all but finitely many $v$, for any $h \in \mathcal{O}$, the trivializations $h_{\overline{F_{v}}}$ on $\mathcal{S}_{\overline{F_{v}}}$ are induced by the pullback of a trivialization $h_{v} \colon \mathcal{T}_{v} \to \underline{G}$ over the ring $O_{F_{v}^{\text{nr}}}^{\text{perf}}$. 
\end{prop}

\begin{proof} This proof follows \cite[\S 6.2]{Taibi} with some adjustments to accommodate the positive-characteristic situation. 
Pick a tower of resolutions by tori $(P_{k} \to T_{k} \to U_{k})_{k}$ as in Lemma \ref{toritower}, and set $T := \varprojlim_{k} T_{k}$, $U := \varprojlim_{k} U_{k}$, which are pro-tori.

By construction of the global canonical class $[\dot{\xi}]$, the image of $[\dot{\xi}]$ in $H^{2}(\overline{\A}/\A, T \to U)$ coincides with the image of the adelic canonical class $[x] \in \check{H}^{2}(\overline{\A}/\A, P)$, which, unpacking the construction of $[x]$, is to say (by the definition of the differentials arising from the double complex associated to $T \to U$) that there is some $a \in T(\overline{\A} \otimes_{\A} \overline{\A})$ and $b \in U(\overline{\A})$ such that
\begin{equation}\label{taibiequation}
\dot{\xi} = [\prod_{v \in V} \dot{S}_{v}^{2}(\text{loc}_{v}(\dot{\xi}_{v}))] \cdot d(a)
\end{equation}
inside $T(\overline{\A}^{\bigotimes_{\A}3})$ and $\overline{a} = db$ inside $U(\overline{\A} \otimes_{\A} \overline{\A})$, for a choice of Shapiro map $\dot{S}_{v}^{2}$. To make sense of the above product expression, we recall that $P(\overline{\A}^{\bigotimes_{\A}3}) = \varprojlim_{i} P_{i}(\overline{\A}^{\bigotimes_{\A}3})$, and for a fixed $i$, all but finitely-many projections $p_{i}[\dot{S}_{v}^{2}(\text{loc}_{v}(\dot{\xi}_{v}))]$ are trivial, and hence it makes sense to take this product in each $P_{i}(\overline{\A}^{\bigotimes_{\A}3}) = \varinjlim_{K/F} \prod_{v}' P_{i}(\A_{K,v}^{\bigotimes_{F_{v}}3})$ (by Corollary \ref{mainappendixBcor}) and then take the inverse limit.

Recall that $\dot{v} \in V_{F^{\text{sep}}}$ determines a ring homomorphism $\text{pr}_{\dot{v}} \colon \overline{\A} \to \overline{F_{v}}$ defined by 
the direct limit of the the projection maps $\A_{K} = \prod_{w \in V_{K}}' K_{w} \to K_{\dot{v}_{K}}$ over all finite extensions $K/F$, where by $\dot{v}_{K}$ we mean the unique extension of $\dot{v}_{K'}$, where $K'$ is the maximal Galois subextension of $K/F$, to a valuation on $K$.  Restricting this ring homomorphism to the subring $\overline{\A}_{v} \subset \overline{\A}$ gives a homomorphism of $F_{v}$-algebras. It is straightforward to check that we may choose our section $\Gamma/\Gamma_{\dot{v}} \to \Gamma$ (cf. the construction of the Shapiro maps in \S 2.2) such that, on $k$-cochains, we have $\restr{\text{pr}_{\dot{v}}}{\overline{\A}_{v}} \circ \dot{S}_{v}^{k} = \text{id}_{\overline{F_{v}}}$. We also have the projection map $\overline{\A} \xrightarrow{\text{pr}_{v}} \overline{\A}_{v}$ defined the same way except using the direct limit of the project maps $\prod_{w \in V_{K}}' K_{w} \to \prod_{w \mid v} K_{w}$.

Applying $\restr{\text{pr}_{\dot{v}}}{\overline{\A}_{v}} \circ \text{pr}_{v}$ to the equality \eqref{taibiequation}, we see that, for a fixed $v \in V$, the image of $\dot{\xi}$ in $T(\overline{F_{v}}^{\bigotimes_{F_{v}}3})$, denoted by $\text{res}_{v}(\dot{\xi})$ is given by $\text{loc}_{v}(\xi_{v}) \cdot d(a_{v})$, where $a_{v} := \text{pr}_{\dot{v}}(a) \in T(\overline{F_{v}} \otimes_{F_{v}} \overline{F_{v}})$. Although this equality is a priori taking place in $T(\overline{F_{v}}^{\bigotimes_{F_{v}}3})$, since the image of $\dot{\xi}$ and $\text{loc}_{v}(\xi_{v})$ both lie in the subgroup $P(\overline{F_{v}}^{\bigotimes_{F_{v}}3})$, we see that in fact $d(a_{v}) \in P(\overline{F_{v}}^{\bigotimes_{F_{v}}3})$ and thus this equality takes place in $P$. Set $b_{v} := \text{pr}_{\dot{v}}(b) \in U(\overline{F_{v}})$, and choose a lift $\tilde{b}_{v} \in T(\overline{F_{v}})$ of $b_{v}$, which is possible because $\varprojlim_{i}^{1} P_{i}(\overline{F_{v}})$ vanishes, since it consists of surjective maps. Define $a'_{v} := a_{v}/d(\tilde{b}_{v})$, which lies in $P(\overline{F_{v}}^{\bigotimes_{F_{v}}2})$ since its image under $T \to U$ equals $\text{pr}_{\dot{v}}(\overline{a})/\text{pr}_{\dot{v}}(db)$ (using that the isogenies $T_{k} \to U_{k}$ are defined over $F$, so they commute with \v{C}ech differentials), which is trivial by construction. We may replace $a_{v}$ by $a'_{v}$ and retain the equality \begin{equation}\label{Taibi1}
\text{res}_{v}(\dot{\xi}) = \text{loc}_{v}(\dot{\xi}_{v}) \cdot d(a'_{v}).
\end{equation}

For $k \geq 0$ and $v \in V$, we denote by $a_{v,k}$ (resp. $b_{v,k}$, $\tilde{b}_{v,k}$, $a'_{v,k}$) the image of $a_{v}$ (resp. $b_{v}$, $\tilde{b}_{v}$, $a'_{v}$) in $T_{k}(\overline{F_{v}}^{\bigotimes_{F_{v}}2})$ (resp. $U_{k}(\overline{F_{v}})$, $T_{k}(\overline{F_{v}})$, $P_{k}(\overline{F_{v}}^{\bigotimes_{F_{v}}2})$). We claim that there is a finite set of places $S'$ of $F$ such that for all $v \notin S$, the element $a'_{v,k}$ lies in the subgroup $P_{k}([O_{F_{v}^{\text{nr}}}^{\text{perf}}]^{\bigotimes_{O_{F_{v}}}2})$. Recall that $$a_{k} \in T_{k}(\overline{\A} \otimes_{\A} \overline{\A}) = \varinjlim_{E/F} T_{k}(\A_{E} \otimes_{\A} \A_{E}) = \varinjlim_{E/F} (\varinjlim_{S} T_{k}(\A_{E,S} \otimes_{\A_{S}} \A_{E,S})),$$ where the outside limit is over all finite extensions $E/F$ and the inside limit is over all finite sets of places of $F$. It follows that we may find $K/F$ finite containing $E_{k}$ and finite $S' \subset V$ containing $S_{k}$ such that the maximal Galois subextension $K'/F$ of $K$ is unramified outside $S'$, $T_{k}$ is split over $K'$, $a_{k} \in T_{k}(\A_{K,S'} \otimes_{\A_{S'}} \A_{K,S})$, and $b_{k} \in U_{k}(\A_{K,S'})$. Then for $v \notin S'$, we have $a_{k,v} \in T_{k}(O_{K_{\dot{v}}} \otimes_{O_{F_{v}}} O_{K_{\dot{v}}})$, and moreover, $K'_{\dot{v}}/F_{v}$ is unramified, so that $a_{k,v} \in T_{k}(O_{F_{v}^{\text{nr}}}^{\text{perf}} \otimes_{O_{F_{v}}} O_{F_{v}^{\text{nr}}}^{\text{perf}})$. 

The group $P_{k}$ is killed by the $n_{k}$-power map, and so there is a unique morphism $U_{k} \to T_{k}$ such that $U_{k} \to T_{k} \to U_{k}$ is the $n_{k}$-power map. Since $b_{k,v} \in U_{k}(O_{K_{\dot{v}}})$ for all $v \notin S'$ and $T_{k}$ and $U_{k}$ are split over $K$, any preimage of $b_{k,v}$ lies in $T_{k}([O_{K_{\dot{v}}}^{(n'_{k})}]^{(1/p^{m_{k}})})$, where $[O_{K_{\dot{v}}}^{(n'_{k})}]^{(1/p^{m_{k}})}$ denotes the fppf extension of $O_{K_{\dot{v}}}$ given by the composition of two extensions defined as follows: For $(n'_{k}, p) = 1$ and $n_{k} = n'_{k} \cdot p^{m_{k}}$, we first take the extension $O_{K_{\dot{v}}}^{(n'_{k})}/ O_{K_{\dot{v}}}$ obtained by adjoining all $n'_{k}$-roots of elements of $O_{K_{\dot{v}}}^{\times}$, which is finite \'{e}tale, followed by the extension $[O_{K_{\dot{v}}}^{(n'_{k})}]^{(1/p^{m_{k}})}$ defined by adjoining all $p^{m_{k}}$-power roots to $O_{K_{\dot{v}}}^{(n'_{k})}$, which is finite flat.

We claim that the extension $[O_{K_{\dot{v}}}^{(n'_{k})}]^{(1/p^{m_{k}})}/O_{F_{v}}$ lies in $O_{F_{v}^{\text{nr}}}^{\text{perf}}$. Indeed, since $O_{K_{\dot{v}}^{\text{nr}}}^{\text{perf}} = O_{F_{v}^{\text{nr}}}^{\text{perf}}$, it's enough to check that $[O_{K_{\dot{v}}}^{(n'_{k})}]^{(1/p^{m_{k}})}$ lies in $O_{K_{\dot{v}}^{\text{nr}}}^{\text{perf}}$, which is clear since it factors as a finite \'{e}tale extension of $O_{K_{\dot{v}}}$ followed by the extension obtained by adjoining all $p^{m_{k}}$-power roots. Thus, for any $v \notin S'$, we have $a'_{v,k} \in P_{k}([O_{K_{\dot{v}}}^{(n'_{k})}]^{(1/p^{m_{k}})} \otimes_{O_{F_{v}}} [O_{K_{\dot{v}}}^{(n'_{k})}]^{(1/p^{m_{k}})})$, and since the image of $\text{loc}_{v}(\xi_{v})$ is trivial in $P_{k}$ for all $v \notin S_{k} \subseteq S'$, we get the equality $$\text{res}_{v}(\dot{\xi}_{k}) = d(a'_{v,k}) \in P_{k}(\overline{F_{v}}^{\bigotimes_{F_{v}}3}),$$ where $\dot{\xi}_{k}$ denotes the image of $\dot{\xi}$ in $P_{k}(\overline{F_{v}}^{\bigotimes_{F_{v}}3})$.

 Let $\mathscr{T} \in Z^{1}(\gerbeE_{\dot{V}}, Z \to G)$, and choose normalizations of $\gerbeE_{v}$ and $\gerbeE_{\dot{V}}$ identifying them with $\gerbeE_{\dot{\xi}_{v}}$ and $\gerbeE_{\dot{\xi}}$, respectively. Recall that, after passing from $\gerbeE_{\dot{\xi}}$ to $\gerbeE_{\text{res}_{v}(\dot{\xi})} = (\gerbeE_{\dot{\xi}})_{\overline{F_{v}}}$, choosing different (global) normalizations has the effect of twisting $\text{loc}_{v}(\mathscr{T})$ by $d(z)$ for $z \in Z(O_{F_{v}^{\text{nr}}})$ with $z = \text{Res}([\mathscr{T}])(x)$ for some $x \in u_{v}(\overline{F_{v}})$, and thus does not affect the statement of the proposition. Changing the representatives $\dot{\xi}$ and $\dot{\xi}_{v}$ for the canonical classes has the same effect.
 
These normalizations let us canonically identify $G_{\gerbeE_{?}}$-torsors on the gerbes $\gerbeE_{?}$ with $?$-twisted $G$-torsors, for $? = \text{res}_{v}(\dot{\xi}), \dot{\xi}_{v}, \dot{\xi}$, by \cite[Prop. 2.50]{Dillery}; write $\mathscr{T} = (\mathcal{S}', \text{Res}(\mathscr{T}), \psi')$ under this identification. Choose $k$ sufficiently large so that: $\text{Res}(\mathscr{T}) \in \Hom_{F}(P, Z)$ factors through $P_{k}$ via $\varphi_{k} \in \Hom_{F}(P_{k}, Z)$, $\mathcal{S}'$ equals $j^{*}\mathcal{S}''$ for a $\mathcal{G}$-torsor $\mathcal{S}''$ over $O_{S_{k}}^{\text{perf}}$ for $\Spec(\overline{F}) \xrightarrow{j} \Spec(O_{S_{k}}^{\text{perf}})$, $h$ equals $j^{*}h_{S_{k}}$ for an $O_{S_{k}}^{\text{perf}}$-trivialization $h_{S_{k}}$ of $\mathcal{S}''$, and the ``twisted gluing isomorphism" $\psi' \colon p_{2}^{*}\mathcal{S}' \to p_{1}^{*}\mathcal{S}'$ is given by $j^{*}\psi$ for an isomorphism of $\mathcal{G}$-torsors $p_{2}^{*}\mathcal{S}'' \xrightarrow{\psi} p_{1}^{*}\mathcal{S}''.$

Take $S' \supseteq S_{k}$ corresponding to $k$ as in the above paragraphs. Construction \ref{changeofgerbe} and the equality \eqref{Taibi1} gives a morphism $\gerbeE_{\dot{\xi}_{v}} \to \gerbeE_{\text{res}_{v}(\dot{\xi})}$ which via pullback sends the $\text{res}_{v}(\dot{\xi})$-twisted $G$-torsor $(\mathcal{S}'_{\overline{F_{v}}}, \text{Res}(\mathscr{T}), \psi')$ to the $\dot{\xi}_{v}$-twisted $G$-torsor $(\mathcal{S}'_{\overline{F_{v}}}, \text{Res}(\mathscr{T}) \circ \text{loc}_{v}, m_{a_{v}'} \circ \psi')$.  By construction, for any $v \notin S'$ the homomorphism $\text{Res}(\mathscr{T}) \circ \text{loc}_{v}$ is trivial on $u_{v}$, and hence $(\mathcal{S}'_{\overline{F_{v}}}, m_{a_{v}'} \circ \psi')$ gives a descent datum for a $G$-torsor $\mathcal{S}_{v}$ over $F_{v}$; we claim that the pair of $\mathcal{S}_{v}$ and the $\overline{F_{v}}$-trivialization induced by $h_{\overline{F_{v}}}$ descends further to a $\mathcal{G}$-torsor $\mathcal{T}_{v}$ over $O_{F_{v}}$ with an $O_{F_{v}^{\text{nr}}}^{\text{perf}}$-trivialization. 

We define $\mathcal{T}_{v}$ via the descent datum $(\mathcal{S}''_{O_{F_{v}^{\text{nr}}}^{\text{perf}}}, m_{\text{Res}(\mathscr{T})(a'_{v})} \circ \psi)$  with respect to the fpqc cover $O_{F_{v}^{\text{nr}}}^{\text{perf}}/O_{F_{v}}$. The torsor $\mathcal{S}''_{O_{F_{v}^{\text{nr}}}^{\text{perf}}}$ is well-defined because $v \notin S'$ and $O_{F_{v}^{\text{nr}}}^{\text{perf}}$ is an $O_{S_{k}}^{\text{perf}}$-algebra, and $m_{\text{Res}(\mathscr{T})(a'_{v})} \circ \psi$ makes sense, since $\text{Res}(\mathscr{T})(a'_{v}) = \varphi_{k}(a'_{v,k})$, the morphism $\varphi_{k}$ is defined over $O_{F,S'}$, and $a'_{v,k} \in P_{k}(O_{F_{v}^{\text{nr}}}^{\text{perf}} \otimes_{O_{F_{v}}} O_{F_{v}^{\text{nr}}}^{\text{perf}})$; this finishes the construction of $\mathcal{T}_{v}$---by design, $h_{v}:= (h_{S_{k}})_{O_{F_{v}^{\text{nr}}}^{\text{perf}}}$ trivializes $\mathcal{T}_{v}$ over $O_{F_{v}^{\text{nr}}}^{\text{perf}}$. The pullback of $\mathcal{T}_{v}$ is evidently equal to $\mathcal{S}_{v}$, since the descent datum giving $\mathcal{T}_{v}$ pulls back via the morphisms $\Spec(F_{v}) \to \Spec(O_{F_{v}})$ and $\Spec(\overline{F_{v}}) \to \Spec(O_{F_{v}^{\text{nr}}}^{\text{perf}})$ to the descent datum giving $\mathcal{S}_{v}$; similarly, $h_{v}$ pulls back to $h_{\overline{F_{v}}}$. This proves the result. 
\end{proof}

\section{Applications to endoscopy}
In this section, we use the above constructions to analyze an adelic transfer factor for a global function field $F$. In what follows, $G$ will be a connected reductive group over $F$.

\subsection{Adelic transfer factors for function fields}
We follow \cite[\S 6.3]{LS} to construct adelic transfer factors for connected reductive groups over a global function field $F$.  Let $\psi \colon G_{F_{s}} \to G^{*}_{F_{s}}$ be a quasi-split inner form of $G$, with Langlands dual group $\widehat{G}^{*}$ and Weil-form $\prescript{L}{}G^{*}:= \widehat{G}^{*} \rtimes W_{F}$.  

\begin{Def} A \textit{global endoscopic datum} for $G$ is a tuple $(H, \mathcal{H}, s, \xi)$ where $H$ is a quasi-split connected reductive group over $F$, $\mathcal{H}$ is a split extension of $W_{F}$ by $\widehat{H}$, $s \in Z(\widehat{H})$ is any element, and $\xi \colon \mathcal{H} \to \prescript{L}{}G^{*}$ is an $L$-embedding such that:
\begin{enumerate}
\item{The homomorphism $W_{F} \to \text{Out}(\widehat{H}) = \text{Out}(H)$ determined by $\mathcal{H}$ is the same as the homomorphism $W_{F} \to \Gamma \to \text{Out}(H)$ induced by the usual $\Gamma$-action on $H$.}
\item{The map $\xi$ restricts to an isomorphism of algebraic groups over $\mathbb{C}$ from $\widehat{H}$ to $Z_{\widehat{G}^{*}}(t)^{\circ}$, where $t := \xi(s)$.}
\item{The first two conditions imply that we have a $\Gamma$-equivariant embedding $Z(\widehat{G}^{*}) \to Z(\widehat{H})$. We require that the image of $s$ in $Z(\widehat{H})/Z(\widehat{G}^{*})$, denoted by $\bar{s}$, is fixed by $W_{F}$ and maps under the connecting homomorphism $H^{0}(W_{F}, Z(\widehat{H})/Z(\widehat{G}^{*})) \to H^{1}(W_{F}, Z(\widehat{G}^{*}))$ to an element which is killed by the homomorphism $H^{1}(W_{F}, Z(\widehat{G}^{*})) \to H^{1}(W_{F_{v}}, Z(\widehat{G}^{*}))$ for all $v \in V_{F}$ (such an an element is called \textit{locally trivial}).}
\end{enumerate}
\end{Def}

Such a datum $\mathfrak{e}$ induces, for any $v \in V$, a local endoscopic datum $\mathfrak{e}_{v} := (H_{F_{v}}, \mathcal{H}_{v}, s_{v}, \xi_{v})$, where $\mathcal{H}_{v}$ is the fibered product of $\mathcal{H} \to W_{F}$ and $W_{F_{v}} \to W_{F}$, $\xi_{v} \colon \mathcal{H}_{v} \to \prescript{L}{}(G^{*}_{F_{v}})$ is induced by $\xi$ and the natural map $\mathcal{H}_{v} \to \mathcal{H}$, and $s_{v} = s \in Z(\widehat{H})$. Fix such an $\mathfrak{e}$; we will temporarily assume that $\mathcal{H} = \prescript{L}{}H$. Up to equivalence, $\mathfrak{e}$ only depends on the image of $s$ in $\pi_{0}([Z(\widehat{H})/Z(\widehat{G}^{*})]^{\Gamma})$. Recall that a strongly-regular semisimple element $\gamma_{H} \in H(F)$ with centralizer $T_{H}$ (a maximal torus of $H$ defined over $F$) is called \textit{$G$-regular} if it is the preimage of a strongly-regular semisimple element $\gamma_{G} \in G(F)$ under an admissible isomorphism $T_{H} \to T_{G} := Z_{G}(\gamma_{G})$. First, we need:

\begin{lem} There is an admissible embedding of $T_{G}$ into $G^{*}$. 
\end{lem}

\begin{proof} This follows from \cite[Corollary 6.3]{Dillery} (the proof loc. cit. works for our field $F$). 
\end{proof}

It follows that for any $G$-regular strongly-regular semisimple $\gamma_{H} \in H(F)$, we have an admissible embedding of $T_{H}$ in $G^{*}$ (which is not unique). We say that $\gamma_{H}$ is a \textit{related to} $\gamma_{G} \in G(\A)$ if for all $v \in V$, the image $\gamma_{H,v}$ of $\gamma_{H}$ in $H(F_{v})$ is an image (under an admissible embedding $(T_{H})_{F_{v}} \to G_{F_{v}}$) of the element $\gamma_{G,v} \in G(F_{v})$. If we fix an admissible embedding of $T_{H}$ in $G^{*}$, with image a maximal $F$-torus denoted by $T$ and image of $\gamma_{H}$ denoted by $\gamma \in G^{*}(F)$, then the above condition means requiring that there exist $x_{v} \in G^{*}(F_{v}^{\text{sep}})$ such that $\text{Ad}(x_{v}) \circ \psi$ maps the maximal torus $T_{G,v}$ in $G_{F_{v}}$ containing $\gamma_{G,v}$ to $T_{F_{v}}$ (over $F_{v}$) and sends $\gamma_{G,v}$ to (the restriction of) $\gamma$. 

\begin{prop}\label{LS6.4}(\cite[Thm. 6.4.A]{LS}) \begin{enumerate}
\item{For almost all $v$, the values $\Delta_{i}(\gamma_{H,v}, \gamma_{G,v})$ (from \cite[\S 6.3]{Dillery}) equal $1$ for $i=I,II,III_{2}, IV$.}
\item{$\prod_{v} \Delta_{i}(\gamma_{H,v}, \gamma_{G,v})  = 1$ for $i=I,II,III_{2}, IV$.}
\end{enumerate}
\end{prop}

\begin{proof} We closely follow the analogous proof in \cite{LS}. As in \cite[\S 6.2.1]{Dillery}, we may define, for the quasi-split simply-connected reductive group $G^{*}_{\text{sc}}$ with maximal torus $T_{\text{sc}}$, a \textit{global splitting invariant} $\lambda_{\{a_{\alpha}\}}(T_{\text{sc}}) \in H^{1}(F, T_{\text{sc}})$ which depends on an $F$-pinning of $G^{*}_{\text{sc}}$ and a choice of $a$-data $\{a_{\alpha}\}$ for $T$ (see \cite[\S 6.2.1]{Dillery}). By the construction of the local splitting invariant, it is clear that $\lambda_{\{a_{\alpha}\}}(T_{\text{sc}})$ maps to the local splitting invariant $\lambda_{\{a_{\alpha}\}}(T_{F_{v},\text{sc}})$ (where we are viewing the $a$-data $\{a_{\alpha}\}$ as an $a$-data for $T_{F_{v}}$) under the canonical map $H^{1}(F, T_{\text{sc}}) \to H^{1}(F_{v}, T_{F_{v}, \text{sc}})$. Since for all but finitely many $v$ the image of $\lambda_{\{a_{\alpha}\}}(T_{\text{sc}})$ lands in the subgroup $H^{1}(O_{F_{v}}, T_{F_{v},\text{sc}}) = 0$, it follows that $\langle \lambda_{\{a_{\alpha}\}}(T_{F_{v},\text{sc}}), \textbf{s}_{T,v} \rangle = \Delta_{1}(\text{res}_{v}(\gamma_{H}), \gamma_{G,v}) = 1$ for all but finitely many $v$. 

Our above observation and the exact sequence $$H^{1}(F, T_{\text{sc}}) \to H^{1}(\A, T_{\text{sc}}) = \check{H}^{1}(\overline{\A}/\A, T_{\text{sc}}) \to \bar{H}^{1}(\overline{\A}/\A, T_{\text{sc}})$$ (see \cite[\S D.1]{KS1}) imply that the image $\bar{\lambda}$ of the element $(\langle \lambda_{\{a_{\alpha}\}}(T_{F_{v},\text{sc}}), \textbf{s}_{T,v} \rangle)_{v} \in H^{1}(\A, T_{\text{sc}}) = \bigoplus_{v} H^{1}(F_{v}, T_{F_{v}, \text{sc}})$ is trivial in $\bar{H}^{1}(\overline{\A}/\A, T_{\text{sc}})$, and so it follows by local-global compatibility of the Tate-Nakayama pairing that $$\prod_{v} \langle \lambda_{\{a_{\alpha}\}}(T_{F_{v},\text{sc}}), \textbf{s}_{T,v} \rangle = \langle \bar{\lambda}, \textbf{s}_{T} \rangle = 1,$$ as desired for the case $i=I$. The arguments for the remaining cases of $i=II$, $III_{2}$, and $IV$ may be taken verbatim from the proof of \cite[Thm. 6.4.A]{LS}.
\end{proof}


Note that if $\gamma_{H} \in H(F)$ is a strongly $G$-regular semisimple element which is related to $\gamma_{G} \in G(\A)$, then for all $v \in V$ its image in $H(F_{v})$, is strongly $G_{F_{v}}$-regular and is related to the element $\gamma_{G,v} \in G(F_{v})$. We can now define the adelic transfer factor. Call an element $\gamma \in H(\A)$ \textit{semisimple}  if $\gamma_{v} \in H(F_{v})$ is semisimple for all $v$, and \textit{strongly $G$-regular} if $\gamma \in H_{G-\text{sr}}(\A_{\overline{F}})$, where $H_{G-\text{sr}} \subset H_{\overline{F}}$ is the $\overline{F}$-scheme characterized by the Zariski open subset of strongly $G$-regular semisimple elements of the variety $H(\overline{F})$.  Similarly, we call a semisimple element $\delta \in G(\A)$ \textit{strongly regular} if it lies in $G_{\text{sr}}(\A_{\overline{F}})$, where $G_{\text{sr}} \subset G_{\overline{F}}$ is the Zariski open subscheme characterized by the strongly regular elements of $G(\overline{F})$.

\begin{Def}\label{adelictf}
For $\gamma \in H_{G-\text{sr}}(\A)$ and $\delta \in G_{\text{sr}}(\A)$, we set $\Delta_{\A}(\gamma, \delta) = 0$ if there is no strongly $G$-regular element of $H(F)$ which is related to an element of $G(F)$, and otherwise fix such a pair $\bar{\gamma}_{H}, \bar{\gamma}_{G}$ and define (using \cite[(16)]{Dillery} with the local endoscopic datum $\mathfrak{e}_{v}$ for each $v$)
\begin{equation}\label{adelictransferfactor2}
\Delta_{\A}(\gamma, \delta) := \prod_{v} \Delta(\gamma_{v}, \delta_{v}, \bar{\gamma}_{H,v}, \bar{\gamma}_{G,v}).
\end{equation}
\end{Def}
This product is well-defined due to the following result:

\begin{lem} In the notation of the above definition, the local transfer factor $\Delta(\gamma_{v}, \delta_{v}, \bar{\gamma}_{H,v}, \bar{\gamma}_{G,v})$ equals one for all but finitely many $v$.
\end{lem}

\begin{proof} For all but finitely many $v$, the group $G_{F_{v}}$ is quasi-split, in which case we may write $$\Delta(\gamma_{v}, \delta_{v}, \bar{\gamma}_{H,v}, \bar{\gamma}_{G,v}) = \frac{\Delta(\gamma_{v}, \delta_{v})}{\Delta(\bar{\gamma}_{H,v}, \bar{\gamma}_{G,v})}.$$ For a quasi-split connected reductive group over a local field, the (absolute) local transfer factor may be defined purely using Galois cohomology (cf. \cite[\S 6.3, \S 6.2.1]{Dillery}), where the claim of the Lemma follows from its characteristic-zero analogue, which is stated in \cite[\S 7.3, p. 109]{KS1} (cf. also Proposition \ref{LS6.4} above).
\end{proof}


The independence of $\Delta_{\A}$ of our choice of $\bar{\gamma}_{H}, \bar{\gamma}_{G}$ will follow from Proposition \ref{locglobtransfer}.

\begin{remark} In the case that $\mathcal{H} \neq \prescript{L}{}H$ in our global endoscopic datum, the formula for $\Delta_{\A}$ is slightly more complicated. To begin, we fix a $z$-pair $(H_{1}, \xi_{H_{1}})$ for the endoscopic datum $\mathfrak{e} = (H, \mathcal{H}, s, \xi)$, which always exist over fields of arbitrary characteristic. For any place $v$ of $F$, this $z$-pair gives rise to a $z$-pair $(H_{1,v}, \xi_{H_{1},v})$ for the local endoscopic datum $\mathfrak{e}_{v}$. We may then define the adelic transfer factor for pairs of elements $\gamma_{1} \in H_{1, G-\text{sr}}(\A)$ and $\delta_{v} \in G_{\text{sr}}(\A)$, where $\gamma_{1} \in H_{1, G-\text{sr}}(\A)$ means that its image in $H(\A)$ is $G$-strongly regular, using the relative local transfer factors for $z$-pairs as in \cite[\S 6.4]{Dillery}:
$\Delta_{\A}(\gamma_{1}, \delta) := \prod_{v} \Delta(\gamma_{1,v}, \delta_{v}, \bar{\gamma}_{H,v},\bar{\gamma}_{G,v})$.
\end{remark}

\subsection{Endoscopic setup}
This subsection is an analogue of \cite[\S 4.2, \S 4.3]{Tasho2}, which explain how to pass from global to local refined endoscopic data and discuss coherent families of local rigid inner twists; recall the notion of a refined endoscopic datum $(H_{F_{v}}, \mathcal{H}_{v}, \dot{s}_{v}, \xi_{v})$  over local $F_{v}$, defined in \cite[\S 7.2]{Dillery}. A fixed global endoscopic datum $\mathfrak{e} = (H, \mathcal{H}, s, \xi)$ induces a canonical embedding $Z(G) \to Z(H)$, and we set $\bar{H} := H/Z_{\text{der}}$, where $Z_{\text{der}} := Z(\mathscr{D}(G^{*}))$, $Z_{\text{sc}} := Z(G^{*}_{\text{sc}})$, and $\bar{G}^{*}:= G^{*}/Z_{\text{der}}$. Note that $\bar{G}^{*} = G^{*}_{\text{ad}} \times Z(G^{*})/Z_{\text{der}}$, $\widehat{\bar{G}^{*}} = \widehat{G^{*}}_{\text{sc}} \times Z(\widehat{G^{*}})^{\circ}$; we set $Z := Z(G)$.

The $L$-embedding $\xi$ induces an embedding $\widehat{\bar{H}} \to \widehat{\bar{G}^{*}}$ with image equal to $Z_{\widehat{\bar{G}^{*}}}(t)^{\circ}$, where recall that $t := \xi(s)$ (this is well-defined because $\widehat{\bar{G}^{*}}$ maps to $\widehat{G^{*}}$, which contains $t$). Then for $s_{\text{sc}} \in \widehat{G^{*}}_{\text{sc}}$ a fixed preimage of the image $s_{\text{ad}}$ of $s$ in $\widehat{G^{*}}_{\text{ad}}$ and a place $v \in V$, by definition of a global endoscopic datum we may find an element $y_{v} \in Z(\widehat{G^{*}})$ such that $s_{\text{der}} \cdot y_{v} \in Z(\widehat{H})^{\Gamma_{v}}$, where $s_{\text{der}} \in \mathscr{D}(\widehat{G^{*}})$ denotes the image of $s_{\text{sc}}$. We can write $y_{v} = y_{v}' \cdot y_{v}''$ for $y_{v}' \in Z( \mathscr{D}(\widehat{G^{*}}))$, $y_{v}'' \in Z(\widehat{G^{*}})^{\circ}$, and we choose a lift $\dot{y}_{v}' \in \widehat{Z}_{\text{sc}}$ of $y_{v}'$. The element $(s_{\text{sc}} \cdot \dot{y}_{v}', y''_{v}) =: \dot{s}_{v}$ lies in $\widehat{\bar{G}^{*}} = \widehat{G^{*}}_{\text{sc}} \times Z(\widehat{G^{*}})^{\circ}$, which, via the above $L$-embedding, belongs to the group $Z(\widehat{\bar{H}})^{+v}$, and $\dot{\mathfrak{e}}_{v} := (H_{F_{v}}, \mathcal{H}_{v}, \dot{s}_{v}, \xi_{v})$ defines a local refined endoscopic datum. 

We now discuss coherent families of local rigid inner twists. For an equivalence class $\Psi$ of inner twists $G^{*}_{F^{\text{sep}}} \to G_{F^{\text{sep}}}$, 
The class $\Psi$ gives an element of $H^{1}(F, G^{*}_{\text{ad}})$ which by Lemma \ref{tasho3.6.1} has a preimage in the set $H^{1}(\gerbeE_{\dot{V}}, Z_{\text{sc}} \to G^{*}_{\text{sc}})$. It follows that for every $\psi \in \Psi$, we can find a $Z_{\text{sc}}$-twisted $G^{*}_{\text{sc},\gerbeE_{\dot{V}}}$-torsor $\mathscr{T}_{\text{sc}}$ along with an isomorphism of $(G^{*}_{\text{ad}})_{\gerbeE_{\dot{V}}}$-torsors $\bar{h} \colon (\overline{\mathscr{T}_{\text{sc}}})_{\overline{F}} \xrightarrow{\sim} (\underline{(G^{*}_{\text{ad}})_{\gerbeE_{\dot{V}}}})_{\overline{F}}$, where $\overline{\mathscr{T}_{\text{sc}}} := \mathscr{T}_{\text{sc}} \times^{G^{*}_{\text{sc},\gerbeE_{\dot{V}}}} (G^{*}_{\text{ad}})_{\gerbeE_{\dot{V}}}$ and $\underline{(G^{*}_{\text{ad}})_{\gerbeE_{\dot{V}}}}$ denotes the trivial $(G^{*}_{\text{ad}})_{\gerbeE_{\dot{V}}}$-torsor, such that $p_{1}^{*}\bar{h} \circ p_{2}^{*}\bar{h}^{-1}$ is translation by $\bar{x} \in G^{*}_{\text{ad}}(\overline{F} \otimes_{F} \overline{F})$ which satisfies $\text{Ad}(\bar{x}) = p_{1}^{*}\psi^{-1} \circ p_{2}^{*}\psi$. 

For each $v \in V$, we set $\mathscr{T}_{v}$ to be the $Z$-twisted $G^{*}_{\gerbeE_{v}}$-torsor given by $\text{loc}_{v}(\mathscr{T})$, where $\mathscr{T} := \mathscr{T}_{\text{sc}} \times^{G^{*}_{\text{sc},\gerbeE_{\dot{V}}}} G^{*}_{\gerbeE_{\dot{V}}}$, and $\text{loc}_{v}$ is as in \S 4.4; the $\overline{F}$-trivialization $\bar{h}$ evidently induces a $\overline{F_{v}}$-trivialization of $\overline{\mathscr{T}_{v}}$ (noting that $\overline{\mathscr{T}_{\text{sc}}} = \overline{\mathscr{T}}$), denoted by $\bar{h}_{v}$.  Note that, by construction, the triple $(\psi_{v}, \mathscr{T}_{v}, \bar{h}_{v})$ ($\psi_{v} := \psi_{\overline{F_{v}}}$) is a rigid inner twist over $F_{v}$ (\cite[Def. 7.1]{Dillery}); we thus get a collection $(\psi_{v}, \mathscr{T}_{v}, \bar{h}_{v})_{v}$ of local rigid inner twists which depends on the choice of $\text{loc}_{v}$, but only up to twisting torsors by $d(z)$ for an element $z \in Z_{\text{sc}}(\overline{F_{v}})$, which does not affect any associated cohomology sets. However, this family will in general depend on the choice of torsor $\mathscr{T}_{\text{sc}}$. We observe that, since $\mathscr{T}$ is induced by the $Z_{\text{sc}}$-twisted $G^{*}_{\text{sc}}$-torsor $\mathscr{T}_{\text{sc}}$, we have $\mathscr{T}_{v} \in Z^{1}(\gerbeE_{v}, Z_{\text{der}} \to G^{*}) \subset  Z^{1}(\gerbeE_{v}, Z \to G^{*}) $.

\subsection{Product decomposition of the adelic transfer factor}
We continue with the notation of \S 5.2. We now show that $\Delta_{\A}$ from Definition \ref{adelictf} admits a decomposition in terms of the normalized local transfer factors constructed in \cite[\S 7.2]{Dillery}, following \cite[\S 4.4]{Tasho2}.  Fix an equivalence class $\Psi$ of inner twists $G^{*}_{F^{\text{sep}}} \to G_{F^{\text{sep}}}$, endoscopic datum $\mathfrak{e} = (H, \mathcal{H}, s, \xi)$ for $G^{*}$, and a $z$-pair $\mathfrak{z} = (H_{1}, \xi_{1})$ for $\mathfrak{e}$. We assume that there exist strongly $G$-regular $\gamma_{1,0} \in H_{1}(F)$ and $\delta_{0} \in G(F)$ such that $\gamma_{1,0}$ is related to $\delta_{0}$ (so that, in particular, the image of $\gamma_{1,0}$ in $H(F)$, denoted by $\gamma_{0}$, is related to $\delta_{0}$). We can associate to $\mathfrak{e}$ the collection of refined local endoscopic data $(\dot{\mathfrak{e}}_{v})_{v \in V}$ and to $\Psi$ a coherent family of local rigid inner twists $(\psi_{v}, \mathscr{T}_{v}, \bar{h}_{v})_{v \in V}$ as explained in \S 5.2, to the global $z$-pair $\mathfrak{z}$ a collection of local $z$-pairs $(\mathfrak{z}_{v})_{v \in V}$, and to a fixed global Whittaker datum $\mathfrak{w}$ for $G^{*}$, a collection of local Whittaker data $(\mathfrak{w}_{v})_{v \in V}$. 

For any $v$, we can use the local Whittaker datum and $z$-pair to obtain from \cite[\S 7.2]{Dillery}, the $\mathfrak{w}_{v}$-normalized local transfer factor $$\Delta[\mathfrak{w}_{v}, \dot{\mathfrak{e}}_{v}, \mathfrak{z}_{v}, \psi_{v}, (\mathscr{T}_{v}, \bar{h}_{v})] \colon H_{1,G-\text{sr}}(F_{v}) \times G_{\text{sr}}(F_{v}) \to \mathbb{C}.$$


\begin{prop}\label{locglobtransfer} For any $\gamma_{1} \in H_{1, G-{\text{sr}}}(\A)$ and $\delta \in G_{\text{sr}}(\A)$, we have 
$$\Delta_{\A}(\gamma_{1}, \delta) = \prod_{v \in V} \langle \text{loc}_{v}(\mathscr{T}_{\text{sc}}), \dot{y}_{v}' \rangle \cdot \Delta[\mathfrak{w}_{v}, \dot{\mathfrak{e}}_{v}, \mathfrak{z}_{v}, \psi_{v}, (\mathscr{T}_{v}, \bar{h}_{v})](\gamma_{1,v}, \delta_{v}).$$
In the above formula, $\dot{y}_{v}' \in \widehat{Z}_{\text{sc}}$ as in \S 5.2 and the pairing $\langle -, - \rangle \colon H^{1}(\gerbeE_{v}, Z_{\text{sc}} \to G^{*}_{\text{sc}}) \times \widehat{Z}_{\text{sc}} \to \mathbb{C}$ is from \cite[Corollary 7.11]{Dillery}, which is well-defined since $\dot{y}_{v}' \in \widehat{Z}_{\text{sc}} = Z(\widehat{G^{*}_{\text{sc}}/Z_{\text{sc}})}^{+}.$ For almost all $v \in V$, the corresponding factor in the product equals $1$. For all $v$, the corresponding factor is independent of the choices of $\dot{y}_{v}'$ and $y_{v}''$ made in \S 5.2.
\end{prop}

\begin{proof} The argument closely follows \cite[Prop. 4.4.1]{Tasho2}; as in the proof of the result loc. cit.,  \cite[Corollary 6.4.B]{LS} gives the above product identity if we can show that the normalized factors $\langle \text{loc}_{v}(\mathscr{T}_{\text{sc}}), \dot{y}_{v}' \rangle \cdot \Delta[\mathfrak{w}_{v}, \dot{\mathfrak{e}}_{v}, \mathfrak{z}_{v}, \psi_{v}, (\mathscr{T}_{v}, \bar{h}_{v})](\gamma_{1,v}, \delta_{v})$ satisfy the following properties: First, that they are absolute transfer factors, and second, that their values at the $F$-rational pair $(\gamma_{1,0,v}, \delta_{0,v})$ equal $1$ for all but finitely many $v$ and have a product over all $v$ that equals $1$. The first property automatically holds for the above factors by \cite[Prop. 7.12]{Dillery} (the extra $\langle \text{loc}_{v}(\mathscr{T}_{\text{sc}}), \dot{y}_{v}' \rangle$-factor cancels out and thus makes no difference for this verification).

Proposition \ref{LS6.4} and the argument in the proof of \cite[Prop. 4.4.1]{Tasho2} regarding local and global $\epsilon$-factors (which works the same in our situation) reduces the second property above to proving its analogue for the terms
\begin{equation}\label{twopairings}
\langle \text{loc}_{v}(\mathscr{T}_{\text{sc}}), \dot{y}_{v}' \rangle ^{-1} \langle \text{inv}((G_{F_{v}}, \psi_{v}, (\mathscr{T}_{v}, \bar{h}_{v}), \delta_{0,v}), \delta_{0,v}^{*}), \dot{s}_{v,\gamma_{0}, \delta^{*}_{0}} \rangle ,
\end{equation}
where $\delta_{0}^{*} \in G^{*}(F)$ is the image of $\gamma_{0}$ under a choice of admissible embedding of $T_{0,H}$ into $G^{*}$, the map $\text{inv}(-,\delta^{*}_{0,v}) \colon C_{Z_{\text{der}}}(\delta_{0,v}^{*}) \to H^{1}(\gerbeE_{v}, Z_{\text{der}} \to T_{0})$ with $T_{0}:= Z_{G^{*}}(\delta_{0}^{*})$ is from \cite[\S 7.1]{Dillery}, the element $\dot{s}_{v,\gamma_{0}, \delta^{*}_{0}} \in \pi_{0}(\widehat{\bar{T_{0}}}^{+,v})$ is the image of $\dot{s}_{v} \in \pi_{0}(Z(\widehat{\bar{H}})^{+,v})$ under the composition $\hat{\varphi} \colon Z(\widehat{\bar{H}}) \to \widehat{\overline{T_{0,H}}} \to \widehat{\bar{T}_{0}}$ (as in \S 5.2, the bar indicates that we are quotienting out by $Z_{\text{der}}$) induced by our choice of admissible embedding of $T_{0,H}$ into $G^{*}$, and the right-hand pairing is from \cite[Cor. 7.11]{Dillery}. 

In order to work explicitly with $\text{inv}(-,-)$ at each $v$, we fix explicit \v{C}ech $2$-cocycles $\dot{\xi}_{v}$ representing each $\xi_{v} \in \check{H}^{2}(\overline{F_{v}}/F_{v}, u_{v})$ and replace the notion of $Z_{\text{der}}$-twisted torsors on $\gerbeE_{v}$ with $\dot{\xi}_{v}$-twisted $Z_{\text{der}}$-cocycles; we know by \S 7.2 loc. cit. that $\text{inv}(-,-)$ and corresponding local transfer factor do not depend on such a choice, and hence we may do so without loss of generality.

By construction, the elements $\delta_{0}^{*}$ and $\delta_{0}$ are stably conjugate, so that there exists $g \in G^{*}(\overline{F})$ such that $\psi(g\delta_{0}^{*}g^{-1}) = \delta_{0}$, and then $\text{inv}((G_{F_{v}}, \psi_{v}, (\mathscr{T}_{v}, \bar{h}_{v}), \delta_{0,v}), \delta_{0,v}^{*}) \in H^{1}(\gerbeE_{v}, Z_{\text{der}} \to T_{0}) = H^{1}(\gerbeE_{\dot{\xi}_{v}}, Z_{\text{der}} \to T_{0})$ is represented by the $\dot{\xi}_{v}$-twisted cocycle $x_{v}:= (p_{1}(g)^{-1}z_{v}p_{2}(g), \phi_{v}),$ where $(z_{v}, \phi_{v})$ is a choice of $\dot{\xi}_{v}$-twisted cocycle corresponding to the $Z_{\text{der}}$-twisted $G^{*}_{\gerbeE_{v}}$-torsor $\mathscr{T}_{v}$ (cf. \cite[Prop. 7.3]{Dillery}). We may choose $g$ so that it is the image of some $g_{\text{sc}} \in G^{*}_{\text{sc}}(\overline{F})$, and then we may lift the twisted cocycle $x_{v}$ to the $\dot{\xi}_{v}$-twisted cocycle $x_{v,\text{sc}} := (p_{1}(g_{\text{sc}})^{-1}z_{\text{sc},v}p_{2}(g_{\text{sc}}), \phi_{\text{sc},v})$, where $(z_{\text{sc},v}, \phi_{\text{sc},v}) \in Z^{1}(\gerbeE_{\dot{\xi}_{v}}, Z_{\text{sc}} \to G^{*}_{\text{sc}})$ is a choice of twisted cocycle corresponding to the $Z_{\text{sc}}$-twisted $G^{*}_{\text{sc}}$-torsor $\text{loc}_{v}(\mathscr{T}_{\text{sc}})$ on $\gerbeE_{v}$.

Using the decomposition $\widehat{\bar{T}_{0}} = (\widehat{T_{0}})_{\text{sc}} \times Z(\widehat{G^{*}})^{\circ}$, we may use the notation of \S 5.2 to write $\dot{s}_{v,\gamma_{0},\delta_{0}^{*}} = (\dot{y}'_{v} \hat{\varphi}(s_{\text{sc}}), y''_{v})$. The functoriality of the pairing from \cite[Corollary 7.11]{Dillery} with respect to the morphism $[Z_{\text{sc}} \to T_{0,\text{sc}}] \to [Z_{\text{der}} \to T_{0}]$, then implies that 
$$ \langle \text{inv}((G_{F_{v}}, \psi, (\mathscr{T}_{v}, \bar{h}_{v}), \delta_{0,v}), \delta_{0,v}^{*}), \dot{s}_{v,\gamma_{0}, \delta^{*}_{0}} \rangle  =
\langle x_{v,\text{sc}}, \dot{y}'_{v} \hat{\varphi}(s_{\text{sc}}) \rangle .$$
By construction, the restriction of the character $\langle x_{v,\text{sc}}, - \rangle$ on $\pi_{0}(\widehat{\overline{T_{0,\text{sc}}}}^{+v})$ to $Z(\widehat{\overline{{G}^{*}_{\text{sc}}}})^{+v}$ equals the character $\langle (z_{\text{sc,}v}, \phi_{\text{sc},v}), - \rangle$ by the functoriality of the pairing with respect to the morphism $[Z_{\text{sc}} \to T_{0,\text{sc}}] \to [Z_{\text{sc}} \to G^{*}_{\text{sc}}]$. It then follows by bilinearity that the expression \eqref{twopairings} reduces to 
\begin{equation}\label{onepairing}
\langle x_{v,\text{sc}}, \hat{\varphi}(s_{\text{sc}}) \rangle.
\end{equation}

We have already fixed normalizations $\gerbeE_{\dot{\xi}_{v}}$ of the gerbes $\gerbeE_{v}$ for all $v$---we now also fix a normalization $\gerbeE_{\dot{\xi}}$ of the gerbe $\gerbeE_{\dot{V}}$, which identifies $\mathscr{T}_{\text{sc}}$ with a $\dot{\xi}$-twisted $Z_{\text{sc}}$-cocycle $(z_{\text{sc}},\phi_{\text{sc}})$, where $z_{\text{sc}} \in G^{*}_{\text{sc}}(\overline{F} \otimes_{F} \overline{F})$, which by construction has image in $Z^{1}(\gerbeE_{\dot{\xi}_{v}}, Z_{\text{sc}} \to G^{*}_{\text{sc}})$ equal to $(z_{\text{sc},v}, \phi_{\text{sc},v})$. We may thus define a global twisted cocycle by the formula $$x_{\text{sc}} := (p_{1}(g_{\text{sc}})^{-1}z_{\text{sc}}p_{2}(g_{\text{sc}}), \phi_{\text{sc}}) \in Z^{1}(\gerbeE_{\dot{\xi}}, Z_{\text{sc}} \to T_{0,\text{sc}}),$$ which satisfies $\text{loc}_{v}(x_{\text{sc}}) = x_{v,\text{sc}}$, where $\text{loc}_{v}$ on twisted cocycles is induced by the maps $u_{v} \to (P_{\dot{V}})_{F_{v}}$ and $G^{*}_{\text{sc}}(\overline{F} \otimes_{F} \overline{F}) \to G^{*}_{\text{sc}}(\overline{F_{v}} \otimes_{F_{v}} \overline{F_{v}})$ for a fixed $v$. It then follows from Corollary \ref{Tasho3.7.4} that the class $[x_{\text{sc}}] \in H^{1}(\gerbeE_{\dot{\xi}}, Z_{\text{sc}} \to T_{0,\text{sc}})$ is such that $[\text{loc}_{v}(x_{\text{sc}})] = [x_{v,\text{sc}}] \in H^{1}(\gerbeE_{\dot{\xi}_{v}}, Z_{\text{sc}} \to T_{0,\text{sc}})$ is trivial for all but finitely-many $v$, which shows that the expression \eqref{onepairing}, and thus also the expression \eqref{twopairings}, is $1$ for all but finitely many $v$, as desired. 

To finish proving the product identity, we first recall the functor $\overline{Y}_{+,\text{tor}} \colon \mathcal{R} \to \text{AbGrp}$. It follows from the proof of \cite[Prop. 5.3]{Tasho}, (the proof of which is purely character-theoretic) that we have a functorial embedding 
\begin{equation}\label{pairingembedding}
\overline{Y}_{+,\text{tor}}([Z \to G]) \hookrightarrow \pi_{0}([Z(\widehat{\overline{G}})^{+}])^{*},
\end{equation}
and it is straightforward to check that for any $[Z \to G] \in \mathcal{R}$, the following diagram commutes
\[
\begin{tikzcd}
\bigoplus_{v} \overline{Y}_{+v,\text{tor}}([Z \to G]) \arrow{d}  \arrow["\Sigma"]{r} & Y_{+,\text{tor}}([Z \to G]) \arrow{d} \\
\bigoplus_{v} \pi_{0}([Z(\widehat{\overline{G}})^{+v}])^{*} \arrow{r}&  \pi_{0}([Z(\widehat{\overline{G}})^{+}])^{*},
\end{tikzcd}
\]
where the left map is the sum of the local embeddings $\overline{Y}_{+v,\text{tor}}([Z \to G]) \hookrightarrow \pi_{0}([Z(\widehat{\overline{G}})]^{+v})^{*}$ and the bottom map is induced by restricting characters from $\pi_{0}([Z(\widehat{\overline{G}})]^{+v})$ to $\pi_{0}([Z(\widehat{\overline{G}})]^{+})$.

If for each $v$ we restrict the character $\langle \text{loc}_{v}([x_{\text{sc}}]), - \rangle$ on $\pi_{0}([\widehat{\overline{T_{0,\text{sc}}}}]^{+v})$ to $\pi_{0}([\widehat{\overline{T_{0,\text{sc}}}}]^{+})$ and then take the product over all $v$ (these characters are trivial for all but finitely-many $v$ due to the above discussion and Corollary \ref{localtoglobal3}), we obtain the trivial character on $\pi_{0}([\widehat{\overline{T_{0,\text{sc}}}}]^{+})$  via combining the above discussion with Corollary \ref{localtoglobal4}. By construction, we have that the image of $\hat{\varphi}(s_{\text{sc}}) \in \widehat{\bar{T_{0}}}$ in  $\widehat{\overline{T_{0,\text{sc}}}}/([\widehat{\overline{T_{0,\text{sc}}}}]^{+,\circ})$ lies in $\pi_{0}([\widehat{\overline{T_{0,\text{sc}}}}]^{+})$, which combines with the first part of this paragraph to give the equality $\langle x_{\text{sc}}, \hat{\varphi}(s_{\text{sc}}) \rangle = 1$, where the pairing is induced by the embedding \eqref{pairingembedding} and Theorem \ref{reductivetatenak}, proving that the above product over all places equals $1$, as desired. Finally, as in the number field case, the absence of $\dot{y}'_{v}$ and $\dot{y}''_{v}$ in the expression \eqref{twopairings} implies that the product does not depend on the choice of such elements. Moreover, since $\langle x_{v,\text{sc}}, \hat{\varphi}(s_{\text{sc}}) \rangle$ only depends on the cohomology class of $x_{\text{sc},v}$, the product also does not depend on the choices used to define the torsors $\text{loc}_{v}(\mathscr{T}_{\text{sc}})$.
\end{proof}

\subsection{The multiplicity formula for discrete automorphic representations}
We use the same notation as in \S 5.3. As in \cite[\S 4.5]{Tasho2}, fix an $L$-homomorphism $\varphi \colon L_{F} \to \prescript{L}{}G^{*}$ with bounded image, where $L_{F}$ is the hypothetical Langlands group of $F$. For each $v \in V$, the parameter $\varphi$ has a localization, which is a parameter $\varphi_{v} \colon L_{F_{v}} \to \prescript{L}{}G^{*}$. The local conjecture ensures that there exists an $L$-packet $\Pi_{\varphi_{v}}$ of tempered representations of rigid inner twists of $G^{*}$ together with a bijection $$\iota_{\varphi_{v}, \mathfrak{w}_{v}} \colon \Pi_{\varphi_{v}} \to \text{Irr}(S^{+}_{\varphi_{v}}),$$
where $\Pi_{\varphi_{v}}$ consists of equivalence classes of tuples $(G'_{v}, \psi'_{v}, (\mathscr{T}'_{v}, \bar{h}'_{v}), \pi'_{v})$ with $(\psi'_{v}, \mathscr{T}'_{v}, \bar{h}'_{v})$ a rigid inner twist of $G^{*}_{F_{v}}$ and $\pi'_{v}$ an irreducible tempered representation of $G'_{v}(F_{v})$ and $S^{+}_{\varphi_{v}} :=  Z_{\widehat{\bar{G}^{*}}}(\varphi_{v}).$

Recall that we have fixed a coherent family of rigid inner twists $(\psi_{v}, \mathscr{T}_{v}, \bar{h}_{v})_{v}$. Consider the subset $\Pi_{\varphi_{v}}(G) \subseteq \Pi_{\varphi_{v}}$ of (classes of) tuples $(G_{F_{v}}, \psi_{v}, (\mathscr{T}_{v}, \bar{h}_{v}), \pi_{v})$. We define the $L$-packet 
$$\Pi_{\varphi} := \{ \pi = \otimes_{v}' \pi_{v} \mid (G_{F_{v}}, \psi, \mathscr{T}_{v}, \bar{h}_{v}, \pi_{v}) \in \Pi_{\varphi_{v}}(G), \iota_{\varphi_{v},\mathfrak{w}_{v}}((G_{F_{v}}, \psi, (\mathscr{T}_{v}, \bar{h}_{v}), \pi_{v})) = 1\text{ for a.a. } v\}.$$

\begin{lem}\label{admissible} The set $\Pi_{\varphi}$ consists of irreducible admissible tempered representations of $G(\A)$.
\end{lem}

\begin{proof} We may assume without loss of generality that  $\mathcal{E}_{\dot{V}} = \mathcal{E}_{\dot{\xi}}$ for some representative $\dot{\xi}$ of the canonical class. As in the proof of \cite[Lem. 4.5.1]{Tasho2}, everything is clear except for the fact that the representation $\pi_{v}$ is unramified for almost all $v$. As explained loc. cit., we may find a finite set $S \subset V$ such that $G^{*}$ and $G$ have $O_{F,S}$-models $\mathcal{G}^{*}$, $\mathcal{G}$ (respectively), the inner twist isomorphism $\psi$ is defined over $O_{S} \subset F^{\text{sep}}$, the Whittaker datum $\mathfrak{w}_{v}$ is unramified for every $v \notin S$, and each local parameter $\varphi_{v}$ is unramified. We have the $G^{*}_{\gerbeE_{\dot{V}}}$-torsor $\mathscr{T}$ with fixed $\overline{F}$-trivialization $\overline{h}$ of $\overline{\mathscr{T}}$; 
via \cite[Prop. 7.3]{Dillery}, we can identify $\mathscr{T},  \overline{\mathscr{T}}$ with twisted torsors $(\mathcal{S}, \text{Res}(\mathscr{T}), \psi_{\mathcal{S}})$, $(\mathcal{S} \times^{G^{*}}G^{*}_{\text{ad}}, 0, \bar{\psi}_{\mathcal{S}})$ (the latter descends to a genuine $G^{*}_{\text{ad}}$ torsor $\overline{\mathcal{S}}$) and find a $\overline{F}$-trivialization $h$ of $\mathcal{S}$ such that the trivialization of $\overline{\mathcal{S}}$ induced by $h$ equals the one induced by $\bar{h}$.

We know from Proposition \ref{taibiprop} that we may enlarge $S$ to ensure that, for all $v \notin S$, the pair of each localization $\mathscr{T}_{v}$ and $\overline{F_{v}}$-trivialization $h_{v}$ (induced by $h$) is the pullback of a $\mathcal{G}^{*}_{O_{F_{v}}}$-torsor $T_{v}$ over $O_{F_{v}}$ with trivialization $h_{O_{F_{v}}}$ over $O_{F_{v}^{\text{nr}}}^{\text{perf}}$. Note that a priori each $\mathscr{T}_{v}$ is a torsor on $\gerbeE_{v}$, not on $\text{Sch}/F_{v}$, but we may enlarge $S$ to ensure that $\mathscr{T}_{v}$ is the pullback of a unique $G^{*}$-torsor over $F_{v}$, which we identify with $\mathscr{T}_{v}$ (see \S 4.4), so that this latter statement makes sense. 

The cohomology set $\check{H}^{1}(O_{F_{v}^{\text{nr}}}^{\text{perf}}/O_{F_{v}}, \mathcal{G}^{*})$ classifies isomorphism classes of $\mathcal{G}^{*}$-torsors over $O_{F_{v}}$ which have a trivialization over the fpqc extension $O_{F_{v}^{\text{nr}}}^{\text{perf}}$. We have a natural injective map $$\check{H}^{1}(O_{F_{v}^{\text{nr}}}^{\text{perf}}/O_{F_{v}}, \mathcal{G}^{*}) \to \check{H}^{1}_{\text{fppf}}(O_{F_{v}}, \mathcal{G}^{*}),$$ where the latter set classifies isomorphism classes of $\mathcal{G}^{*}$-torsors over $O_{F_{v}}$. Moreover, the set $\check{H}^{1}_{\text{fppf}}(O_{F_{v}}, \mathcal{G}^{*})$ is trivial, by \cite[Corollary 2.9]{Cesnavicius} (and Lang's theorem), giving the triviality of $\check{H}^{1}(O_{F_{v}^{\text{nr}}}^{\text{perf}}/O_{F_{v}}, \mathcal{G}^{*})$. It follows that we may find an element $g \in G^{*}(O_{F_{v}^{\text{nr}}}^{\text{perf}}) = \mathcal{G}^{*}(O_{F_{v}^{\text{nr}}}^{\text{perf}})$ whose \v{C}ech differential coincides with the element of $\mathcal{G}^{*}(O_{F_{v}^{\text{nr}}}^{\text{perf}} \otimes_{O_{F_{v}}} O_{F_{v}^{\text{nr}}}^{\text{perf}})$ whose left-translation gives $p_{1}^{*}h_{O_{F_{v}}} \circ p_{2}^{*}h_{O_{F_{v}}}^{-1}$ on $\mathcal{G}^{*}_{O_{F_{v}^{\text{nr}}}^{\text{perf}} \otimes_{O_{F_{v}}} O_{F_{v}^{\text{nr}}}^{\text{perf}}}$. As a consequence, we get by fpqc descent that the morphism $f' := \psi_{O_{F_{v}^{\text{nr}}}^{\text{perf}}} \circ \text{Ad}(g^{-1})$ descends to an $O_{F_{v}}$-morphism $f \colon \mathcal{G}^{*} \to \mathcal{G}$. 

The element $g \in G^{*}(\overline{F_{v}})$ defines an $F_{v}$-trivialization of $\mathscr{T}_{v}$ via the descent of the composition $$\Psi:= \ell_{g} \circ h_{v} \colon (\mathscr{T}_{v})_{\overline{F_{v}}} \to (\underline{G^{*}_{\gerbeE_{v}}})_{\overline{F_{v}}},$$ where $\ell_{g}$ denotes left-translation by $g$. As a consequence, $(f, \Psi)$ defines an isomorphism of rigid inner twists from $(\psi_{v}, \mathscr{T}_{v}, \bar{h}_{v})$ to the trivial rigid inner twist $(\text{id}_{G^{*}}, \underline{G^{*}_{\gerbeE_{v}}}, \text{id})$. Choosing $S$ large enough, the construction of $\Pi_{\varphi}$ then implies that $\iota_{\varphi_{v},\mathfrak{w}_{v}}((G^{*}, \text{id}_{G^{*}}, \underline{G^{*}_{\gerbeE_{v}}}, \text{id}, \pi_{v} \circ f)) = 1$, which means that the representation $\pi_{v} \circ f$ of $G^{*}(F_{v})$ is $\mathfrak{w}_{v}$-generic. This latter fact implies, by \cite{CS1}, that the representation $\pi_{v} \circ f$ is unramified with respect to the hyperspecial subgroup $G^{*}(O_{F_{v}})$ of $G^{*}(F_{v})$. The fact that the isomorphism $f$ is defined over $O_{F_{v}}$ then implies that $\pi_{v}$ is unramified with respect to the subgroup $G(O_{F_{v}})$, as desired.
\end{proof}

As conjectured in the number field case, we expect that every tempered discrete automorphic representation of $G(\A)$ belongs to $\Pi_{\varphi}$ for some discrete parameter $\varphi$. Moreover, for any such representation $\pi$, our framework allows for a conjectural description of its multiplicity in the discrete spectrum of $G$, which we turn to now. There is an exact sequence of $L_{F}$-modules (acting via $\text{ad} \circ \varphi$) $$1 \to Z(\widehat{G^{*}}) \to \widehat{G^{*}} \to  (\widehat{G^{*}})_{\text{ad}} \to 1$$ with connecting homomorphism $Z_{(\widehat{G^{*}})_{\text{ad}}}(\varphi) \to H^{1}(L_{F}, Z(\widehat{G^{*}}))$. We then set $$S^{\text{ad}}_{\varphi} := \text{ker}[Z_{(\widehat{G^{*}})_{\text{ad}}}(\varphi) \to H^{1}(L_{F}, Z(\widehat{G^{*}})) \to \prod_{v} H^{1}(L_{F_{v}}, Z(\widehat{G^{*}}))]$$ and set $\mathcal{S}_{\varphi} := \pi_{0}(S^{\text{ad}}_{\varphi})$ (a finite group, cf. \cite[\S 10.2]{Kott84}). We will construct a pairing $$\langle  -, - \rangle \colon \mathcal{S}_{\varphi} \times \Pi_{\varphi} \to \mathbb{C}$$ which yields an integer $$m(\varphi, \pi) := |\mathcal{S}_{\varphi}|^{-1} \sum_{x \in \mathcal{S}_{\varphi}} \langle x, \pi \rangle.$$ 

We then expect (from \cite{Kott84}) the multiplicity of $\pi$ in the discrete spectrum of $G$ to be given by $$\sum_{\varphi} m(\varphi, \pi),$$ where the sum is over all equivalence classes (as in \cite[\S 10.4]{Kott84}) of $\varphi$ such that $\pi \in \Pi_{\varphi}.$ 

The construction of the above pairing is identical to its number field in analogue, but we record it here for completeness. For some $s_{\text{ad}} \in S^{\text{ad}}_{\varphi}$, we choose a lift $s_{\text{sc}} \in S_{\varphi}^{\text{sc}}$ (the preimage of $S^{\text{ad}}_{\varphi}$ in $(\widehat{G^{*}})_{\text{sc}}$). Then, as explained in \cite[\S 4.5]{Tasho2}, we obtain from $s_{\text{sc}}$ an element $\dot{s}_{v} \in S^{+}_{\varphi_{v}}$ for each $v \in \dot{V}$, which we write as $(s_{\text{sc}} \cdot \dot{y}'_{v}, y''_{v})$ for $y'_{v} \in Z((\widehat{G^{*}})_{\text{der}})$ and $y''_{v} \in Z(\widehat{G^{*}})^{\circ}$ via the decomposition $\widehat{\bar{G}^{*}} = (\widehat{\bar{G}^{*}})_{\text{sc}} \times Z(\widehat{G^{*}})^{\circ}$. Following \cite{Tasho2}, we define $$ \langle (s_{\text{sc}} \cdot \dot{y}'_{v}, y''_{v}), (G_{F_{v}}, \psi, (\mathscr{T}_{v}, \bar{h}_{v}), \pi_{v}) \rangle := \text{tr}[\iota_{\varphi_{v},\mathfrak{w}_{v}}((G_{F_{v}}, \psi, (\mathscr{T}_{v}, \bar{h}_{v}), \pi_{v})](\dot{s}_{v}) \in \mathbb{C}.$$ 

\begin{prop}(\cite[Prop. 4.5.2]{Tasho2}) The value $$ \langle \text{loc}_{v}(\mathscr{T}_{\text{sc}}), \dot{y}_{v}' \rangle ^{-1} \cdot \langle (s_{\text{sc}} \cdot \dot{y}'_{v}, y''_{v}), (G_{F_{v}}, \psi, (\mathscr{T}_{v}, \bar{h}_{v}), \pi_{v}) \rangle $$ equals $1$ for all but finitely many $v$, and the product $$\langle s_{\text{ad}}, \pi \rangle := \prod_{v \in V}  \langle \text{loc}_{v}(\mathscr{T}_{\text{sc}}), \dot{y}_{v}' \rangle ^{-1} \cdot \langle (s_{\text{sc}} \cdot \dot{y}'_{v}, y''_{v}), (G_{F_{v}}, \psi, (\mathscr{T}_{v}, \bar{h}_{v}), \pi_{v}) \rangle $$ is independent of the choices of $s_{\text{sc}}, \dot{y}_{v}', y''_{v}$, the torsor $\mathscr{T}_{\text{sc}}$, and the global Whittaker datum $\mathfrak{w}$. Moreover, the function $s_{\text{ad}} \mapsto \langle s_{\text{ad}}, \pi \rangle$ is the character of a finite-dimensional representation of $\mathcal{S}_{\varphi}$.
\end{prop}

\begin{proof} This proof is identical to the proof of the analogous result in \cite{Tasho2}, replacing the use of Corollary 3.7.5 loc. cit. with our Corollary \ref{Tasho3.7.4} and the (conjectural) endoscopic character identities from \cite[\S 3.4]{Tasho}, with the analogous identities from \cite[\S 7.3]{Dillery}.
\end{proof}


\appendix

\section{Complexes of tori and \v{C}ech cohomology} This appendix gives an extension of the theory of \textit{complexes of tori} developed in the appendices of \cite{KS1} to the setting of local and global function fields.

\subsection{Complexes of tori over local function fields---basic results}
Suppose that we have a complex of commutative $R$-groups, which is concentrated in degrees $0$ and $1$, denoted by $G \xrightarrow{f} H$ (or, when both groups are $R$-tori, by $T \xrightarrow{f} U$). For any fpqc ring homomorphism $R \to S$, we obtain a double complex $K^{\bullet, \bullet}$ by taking the \v{C}ech complexes for $G$ and $H$; that is, the double complex
\begin{equation}\label{doublecomplex}
\begin{tikzcd}
G(S) \arrow{r} \arrow{d} & G(S \otimes_{R} S) \arrow{r} \arrow{d} & G(S \otimes_{R} S \otimes_{R} S) \arrow{r} \arrow{d} & \dots \\
H(S) \arrow{r} & H(S \otimes_{R} S) \arrow{r} & H(S \otimes_{R} S \otimes_{R} S) \arrow{r} & \dots;
\end{tikzcd}
\end{equation}
We can associate to this double complex a new complex $L^{\bullet}$, whose degree-$r$ term is given by $$L^{r}(T^{\bullet}) = \bigoplus_{m+n =r} K^{m,n} = G(S^{\bigotimes_{R}r}) \oplus H(S^{\bigotimes_{R} r-1}),$$ with differentials defined by $(d_{G} \oplus f -d_{H})$. Following \cite{KS1}, we call the elements of $L^{r}$ \textit{(\v{C}ech) $r$-hypercochains}, and the elements of the kernel of the $r$th differential \textit{(\v{C}ech) $r$-hypercocycles}. Denote by $H^{r}(S/R, G \xrightarrow{f} H)$ the $r$th cohomology group of the complex $L^{\bullet}$. Note that, by fpqc descent, $H^{0}(S/R, G \xrightarrow{f} H) = \text{ker}(G(R) \to H(R)) = \text{ker}(f)(R)$, which will be useful when $\text{ker}(f)$ is a finite-type $F$-group scheme whose cohomology we want to investigate. 

The spectral sequences associated to a double complex give us the long exact sequence 
\begin{equation}\label{LES1}
\dots \to H^{r}(S/R, G \xrightarrow{f} H) \to \check{H}^{r}(S/R, G) \to  \check{H}^{r}(S/R, H) \to H^{r+1}(S/R, G \xrightarrow{f} H) \to \dots,
\end{equation}
where the first map sends $[(x,y)]$ to $[x]$, the last map sends $[x]$ to $[(0,x)]$, and the middle map is induced by $f$. They also give the long exact sequence
\begin{equation}\label{LES2}
\dots \to \check{H}^{r}(S/R, \text{ker}(f)) \to H^{r}(S/R, G \xrightarrow{f} H) \to H^{r-1}(\text{cok}(f^{\bigotimes \bullet})) \to \check{H}^{r+1}(S/R, \text{ker}(f)) \to \dots,
\end{equation}
where $\text{cok}(f^{\bigotimes \bullet})$ denotes the complex with degree-$r$ term given by $\frac{H(S^{\bigotimes_{R}r})}{f(G(S^{\bigotimes_{R}r}))}.$ 

In the long exact sequence \eqref{LES2}, the first map is given by $[x] \mapsto [(x,0)]$, the middle map by $[(x,y)] \mapsto [\bar{y}]$, and the last map by the composition of the map $H^{r-1}(\text{cok}(f^{\bigotimes \bullet})) \to H^{r}(\text{im}(f^{\bigotimes \bullet}))$ defined by picking a preimage $x \in H(S^{\bigotimes_{R}r})$ of an $r$-cocycle $\bar{x} \in \frac{H(S^{\bigotimes_{R}r})}{f(G(S^{\bigotimes_{R}r}))}$ and then applying the \v{C}ech differential, and the map $H^{r}(\text{im}(f^{\bigotimes \bullet})) \to \check{H}^{r+1}(S/R, \text{ker}(f))$ given by picking a preimage in $G(S^{\bigotimes_{R}(r+1)})$ of $x \in f(G(S^{\bigotimes_{R}(r+1)}))$ and then differentiating. 

We now make the situation more concrete by setting $R=F$ a field; the following result is an immediate extension of the fact that, for a smooth finite type commutative $F$-group scheme $G$, the comparison map $\check{H}^{i}(F^{\text{sep}}/F, G) \to \check{H}^{i}(\overline{F}/F, G)$ is always an isomorphism:

\begin{lem} \label{complexcomparison} For all $i \geq 1$, the natural map $H^{i}(F^{\text{sep}}/F, T \xrightarrow{f} U) \to H^{i}(\overline{F}/F, T \xrightarrow{f} U)$ is an isomorphism.
\end{lem}

\begin{proof} This follows immediately from the five-lemma, applied to the commutative diagram with exact rows induced by \eqref{LES1}
\[
\begin{tikzcd}[column sep=small]
 \check{H}^{i-1}(F^{\text{sep}}/F, T) \arrow{d} \arrow{r} &  \check{H}^{i-1}(F^{\text{sep}}/F, U) \arrow{d} \arrow{r} & H^{i}(F^{\text{sep}}/F, T \xrightarrow{f} U) \arrow{d} \arrow{r} 
&  \check{H}^{i}(F^{\text{sep}}/F, T) \arrow{d} \arrow{r} & \check{H}^{i}(F^{\text{sep}}/F, T) \arrow{d}   \\
\check{H}^{i-1}(\overline{F}/F, T) \arrow{r} &  \check{H}^{i-1}(\overline{F}/F, U) \arrow{r} & H^{i}(\overline{F}/F, T \xrightarrow{f} U)  \arrow{r} 
&  \check{H}^{i}(\overline{F}/F, T)  \arrow{r} & \check{H}^{i}(\overline{F}/F, T),
\end{tikzcd}
\]
where all vertical maps other than the one in consideration are isomorphisms, since $T$ and $U$ are tori (in particular, are smooth). 
\end{proof}

We have an identification of \v{C}ech hypercohomology for $F^{\text{sep}}/F$ and Galois cohomology:

\begin{lem}\label{galoiscomplexes} For all $i$, we have a canonical isomorphism $$H^{i}(F^{\text{sep}}/F, T \xrightarrow{f} U) \xrightarrow{\sim} H^{i}(\Gamma, T(F^{\text{sep}}) \to U(F^{\text{sep}})),$$ where the latter group is as defined in \cite{KS1}, Appendix A.
\end{lem}

\begin{proof} This is immediate from applying the comparison isomorphisms discussed in \S 2.1. 
\end{proof}

Lemmas \ref{complexcomparison} and \ref{galoiscomplexes} show that one can perform computations for $H^{i}(\overline{F}/F, T \xrightarrow{f} U)$ using group cohomology as in \cite{KS1}. Take $R=F$ to be a local function field, $S= \overline{F}$ a fixed algebraic closure, $G \xrightarrow{f} H$ is a complex of $F$-tori, denoted by $T^{\bullet} := T \xrightarrow{f} U$. Using the aforementioned comparison with group cohomology, one can then define a local \textit{Tate-Nakayama pairing} $$H^{r}(\overline{F}/F, T \xrightarrow{f} U) \times H^{3-r}(\overline{F}/F, \underline{X}^{*}(U) \xrightarrow{f^{*}} \underline{X}^{*}(T)) \to \mathbb{Q}/\Z$$
identically as in \cite[\S A.2]{KS1}. Note that for any $F$-torus $S$, we have $H^{i}(F, S) = 0$ for all $i \geq 3$, since $H^{i}(F, S) = H^{i}(\Gamma, S(F^{\text{sep}}))$, and the cohomological dimension of $F$ is $2$. This same reasoning also implies that $H^{i}(\Gamma, X^{*}(S)) = 0$ for all $i \geq 3$. Using the long exact sequence \eqref{LES1}, we deduce that both of the groups in the above pairing are zero for $r \geq 4$ and negative $r$. We have the analogue of \cite[Lem. A.2.A]{KS1}, whose proof is unchanged (using Lemmas \ref{complexcomparison}, \ref{galoiscomplexes}):

\begin{lem} The above pairing induces an isomorphism $$H^{r}(\overline{F}/F, T \xrightarrow{f} U) \to H^{3-r}(\overline{F}/F, \underline{X}^{*}(U) \xrightarrow{f^{*}} \underline{X}^{*}(T))^{*}$$ for $r=2,3$. For $r=2,3$, the group $H^{r}(\overline{F}/F, T \xrightarrow{f} U)$ is finitely-generated, and is free for $r=3$. 
\end{lem}


\subsection{A local pairing}
Recall that the hypercohomology groups $H^{r}(W_{F}, \widehat{U} \xrightarrow{\hat{f}} \widehat{T})$ are defined as follows: For any $F$-torus $S$, we set $C^{0}(W_{F}, \widehat{S}) = \widehat{S}(\mathbb{C})$ (with inflated $W_{F}$-action), $C^{1}(W_{F}, \widehat{S})$ the group of continuous $1$-cocycles of $W_{F}$ in $\widehat{T}(\mathbb{C})$, and all other cochain groups to be zero. We then define $r$-hypercochains with respect to the complex $\widehat{U} \xrightarrow{\hat{f}} \widehat{T}$ to be elements of $$C^{r}(W_{F}, \widehat{U} \xrightarrow{\hat{f}} \widehat{T}) = C^{r}(W_{F}, \widehat{U}) \oplus C^{r-1}(W_{F}, \widehat{T}),$$ with the usual differentials for total complexes, and cohomology groups $H^{r}(W_{F}, \widehat{U} \xrightarrow{\hat{f}} \widehat{T})$.

The exponential sequence for the dual tori over $\mathbb{C}$ gives a pairing (cf. \cite[\S A.3]{KS1}) $$H^{r}(\Gamma, \widehat{U} \xrightarrow{\hat{f}} \widehat{T}) \times H^{2-r}(\Gamma, X^{*}(U) \xrightarrow{f^{*}} X^{*}(T)) \to \mathbb{C}^{\times}$$ which is generalized loc. cit. for $r=1$ to a pairing 
\begin{equation}\label{complexpairing} H^{1}(F^{\text{sep}}/F, T \xrightarrow{f} U) [\stackrel{\text{Lem. \ref{complexcomparison}}}{=}  H^{1}(\overline{F}/F, T \xrightarrow{f} U) ] \times H^{1}(W_{F}, \widehat{U} \xrightarrow{\hat{f}} \widehat{T}) \to \mathbb{C}^{\times}
\end{equation}
which carries over to our situation unchanged.

We have the following two exact sequences $$ \dots \to H^{0}(F, U) \xrightarrow{j} H^{1}(F^{\text{sep}}/F, T \xrightarrow{f} U) \xrightarrow{i} H^{1}(F, T) \to \dots, $$ 
$$\dots \to H^{0}(W_{F}, \widehat{T}) \xrightarrow{\hat{j}} H^{1}(W_{F}, \widehat{U} \xrightarrow{\hat{f}} \widehat{T}) \xrightarrow{\hat{i}} H^{1}(W_{F}, \widehat{U}) \to \dots, $$
from which we derive two compatibilities of pairings. First, we have $\langle j(u), \hat{z} \rangle = \langle u, \hat{i}(\hat{z}) \rangle^{-1}$, where the left-hand pairing is \eqref{complexpairing} and the right-hand pairing $U(F) \times H^{1}(W_{F}, \widehat{U}) \to \mathbb{C}^{\times}$ is given by Langlands duality for tori. Second, we have $\langle z, \hat{j}(\hat{t}) \rangle = \langle i(z), \hat{t} \rangle$, where the left-hand pairing is from \eqref{complexpairing} and the right-hand pairing $H^{1}(F, T) \times \widehat{T}^{\Gamma_{F}} \to \mathbb{C}^{\times}$ comes from Tate-Nakayama duality. 

The first goal is to endow $H^{1}(\overline{F}/F, T \xrightarrow{f} U)$ with a natural locally-profinite topology. We first claim that the image $f(T(F)) \subseteq U(F)$ is closed: The scheme-theoretic image $f(T)$ is a closed subscheme of $U$ by the closed orbit lemma, so that $f(T)(F)$ is closed in $U(F)$, which means that we can replace $U$ by $f(T)$ to reduce to the case where $f$ is (scheme-theoretically) surjective. We then choose an $F$-torus $T'$ such that $f$ factors as a composition $T \xrightarrow{f'} U' \xrightarrow{f''} U$ where the kernel of $f'$ is a torus and $f''$ is an isogeny. Note that $f''$ is finite, and hence proper, which implies that the continuous map $U'(F) \to U(F)$ is proper (as a map of topological spaces), and hence closed (since $U(F)$ is locally compact and Hausdorff), and so we can reduce further to the case where the kernel of $T \to U$ is a torus. 

In this final case the morphism $T \xrightarrow{f} U$ is smooth---indeed, quotient maps are always flat and surjective, and the smoothness of the kernel implies that we get a short-exact sequence at the level of tangent spaces at the identity. It then follows from the inverse function theorem for analytic manifolds (\cite[Thm. III.9.2]{Serre3} (which is proved for all analytic manifolds over complete nonarchimedean fields) that $f$ is open, and thus sends closed full preimages to closed subsets. 

The closedness of $f(T(F))$ in $U(F)$ implies that the quotient $U(F)/f(T(F))$ has the canonical structure of a topological group. We then give $H^{1}(\overline{F}/F, T \xrightarrow{f} U)$ the unique locally-profinite topology such that the map $U(F)/[f(T(F))] \to H^{1}(\overline{F}/F, T \xrightarrow{f} U)$ is an open immersion (note that $H^{1}(F, T)$ is finite). Duality of \eqref{complexpairing} passes to our setting verbatim (as in \cite[Lem. A.3.B]{KS1}):

\begin{prop} Using the above topology, the pairing \eqref{complexpairing} induces a surjective homomorphism $$H^{1}(W_{F}, \widehat{U} \xrightarrow{\hat{f}} \widehat{T}) \to \Hom_{\text{cts}}(H^{1}(\overline{F}/F, T \xrightarrow{f} U), \mathbb{C}^{\times})$$ with kernel equal to the image of $(\widehat{T}^{\Gamma_{F}})^{\circ}$ under the natural map $\hat{j} \colon \widehat{T}^{\Gamma_{F}} \to H^{1}(W_{F}, \widehat{U} \xrightarrow{\hat{f}} \widehat{T})$.
\end{prop}


We set $H^{1}(W_{F}, \widehat{U} \xrightarrow{f} \widehat{T})_{\text{red}}$ to be the quotient $H^{1}(W_{F}, \widehat{U} \xrightarrow{f} \widehat{T})/\hat{j}[(\widehat{T}^{\Gamma_{F}})^{\circ}].$ Note that the group $H^{1}(W_{F}, \widehat{U} \xrightarrow{\hat{f}} \widehat{T})$ is redundant when $f$ is an isogeny, by the following result:

\begin{prop}\label{weiltogalois1} When $f$ is an isogeny, the inflation map $H^{1}(\Gamma, \widehat{U} \xrightarrow{\hat{f}} \widehat{T}) \to H^{1}(W_{F}, \widehat{U} \xrightarrow{\hat{f}} \widehat{T})$ is an isomorphism.
\end{prop}

\begin{proof} For any finite extension $K/F$ splitting $U$ and $T$, we have an ``inflation-restriction" sequence
$$0 \to H^{1}(\Gamma_{K/F}, \widehat{U} \xrightarrow{\hat{f}} \widehat{T}) \to H^{1}(W_{K/F}, \widehat{U} \xrightarrow{\hat{f}} \widehat{T}) \to H^{1}(K^{*}, \widehat{U} \xrightarrow{\hat{f}} \widehat{T}) ,$$ where in the last term we are viewing $K^{*}$ as a topological group. Indeed, suppose that we have a 1-hypercocycle $(\underline{u}, t) \in C^{1}(W_{F}, \widehat{U}) \oplus \widehat{T}(\mathbb{C})$ such that its restriction to $K^{*}$ is a 1-coboundary; that is, we have $x \in \widehat{U}(\mathbb{C})$ such that $(\underline{u}, t) = (dx, f(x)^{-1})$. This means that for all $z \in F^{*}$, we have $\underline{u}(z) = \prescript{z}{}x \cdot x^{-1} = 1$, so that $\underline{u}$ is trivial on $K^{*}$, and is therefore inflated from any 1-cocycle $\tilde{\underline{u}}$ of $\Gamma_{K/F}$ determined by picking a set-theoretic section $\Gamma_{K/F} \to W_{K/F}$. Since the $W_{K/F}$-action is inflated from $\Gamma_{K/F}$, the element $(\tilde{\underline{u}}, t)$ is a 1-hypercocycle of $\Gamma_{K/F}$ mapping to $(\underline{u}, t)$, as desired. 

For $K/F$ as above, fix $x \in H^{1}(W_{K/F}, \widehat{U} \xrightarrow{\hat{f}} \widehat{T})$; to show that, for large enough $L/F$ containing $K$, it lies in the image of the inflation map, it's enough to show that for large enough $L$ its image in $H^{1}(L^{*},  \widehat{U} \xrightarrow{\hat{f}} \widehat{T}) = \Hom_{\text{cts}}(L^{*}, \text{ker}(\hat{f}))$ is zero. This follows from the fact that any continuous homomorphism $\chi \colon K^{*} \to \text{ker}(\hat{f})$ has finite-index open kernel and the norm groups $N_{L/K}(L^{*})$ shrink to the identity as $L/K$ varies over all finite Galois extensions of $F$ containing $K$.
\end{proof}

\subsection{Complexes of tori over global function fields---basic results} The last two subsections extend \cite[Appendix C]{KS1} to a global function field $F$. We fix a complex of $F$-tori $T \xrightarrow{f} U$, let $\A^{\text{sep}} := F^{\text{sep}} \otimes_{F} \A$, and define $\bar{H}^{i}(\overline{\A}/\A, T \xrightarrow{f} U)$ as the hypercohomology of the double complex 
\[ 
\begin{tikzcd}
\frac{T(\overline{\A})}{T(\overline{F})} \arrow{r} \arrow{d} & \frac{T(\overline{\A} \otimes_{\A} \overline{\A})}{T(\overline{F} \otimes_{F} \overline{F})} \arrow{r} \arrow{d} & \frac{T(\overline{\A} \otimes_{\A} \overline{\A} \otimes_{\A} \overline{\A})}{T(\overline{F} \otimes_{F} \overline{F} \otimes_{F} \overline{F})} \arrow{r} \arrow{d} & \dots \\
\frac{U(\overline{\A})}{U(\overline{F})} \arrow{r} & \frac{U(\overline{\A} \otimes_{\A} \overline{\A})}{U(\overline{F} \otimes_{F} \overline{F})} \arrow{r} & \frac{U(\overline{\A} \otimes_{\A} \overline{\A} \otimes_{\A} \overline{\A})}{U(\overline{F} \otimes_{F} \overline{F} \otimes_{F} \overline{F})} \arrow{r}  & \dots,
\end{tikzcd}
\]
giving us a long exact sequence 
\begin{equation}\label{globalLES}
 \dots \to H^{i}(\overline{F}/F, T \xrightarrow{f} U) \to H^{i}(\overline{\A}/\A, T \xrightarrow{f} U) \to \bar{H}^{i}(\overline{\A}/\A, T \xrightarrow{f} U) \to H^{i+1}(\overline{F}/F, T \xrightarrow{f} U) \to \dots 
 \end{equation}
 
Let $S$ be a finite set of places of $F$ containing all places at which $T$ and $U$ are ramified. For every place $v$ of $F$, we fix an algebraic closure $\overline{F_{v}}$ as well as an embedding $\overline{F} \hookrightarrow \overline{F_{v}}$. The following two results let us work in the group-cohomological setting:

\begin{lem}\label{globalcomplexcomparison}  For all $i \geq 0$, the natural map $H^{i}(\A^{\text{sep}}/\A, T \xrightarrow{f} U) \to H^{i}(\overline{\A}/\A, T \xrightarrow{f} U)$ is an isomorphism, and the same is true with $\A$ replaced by $F$.
\end{lem}

\begin{proof} Combining the proof of Lemma \ref{vadelicvanishing1} with our results on adelic tensor products in \S 2.2 shows that $H^{j}((\A^{\text{sep}})^{\bigotimes_{\A}n}, M)$ vanishes for any $F$-torus $M$, $j,n \geq 1$, and so the natural map $\check{H}^{i}(\A^{\text{sep}}/\A, M) \to H^{i}(\A, M)$ is an isomorphism. Since this is also true with $\A^{\text{sep}}$ replaced by $\overline{\A}$, the same argument in the proof of Lemma \ref{complexcomparison} gives the result. The proof for $F$ is the same.
\end{proof}

\begin{cor}\label{globalcomplexcomparison2}\label{adelicgroupcohomology}\label{adelicgroupcohomology2}  For all $i \geq 0$:
\begin{enumerate}
\item{The natural map $\bar{H}^{i}(\A^{\text{sep}}/\A, T \xrightarrow{f} U) \to \bar{H}^{i}(\overline{\A}/\A, T \xrightarrow{f} U)$ is an isomorphism.}
\item{We have a canonical isomorphism $$H^{i}(\A^{\text{sep}}/\A, T \xrightarrow{f} U) \to H^{i}(\Gamma_{F}, T(\A^{\text{sep}}) \xrightarrow{f} U(\A^{\text{sep}})).$$}
\item{We have a canonical isomorphism $$\bar{H}^{i}(\A^{\text{sep}}/\A, T \xrightarrow{f} U) \to \bar{H}^{i}(\Gamma_{F}, T(\A^{\text{sep}})/T(F^{\text{sep}}) \xrightarrow{f} U(\A^{\text{sep}})/U(F^{\text{sep}})).$$}
\end{enumerate}
\end{cor}

\begin{proof} The first statement is an immediate consequence of combining Lemma \ref{globalcomplexcomparison} with the long exact sequence \eqref{globalLES} and the five-lemma. The second two statements follow immediately.
\end{proof}

We now give an analogue of \cite[Lem. C.1.A]{KS1}, which we need to in order to work with restricted products. Note that the complex $T \xrightarrow{f} U$ is defined over $O_{F,S}$. Let $O_{v}$ denote the completion of $O_{F}$ at $v$, and $O_{v}^{\text{nr}}$ the ring of integers of the maximal unramified extension $F_{v}^{\text{nr}}/F_{v}$.

\begin{lem}\label{subgroups} For any place $v \notin S$, the group $H^{i}(O_{v}^{\text{nr}}/O_{v}, T \xrightarrow{f} U)$ is equal to the kernel of $T(O_{v}) \xrightarrow{f} U(O_{v})$ if $i=0$, to the cokernel of the same map if $i=1$, and is trivial if $i \geq 2$. Moreover, for all $i$ we have natural injections $$H^{i}(O_{v}^{\text{nr}}/O_{v}, T \xrightarrow{f} U) \hookrightarrow H^{i}(\overline{F_{v}}/F_{v}, T \xrightarrow{f} U).$$\end{lem}

\begin{proof} To prove the first statement, using the long exact sequence \eqref{LES1}, it's enough to show that $\check{H}^{i}(O_{v}^{\text{nr}}/O_{v}, M) = 0$ for any $F$-torus $M$ which is unramified at $v$ for $i \geq 1$ (applying this result to $T$ and $U$). We first claim that these groups may be identified with $H^{i}(O_{v}, M)$ under the natural \v{C}ech-to-fppf comparison map. As usual, it's enough to show that the fppf cohomology groups $H^{j}((O_{v}^{\text{nr}})^{\bigotimes_{O_{v}}n}, M)$ vanish for all $j, n \geq 1$ . Since $O_{v}$ is the ring of integers in a nonarchimedean local field, for a fixed finite unramified extension $E_{w}/F_{v}$, we have the chain of identifications $$O_{w} \otimes_{O_{v}} O_{w} \xrightarrow{\sim} O_{w} \otimes_{O_{v}} O_{v}[\varpi] \xrightarrow{\sim} O_{w} \otimes_{O_{v}} O_{v}[x]/(f) \xrightarrow{\sim} \prod_{\Gamma_{E_{w}/F_{v}}} O_{w},$$ where $\varpi \in O_{w}$ and $f \in O_{v}[x]$. As in \S 2.1 it's enough to prove the $n=1$ case; i.e., showing that the groups $H^{i}(O_{v}^{\text{nr}}, M)$ vanish for all $i \geq 1$. This follows immediately from the fact that they are the direct limit of the groups $H^{i}(O_{E_{w}}, M)$, where $E_{w}$ is as above, which all vanish by \cite[Corollary 2.9]{Cesnavicius}, using that $O_{E_{w}}$ is a Henselian local ring with finite residue field $k_{w}$, and $M_{k_{w}}$ is connected, being a $k_{w}$-torus. With the claim in hand, the result is immediate from the same Corollary, since $O_{v}$ is a Henselian local ring with finite residue field $k_{v}$ such that $M_{k_{v}}$ is connected. 

We now move on to the second statement. Using the first statement, we only need to show this for $i=1$. As in the proof of \cite[Lem. C.1.A]{KS1}, it's enough to show that any element $u \in U(O_{v}) \cap f(T(F_{v}))$ lies in $f(T(O_{v}))$. To this end, we may assume that $f$ is surjective, and we may again factor $f$ as the composition $T \xrightarrow{f'} U' \xrightarrow{f''} U$, where $f'$ has a torus as its kernel and $f''$ is an isogeny. The argument of the proof of \cite[Lem. C.1.A ]{KS1} proves the result for $f'$, so that $U'(O_{v}) \cap f'(T(F_{v})) = f'(T(O_{v}))$. 

Note that $f''$ is proper as a morphism of $F_{v}$-schemes, so the map $U'(F_{v}) \to U(F_{v})$ is proper as a morphism of topological spaces; this implies that the preimage of the compact subgroup $U(O_{v})$ under $f''$ is a compact subgroup of $U'(F_{v})$, and so lies in $U'(O_{v})$, the maximal compact subgroup. Thus, if $t \in T(F_{v})$ is such that $f(t) \in U(O_{v})$, then $f'(t) \in U'(O_{v})$, so that $f'(t) = f'(x)$ for some $x \in T(O_{v})$, and now $f(t) = f(x)$, as desired.
\end{proof}

We now give a restricted product structure to the groups $H^{i}(\overline{\A}/\A, T \xrightarrow{f} U)$:

\begin{prop}\label{restrictedproduct} We have a canonical isomorphism $$H^{i}(\overline{\A}/\A, T \xrightarrow{f} U) \xrightarrow{\sim} \prod_{v \in V_{F}}' H^{i}(\overline{F_{v}}/F_{v},  T \xrightarrow{f} U),$$ where the product is restricted with respect to the subgroups $H^{i}(O_{v}^{\text{nr}}/O_{v}, T \xrightarrow{f} U)$ for $v \notin S$ (which are indeed subgroups by Lemma \ref{subgroups}). When $i \geq 2$, this restricted product is a direct sum.
\end{prop}

\begin{proof} The first step is to use Lemma \ref{globalcomplexcomparison} to replace $H^{i}(\overline{\A}/F, T \xrightarrow{f} U)$ by $H^{i}(\A^{\text{sep}}/\A, T \xrightarrow{f} U)$, and Lemma \ref{complexcomparison} to replace $H^{i}(\overline{F_{v}}/F_{v}, T \xrightarrow{f} U)$ by $H^{i}(F_{v}^{\text{sep}}/F_{v}, T \xrightarrow{f} U)$. Consider a finite Galois extension $K/F$, and let $S_{(K)}$ denote a large finite set of places containing $S$ such that $K$ is unramified outside $S_{(K)}$. For any place $w \in V_{K}$ lying over $v \notin S_{(K)}$, the natural map $H^{i}(O_{w}/O_{v}, T \xrightarrow{f} U) \to H^{i}(O_{v}^{\text{nr}}/O_{v}, T \xrightarrow{f} U)$ is an isomorphism (replace $O_{v}^{\text{nr}}$ by $O_{w}$ in the proof of Lemma \ref{subgroups}). From here, we may work with group cohomology and use the identical argument of \cite[Lem. C.1.B]{KS1} to deduce the result.
\end{proof}

Continuing to follow \cite[\S C]{KS1}, we topologize our adelic cohomology groups (the versions over $F^{\text{sep}}$). We give $H^{i}(F^{\text{sep}}/F, T \xrightarrow{f} U)$ the discrete topology for all $i$, we give $H^{0}(\A^{\text{sep}}/\A, T \xrightarrow{f} U)$ the topology it inherits as a closed subgroup of $T(\A)$, and $H^{1}(\A^{\text{sep}}/\A, T \xrightarrow{f} U)$ the topology making $U(\A)/f[T(\A)] \to H^{1}(\A^{\text{sep}}/\A, T \xrightarrow{f} U)$ an open immersion; note that $f[T(\A)]$ is closed in $U(\A)$, since $f(T(F_{v})) \cap U(O_{v}) = f(T(O_{v}))$ for $v \notin S$ and $\prod_{v \notin S} f(T(O_{v}))$ is compact, and $f(T(F_{v}))$ is closed in $U(F_{v})$ for $v \in S$ (by an argument that we made earlier in this subsection). In the above discussion, we are using \cite[Thm. 2.20]{Cesnavicius} to decompose $T(\A)$ and $U(\A)$ as restricted products. We give the groups $H^{i}(\A^{\text{sep}}/\A, T \xrightarrow{f} U)$ the discrete topology for $i \geq 2$. 

We now turn to topologizing the groups $\bar{H}^{i}(\A^{\text{sep}}/\A, T \xrightarrow{f} U)$, which is more involved. For any $F$-torus $S$, the group $S(\A^{\text{sep}})$ carries a natural topology given by the direct limit topology of the topological groups $S(\A_{K})$, where $K/F$ ranges over all finite Galois extensions, and this topology coincides with the one induced by giving $\A^{\text{sep}}$ the structure of a topological ring via the direct limit topology. Note that the ring $\A^{\text{sep}}$ is Hausdorff; to see, this, note that each $\A_{K}$ is a metrizable topological space (by \cite[Prop. 1.1]{Kelly}), and is thus normal; now the direct limit of normal spaces with transition maps that are closed immersions (as is the case with $\A_{K} \to \A_{L}$) is a normal topological space, and hence a fortiori Hausdorff.

It follows that $S(\A^{\text{sep}})$ is Hausdorff (by \cite[Prop. 2.1]{Conrad2}). Since $S(K)$ is closed in $S(\A_{K})$ for all $K$, it follows that $S(F^{\text{sep}})$ is a closed subgroup of $S(\A^{\text{sep}})$ (using that $S(F^{\text{sep}}) \cap S(\A_{K}) = S(K)$), so the topological group $S(\A^{\text{sep}})/S(F^{\text{sep}})$ makes sense. Moreover, the subgroup $[S(\A^{\text{sep}})/S(F^{\text{sep}})]^{\Gamma}$ is closed, since it's the intersection over all $\sigma \in \Gamma$ of the subsets $[S(\A^{\text{sep}})/S(F^{\text{sep}})]^{\sigma}$, which are the preimages of the (closed) diagonal $\Delta(S(\A^{\text{sep}})/S(F^{\text{sep}}))$ under the continuous map $\text{id} \times (-)^{\sigma}$. 
Using these topologies, the natural map 
\begin{equation}\label{adelicTtoU}[T(\A^{\text{sep}})/T(F^{\text{sep}})]^{\Gamma} \to [U(\A^{\text{sep}})/U(F^{\text{sep}})]^{\Gamma}
\end{equation} is continuous, and hence the closed kernel (our group $\bar{H}^{0}(\A^{\text{sep}}/\A, T \xrightarrow{f} U)$) has the natural structure of a topological group, settling the $i=0$ case. 

Recall that there is a group homomorphism
 \begin{equation}\label{Hmapeq}
 H \colon [T(\A^{\text{sep}})/T(F^{\text{sep}})]^{\Gamma} \to (X_{*}(T) \otimes q^{\mathbb{Z}})^{\Gamma}
 \end{equation}
determined by, for all $\lambda \in X^{*}(T)^{\Gamma}$, the equality 
\begin{equation}\label{degree}
\langle  \lambda, H(\bar{t}) \rangle = \Vert  \lambda(\bar{t}) \Vert, \end{equation}
where we are using the fact that $[(\A^{\text{sep}})^{\times}/(F^{\text{sep}})^{\times}]^{\Gamma} = \A^{\times}/F^{\times}$,  $q$ is the size of the constant field of $F$, and $\Vert \cdot \Vert$ denotes the adelic norm; we will show in Lemma \ref{compactness1} that $\text{ker}(H)$ is always compact. The map and all of the above properties work just as well for $F^{\text{sep}}$ replaced with finite Galois $K/F$ splitting $T$.

We claim that the image of \eqref{adelicTtoU} is a closed subgroup; this image is the direct limit of the images of the maps of topological groups $[T(\A_{K})/T(K)]^{\Gamma_{K/F}} \to [U(\A_{K})/U(K)]^{\Gamma_{K/F}}$ over all finite Galois $K/F$, and so it's enough to show that all of these images are closed---the following argument was suggested to us by Brian Conrad:

\begin{lem} For $K/F$ finite Galois splitting $T$, the image of the map $[T(\A_{K})/T(K)]^{\Gamma_{K/F}} \to [U(\A_{K})/U(K)]^{\Gamma_{K/F}}$ is closed.
\end{lem}

\begin{proof}
We may assume that $f$ is scheme-theoretically surjective. Letting $S$ be the maximal $F$-split subtorus in $\text{ker}(f)$ and setting $T' := T/S$, we have short exact sequences
\begin{equation*}
1 \to S(A_{K}) \to T(A_{K}) \to T'(A_{K}) \to 1,
\end{equation*}
where $A_{K} = K$ or $\A_{K}$ and deduce the surjectivity of $[T(A_{K})/T(K)]^{\Gamma_{K/F}}  \to [T'(A_{K})/T'(K)]^{\Gamma_{K/F}}$ using Hilbert 90 for id\`{e}le class groups. This allows us to assume that $T$ and $U$ have the same $F$-split rank.

The map $H_{T}$ sends $[T(\A_{K})/T(K)]^{\Gamma_{K/F}}$ to a finite-index subgroup of $ (X_{*}(T) \otimes q^{\mathbb{Z}})^{\Gamma_{K/F}}$, as does $H_{U}$ with the analogous groups. Since $X^{*}(U)^{\Gamma_{K/F}} \to X^{*}(T)^{\Gamma_{K/F}}$ is a finite-index inclusion (by our $F$-split rank assumptions), the continuous map
\begin{equation*}
\frac{[T(\A_{K})/T(K)]^{\Gamma_{K/F}}}{\text{ker}(H_{T})} \to \frac{[U(\A_{K})/U(K)]^{\Gamma_{K/F}}}{\text{ker}(H_{U})}
\end{equation*}
is also a finite-index inclusion. 

Now consider the commutative diagram
\[
\begin{tikzcd}
1 \arrow{r} & \text{ker}(H_{T}) \arrow{d} \arrow{r} &  \left[ T(\A_{K})/T(K) \right] ^{\Gamma_{K/F}} \arrow{d} \arrow{r} & \text{Im}(H_{T}) \arrow{d} \arrow{r} & 1 \\
1 \arrow{r} & \text{ker}(H_{U}) \arrow{r} & \left[ U(\A_{K})/U(K) \right] ^{\Gamma_{K/F}} \arrow{r} & \text{Im}(H_{U}) \arrow{r} & 1,
\end{tikzcd}
\]
where the right-most vertical map is injective and the left-most vertical map is proper, since both groups are Hausdorff and compact (cf. Lemma \ref{compactness1}). The middle term $[U(\A_{K})/U(K)]^{\Gamma_{K/F}}$ has $\text{ker}(H_{U})$ as an open subgroup, and so any compact subset $C$ of $[U(\A_{K})/U(K)]^{\Gamma_{K/F}}$ is contained in finitely many $\text{ker}(H_{U})$-cosets. Moreover, the injectivity of the right-most map implies that the $f$-preimage of $\text{ker}(H_{U})$ is $\text{ker}(H_{T})$, and thus the $f$-preimage of each $\text{ker}(H_{U})$-coset is either empty or a $\text{ker}(H_{T})$-coset, proving the properness of the middle vertical map over each $\text{ker}(H_{U})$-coset, and therefore the properness of the middle vertical map itself, giving the result.
\end{proof}

\begin{remark}
In our applications in the main body of the paper, the map $T \xrightarrow{f} U$ is an isogeny, in which case the proof of the above Lemma follows immediately from the properness of $f$.
\end{remark}

We give $\bar{H}^{1}(\A^{\text{sep}}/\A, T \xrightarrow{f} U)$ the topology determined by declaring that the map $$\text{cok}([T(\A^{\text{sep}})/T(F^{\text{sep}})]^{\Gamma} \to [U(\A^{\text{sep}})/U(F^{\text{sep}})]^{\Gamma}) \to \bar{H}^{1}(\A^{\text{sep}}/\A, T \xrightarrow{f} U)$$ is an open immersion (where the left-hand side has the natural quotient topology). For any $i \geq 2$, we give $\bar{H}^{i}(\A^{\text{sep}}/\A, T \xrightarrow{f} U)$ the discrete topology.


\subsection{Complexes of tori over global function fields---duality} 

One has duality for $\bar{H}^{i}(\overline{\A}/\A, T \xrightarrow{f} U) = \bar{H}^{i}(\A^{\text{sep}}/\A, T \xrightarrow{f} U)$. As in the local case, we have a Tate-Nakayama pairing 
\begin{equation}\label{globalpairing1}
\bar{H}^{r}(\A^{\text{sep}}/F, T \xrightarrow{f} U) \times H^{3-r}(\Gamma,  X^{*}(U) \xrightarrow{f^{*}} X^{*}(T)) \to \mathbb{Q}/\Z,
\end{equation}
where the $\mathbb{Q}/\Z$ comes from identifying $\mathcal{H}^{2}(\mathbb{G}_{m}[(\A^{\text{sep}})^{\bigotimes_{\A}(\bullet+1)}]/\mathbb{G}_{m}[(F^{\text{sep}})^{\bigotimes_{F}(\bullet+1)}])$ with $H^{2}(\Gamma, C)$ (where $C = \varinjlim_{K/F} C_{K}$ is the universal id\'{e}le class group) and then taking the global invariant map. For an $F$-torus $T$, set $\bar{H}^{i}(\A^{\text{sep}}/\A, T) := \mathcal{H}^{i}(T[(\A^{\text{sep}})^{\bigotimes_{\A}(\bullet+1)}]/T[(F^{\text{sep}})^{\bigotimes_{F}(\bullet+1)}])$  (we can define an analogue for $\overline{\A}$, but we won't use that here).  

According to \cite[Lem. D.2.A]{KS1} (which relies on \cite[\S 4]{Milne}, which are stated for arbitrary nonarchimedean local fields), the groups $\bar{H}^{r}(\A^{\text{sep}}/\A, T)$ vanish for $r \geq 3$, and for $r=1,2$ we having a pairing $$ \bar{H}^{r}(\A^{\text{sep}}/\A, T) \times H^{2-r}(\Gamma, X^{*}(T)) \to \mathbb{Q}/\Z $$ which induces isomorphisms $\bar{H}^{r}(\A^{\text{sep}}/\A, T) \xrightarrow{\sim} \Hom_{\Z}(H^{2-r}(\Gamma, X^{*}(T)), \mathbb{Q}/\Z)$, and the group $\bar{H}^{1}(\A^{\text{sep}}/\A, T)$ is finite. We now extend this to our complexes:

\begin{lem} For $r \geq 4$, the groups $\bar{H}^{r}(\A^{\text{sep}}/\A, T \xrightarrow{f} U)$ vanish. For $r=2,3$, the pairing \eqref{globalpairing1} induces an isomorphism $$\bar{H}^{r}(\A^{\text{sep}}/\A, T \xrightarrow{f} U) \xrightarrow{\sim} \Hom_{\Z}(H^{3-r}(\Gamma, X^{*}(U) \xrightarrow{f^{*}} X^{*}(T)), \mathbb{Q}/Z).$$ For $r=2,3$, the group $\bar{H}^{r}(\A^{\text{sep}}/\A, T \xrightarrow{f} U)$ is finitely-generated, and for $r=3$ it is free.
\end{lem}

\begin{proof} See the explanation following \cite[Lem. C.2.A]{KS1}.
\end{proof}

We now give a duality theorem for $r=1$, which will use the absolute Weil group $W_{F}$ of $F$ (given by the inverse limit of extensions of $\A_{K}^{\times}/K^{\times}$ by $\Gamma_{K/F}$ corresponding to the canonical $H^{2}$-class, as $K/F$ ranges over all finite Galois extensions) as in the local case. We define the groups $C^{m}(W_{F}, \widehat{T})$, $H^{m}(W_{F}, \widehat{T})$, $C^{m}(W_{F}, \widehat{U} \xrightarrow{\hat{f}} \widehat{T})$, and $H^{m}(W_{F}, \widehat{U} \xrightarrow{\hat{f}} \widehat{T})$ in the same way as in the local case. Note that $H^{m}(W_{F}, \widehat{T})$ vanishes for $m \geq 2$, $H^{1}(W_{F}, \widehat{T})$ is canonically isomorphic to $\Hom_{\text{cts}}(\bar{H}^{0}(\A^{\text{sep}}/\A, T), \mathbb{C}^{\times})$, by \cite{La}, and $H^{m}(W_{F}, \widehat{U} \xrightarrow{\hat{f}} \widehat{T}) = 0$ for $m \geq 3$. We have the following global analogue of Proposition \ref{weiltogalois1}:

\begin{prop}\label{weiltogalois2} When $T \xrightarrow{f} U$ is an isogeny, the inflation map $H^{1}(\Gamma,  \widehat{U} \xrightarrow{\hat{f}} \widehat{T}) \to H^{1}(W_{F},  \widehat{U} \xrightarrow{\hat{f}} \widehat{T})$ is an isomorphism.
\end{prop}

\begin{proof} As in the proof of Proposition \ref{weiltogalois1}, the inflation-restriction sequence shows that it's enough to show that the image of any element in $\Hom_{\text{cts}}(\A_{K}^{\times}/K^{\times}, \text{ker}(\hat{f}))$ is zero in some large finite Galois extension $L/F$ containing $K$, which follows from the fact that the kernel of any such homomorphism is open and finite-index and the universal norm group of (the idele class groups of) a global function field is trivial (see \cite[Prop. 8.1.26]{NSW}). 
\end{proof}

There is a pairing 
\begin{equation}\label{globalweilpairing}
\bar{H}^{1}(\A^{\text{sep}}/\A, T \xrightarrow{f} U) \times H^{1}(W_{F}, \widehat{U} \xrightarrow{\hat{f}} \widehat{T}) \to \mathbb{C}^{\times}
\end{equation}
defined exactly as in the local case, and, as in that case, it induces a surjective homomorphism $$H^{1}(W_{F}, \widehat{U} \xrightarrow{\hat{f}} \widehat{T}) \to \Hom_{\text{cts}}(\bar{H}^{1}(\A^{\text{sep}}/\A, T \xrightarrow{f} U), \mathbb{C}^{\times})$$ with kernel the image of $(T^{\Gamma})^{\circ} \subseteq H^{0}(W_{F}, \widehat{T})$ in $H^{1}(W_{F}, \widehat{U} \xrightarrow{\hat{f}} \widehat{T})$, the quotient by which we will denote by $H^{1}(W_{F}, \widehat{U} \xrightarrow{\hat{f}} \widehat{T})_{\text{red}}$.

We now define a compact subgroup $\bar{H}^{i}(\A^{\text{sep}}/\A, T \xrightarrow{f} U)_{1}$ of $\bar{H}^{i}(\A^{\text{sep}}/\A, T \xrightarrow{f} U)$ for $i=0,1$. We first set $\bar{H}^{0}(\A^{\text{sep}}/\A, T)_{1}$ to be the kernel of the group homomorphism $H$ from \eqref{Hmapeq}. For expository purposes, we use a slightly different but equivalent version of $H$ from \eqref{Hmapeq} (it obviously has the same kernel)
$$H \colon [T(\A^{\text{sep}})/T(F^{\text{sep}})]^{\Gamma} \to X_{*}(T)^{\Gamma}$$ determined by, for all $\lambda \in X^{*}(T)^{\Gamma}$, the equality 
\begin{equation}\label{degree}
\langle  \lambda, H(\bar{t}) \rangle = \text{deg}(\lambda(\bar{t})), \end{equation}
where $\text{deg} \colon \A^{\times}/F^{\times} \to \Z$ is the homomorphism defined by $\text{deg}(\bar{\alpha}) = \sum_{v \in V} v(\alpha_{v})[k_{v} \colon k]$ and $k$ denotes the constant field of $F$.

\begin{lem}\label{compactness1} The kernel of $H$ is a compact subgroup of $[T(\A^{\text{sep}})/T(F^{\text{sep}})]^{\Gamma}$ (topologized in \S A.3).
\end{lem}

\begin{proof} There is isogeny $T \to T_{a} \times T_{s}$, where $T_{a}$ is the maximal $F$-anisotropic subtorus of $T$ and $T_{s}$ is the maximal $F$-split subtorus of $T$, giving the bottom injective map in the commutative diagram 
\[
\begin{tikzcd}
(T(\A^{\text{sep}})/T(F^{\text{sep}}))^{\Gamma} \arrow{r} \arrow{d} & ((T_{a} \times T_{s})(\A^{\text{sep}})/(T_{a} \times T_{s})(F^{\text{sep}}))^{\Gamma} \arrow{d} \\
X_{*}(T)^{\Gamma} \arrow{r} & X_{*}(T_{a} \times T_{s})^{\Gamma} = X_{*}(T_{s});
\end{tikzcd}
\]
it follows that the kernel of the left-hand map (the group we're analyzing) equals the preimage of the kernel of the right-hand map. Since the top map is induced by the isogeny $T \to T_{a} \times T_{s}$ (which is proper), if we can show that the kernel of the right-hand map is compact, then its preimage in $[T(\A^{\text{sep}})/T(F^{\text{sep}})]^{\Gamma}$ is also compact (the map of topological groups $T(\A^{\text{sep}}) \to (T_{s} \times T_{a})(\A^{\text{sep}})$ is proper, by \cite[Prop. 5.8]{Conrad2}). Rewriting the group $[(T_{a} \times T_{s})(\A^{\text{sep}})/(T_{a} \times T_{s})(F^{\text{sep}})]^{\Gamma}$ as $$[T_{a}(\A^{\text{sep}})/T_{a}(F^{\text{sep}})]^{\Gamma} \times [T_{s}(\A^{\text{sep}})/T_{s}(F^{\text{sep}})]^{\Gamma},$$ it's clear that the kernel in question equals $[T_{a}(\A^{\text{sep}})/T_{a}(F^{\text{sep}})]^{\Gamma}  \times K_{s}$, where $K_{s}$ denotes the kernel of the map $[T_{s}(\A^{\text{sep}})/T_{s}(F^{\text{sep}})]^{\Gamma} \xrightarrow{H_{s}} X_{*}(T_{s})$. The group $[T_{a}(\A^{\text{sep}})/T_{a}(F^{\text{sep}})]^{\Gamma}$ is compact because it contains $T_{a}(\A)/T_{a}(F)$ as a finite-index closed subgroup, and this latter group is compact (by \cite[Thm. 8.1.3]{Conrad3}, since $T_{a}$ is $F$-anisotropic), reducing to the case where $T = T_{s}$ is $F$-split. 

Pick a $\Z$-basis $\lambda_{1}, \dots \lambda_{n}$ of $X^{*}(T) = X^{*}(T)^{\Gamma}$. Then $\bar{t}$ lies in the kernel of $H$ if and only if $\text{deg}[\lambda_{i}(\bar{t})] = 0$ for all $i$. The basis gives an $F$-isomorphism $T \xrightarrow{(\lambda_{i})} \mathbb{G}_{m}^{n}$, and the kernel of $H$ is the preimage under the above isomorphism of the kernel of the map $(\A^{\times}/F^{\times})^{n} \xrightarrow{\text{deg}^{n}} \Z^{n}$, which is the $n$-fold product of the compact subgroups $C^{0}_{F}$ of $\A^{\times}/F^{\times}$ (by \cite[Prop. 8.1.25]{NSW}).
\end{proof}

Define $\bar{H}^{0}(\A^{\text{sep}}/\A, T \xrightarrow{f} U)_{1}$ as the intersection of $\bar{H}^{0}(\A^{\text{sep}}/\A, T \xrightarrow{f} U) \subseteq [T(\A^{\text{sep}})/T(F^{\text{sep}})]^{\Gamma}$ with $\text{ker}(H)$; one checks that when $f$ is an isogeny this intersection is all of $\bar{H}^{0}(\A^{\text{sep}}/\A, T \xrightarrow{f} U)$. It remains to define $\bar{H}^{1}(\A^{\text{sep}}/\A, T \xrightarrow{f} U)_{1}$ . For any $\lambda \in X^{*}(U)^{\Gamma}$, we have a map of complexes from $[T \xrightarrow{f} U]$ to $[1 \to \mathbb{G}_{m}]$ via $T \to 1$, $U \xrightarrow{\lambda} \mathbb{G}_{m}$
which induces a map $\bar{H}^{1}(\A^{\text{sep}}/\A, T \xrightarrow{f} U) \to  \bar{H}^{1}(\A^{\text{sep}}/\A, 1 \to \mathbb{G}_{m}) = \bar{H}^{0}(\A^{\text{sep}}/\A, \mathbb{G}_{m}) = \A^{\times}/F^{\times}$, which we may then map to $\mathbb{\Z}$ via $\text{deg}$, as above, giving a map $\bar{H}^{1}(\A^{\text{sep}}/\A, T \xrightarrow{f} U) \to X_{*}(U)^{\Gamma}$. We define $\bar{H}^{1}(\A^{\text{sep}}/\A, T \xrightarrow{f} U) _{1}$ as the kernel of the composition $$H^{(1)} \colon \bar{H}^{1}(\A^{\text{sep}}/\A, T \xrightarrow{f} U) \to X_{*}(U)^{\Gamma} \to \frac{X_{*}(U)^{\Gamma}}{f_*(X_{*}(T)^{\Gamma})}.$$

Note that we have a commutative diagram with exact rows
\begin{equation}\label{twoHmaps}
\begin{tikzcd}[column sep=small]
1 \arrow{r} & \bar{H}^{0}(\A^{\text{sep}}/\A, T \xrightarrow{f} U) \arrow{r} \arrow{d} & \bar{H}^{0}(\A^{\text{sep}}/\A, T) \arrow["f"]{r} \arrow["H_{T}"]{d} & \bar{H}^{0}(\A^{\text{sep}}/\A, U) \arrow["H_{U}"]{d} \arrow["\delta"]{r} &  \bar{H}^{1}(\A^{\text{sep}}/\A, T \xrightarrow{f} U) \arrow["H^{(1)}"]{d} \\
0 \arrow{r} & \text{Ker}(\restr{f_{*}}{X_{*}(T)^{\Gamma}}) \arrow{r} & X_{*}(T)^{\Gamma} \arrow["f_{*}"]{r} & X_{*}(U)^{\Gamma} \arrow{r} & X_{*}(U)^{\Gamma}/f_{*}(X_{*}(T)^{\Gamma}).
\end{tikzcd}
\end{equation}

We claim that the map $H_{T} \colon \bar{H}^{0}(\A^{\text{sep}}/\A, T) \to X_{*}(T)^{\Gamma}$ (and hence also $H_{U}$, by symmetry) is split; indeed, the isogeny $T_{a} \times T_{s} \to T$ gives the commutative diagram 
\[
\begin{tikzcd}
((T_{a} \times T_{s})(\A^{\text{sep}})/(T_{a} \times T_{s})(F^{\text{sep}}))^{\Gamma} \arrow["H_{T_{a} \times T_{s}}"]{d} \arrow{r} & (T(\A^{\text{sep}})/T(F^{\text{sep}}))^{\Gamma} \arrow["H_{T}"]{d}  \\
X_{*}(T_{s}) \arrow["\sim"]{r} & X_{*}(T)^{\Gamma}.
\end{tikzcd}
\]
As in the proof of the Lemma \ref{compactness1}, to split $H_{T_{a} \times T_{s}}$, it's enough to split $H_{T_{s}}$, and we have $\lambda_{i} \in X^{*}(T_{s})$ such that $T_{s} \xrightarrow{(\lambda_{i})} \mathbb{G}_{m}^{n}$ is an isomorphism, and so it's enough to split the map $(\A^{\times}/F^{\times})^{n} \xrightarrow{\text{deg}^{n}} \mathbb{\Z}^{n}$, which is easy. This splitting of $H_{T_{a} \times T_{s}}$ gives a splitting of $H_{T}$ by applying the inverse isomorphism $X_{*}(T)^{\Gamma} \to X_{*}(T_{s})$, giving the claim. The compatibility of $f$ with the these splittings gives an induced splitting $X_{*}(U)^{\Gamma}/f_{*}(X_{*}(T)^{\Gamma}) \to \bar{H}^{1}(\A^{\text{sep}}/\A, T \xrightarrow{f} U)$ of $H^{(1)}$.

\begin{lem} The closed subgroup $\bar{H}^{1}(\A^{\text{sep}}/\A, T \xrightarrow{f} U)_{1} \subset \bar{H}^{1}(\A^{\text{sep}}/\A, T \xrightarrow{f} U)$ is compact.
\end{lem}

\begin{proof} We have a natural injection $\bar{H}^{0}(\A^{\text{sep}}/\A, U)_{1}/f(\bar{H}^{0}(\A^{\text{sep}}/\A, T)_{1}) \hookrightarrow \bar{H}^{1}(\A^{\text{sep}}/\A, T \xrightarrow{f} U)_{1}$, which, by the definition of our topologies, is a closed immersion. We claim that, in fact, this is a subgroup of a finite index in the target. By the commutative diagram \eqref{twoHmaps}, we have $$\delta^{-1}[ \bar{H}^{1}(\A^{\text{sep}}/\A, T \xrightarrow{f} U)_{1}] = \bar{H}^{0}(\A^{\text{sep}}/\A, U)_{1} \cdot f[\bar{H}^{0}(\A^{\text{sep}}/\A, T)] \subseteq \bar{H}^{0}(\A^{\text{sep}}/\A, U),$$ and hence the image of the above natural injection is $\delta[\bar{H}^{0}(\A^{\text{sep}}/\A, U)] \cap  \bar{H}^{1}(\A^{\text{sep}}/\A, T \xrightarrow{f} U)_{1}$, which is of finite index, since $\delta[\bar{H}^{0}(\A^{\text{sep}}/\A, U)]$ is of finite index in $\bar{H}^{1}(\A^{\text{sep}}/\A, T \xrightarrow{f} U)$,  by the finiteness of $\bar{H}^{1}(\A^{\text{sep}}/\A, T)$. Since  $\bar{H}^{0}(\A^{\text{sep}}/\A, U)_{1}/f(\bar{H}^{0}(\A^{\text{sep}}/\A, T)_{1})$ is itself compact (by Lemma \ref{compactness1}), the result follows. 
\end{proof}

\begin{cor} When $f$ is an isogeny, the group $\bar{H}^{1}(\A^{\text{sep}}/\A, T \xrightarrow{f} U)$ is compact.
\end{cor}

\begin{proof} This follows immediately from the above lemma and the fact that $X_{*}(U)^{\Gamma}/X_{*}(T)^{\Gamma}$ is finite, due to the fact that $X_{*}(T) \subseteq X_{*}(U)$ is finite-index and $X_{*}(U)^{\Gamma} \cap X_{*}(T) = X_{*}(T)^{\Gamma}$. 
\end{proof}

We conclude this section by giving new global duality results that involve the above cohomology groups. We have a natural map $H^{i}(F^{\text{sep}}/F, T \xrightarrow{f} U) \to H^{i}(\A^{\text{sep}}/\A, T \xrightarrow{f} U)$, and we will denote its kernel by $\text{ker}^{i}(F^{\text{sep}}/F, T \xrightarrow{f} U)$ and its cokernel by $\text{cok}^{i}(F^{\text{sep}}/F, T \xrightarrow{f} U)$; our primary case of interest in this paper is when $i=1$. 
Moroever, Proposition \ref{restrictedproduct} induces an identification $$\text{ker}^{i}(F^{\text{sep}}/F, T \xrightarrow{f} U) \xrightarrow{\sim} \text{ker}[H^{i}(F^{\text{sep}}/F, T \xrightarrow{f} U) \to \prod_{v \in V} H^{i}(F_{v}^{\text{sep}}/F_{v}, T \xrightarrow{f} U)].$$ We have, from the long exact sequence \eqref{globalLES}, the short exact sequences 
\begin{equation}\label{cokernelSES}
1 \to \text{cok}^{i}(F^{\text{sep}}/F, T \xrightarrow{f} U) \to \bar{H}^{i}(\A^{\text{sep}}/\A, T \xrightarrow{f} U) \to \text{ker}^{i+1}(F^{\text{sep}}/F, T \xrightarrow{f} U) \to 1.
\end{equation}

The following is an analogue of \cite[Lem. C.3.A]{KS1}:

\begin{lem}\label{opencokernels} For all $i$, the image of $H^{i}(F^{\text{sep}}/F, T \xrightarrow{f} U)$ is discrete in $H^{i}(\A^{\text{sep}}/\A, T \xrightarrow{f} U)$. Moreover, the map $$\text{cok}^{i}(F^{\text{sep}}/F, T \xrightarrow{f} U) \to \bar{H}^{i}(\A^{\text{sep}}/\A, T \xrightarrow{f} U)$$ induces an isomorphism of topological groups from $\text{cok}^{i}(F^{\text{sep}}/F, T \xrightarrow{f} U)$ to an open subgroup of $\bar{H}^{i}(\A^{\text{sep}}/\A, T \xrightarrow{f} U)$ for $i=0,1$. 
\end{lem}

\begin{proof} The first statement is clear for $i \neq 1$ (cf. the analogous argument in \cite{KS1}), so we only need to prove both statements for $i=1$. For the first statement, it's enough to show that the intersection of $f[H^{1}(F^{\text{sep}}/F, T \xrightarrow{f} U)]$ with the open subgroup $U(\A)/f(T(\A))$ is discrete. Since the image of $U(F)/f(T(F))$ is of finite index in $[U(\A)/f(T(\A))] \cap f[H^{1}(F^{\text{sep}}/F, T \xrightarrow{f} U)]$ (because the kernel of $H^{1}(F, T) \to \prod'_{v} H^{1}(F_{v}, T)$ is finite), it's enough to show that the image of $U(F)/f(T(F))$ is discrete in $U(\A)/f(T(\A))$. 

Using \eqref{degree}, we have a split surjective homomorphism $T(\A) \to X_{*}(T)^{\Gamma}$ with closed (not necessarily compact) kernel $T(\A)_{1}$, similarly for $U$, and the induced product structures are compatible with the homomorphism $f$, allowing us to rewrite $f$ as $$T(\A)_{1} \times X_{*}(T)^{\Gamma} \xrightarrow{f \times f_{*}} U(\A)_{1} \times X_{*}(U)^{\Gamma},$$ leading to a decomposition $$U(\A)/f(T(\A)) = U(\A)_{1}/f(T(\A)_{1}) \times X_{*}(U)^{\Gamma}/f_{*}(X_{*}(T)^{\Gamma}),$$ and the image of $U(F)/f(T(F))$ in $U(\A)/f(T(\A))$ lands in the factor $U(\A)_{1}/f(T(\A)_{1})$.

The subgroup $f(T(F))$ is evidently discrete in $U(\A)_{1}$, since the subgroup $U(F)$ is discrete in $U(\A)$ (by \cite{Conrad2}, Example 2.2, using that $F$ is discrete in $\A$). Thus, $U(\A)_{1}/f(T(F))$ contains the discrete subgroup $U(F)/f(T(F))$ and the compact subgroup $f(T(\A)_{1})/f(T(F))$ (the compactness follows from Lemma \ref{compactness1}). The desired discreteness then follows by the analogous argument in the proof of \cite[Lem. C.3.A]{KS1}.

As in \cite{KS1}, to prove the second statement for $i=1$ it suffices to show that the map $$U(\A) \to [U(\A^{\text{sep}})/U(F^{\text{sep}})]^{\Gamma}/f[T(\A^{\text{sep}})/T(F^{\text{sep}})]^{\Gamma}$$ is open. The image $U(\A)/U(F) \hookrightarrow [U(\A^{\text{sep}})/U(F^{\text{sep}})]^{\Gamma}$ is closed (a straightforward exercise in the topology of adelic points), and is also finite index (by the finiteness of the kernel of $H^{1}(F, U) \to H^{1}(\A, U)$), and is thus an open immersion. Since quotient maps are open, the composition $$U(\A) \to U(\A)/U(F) \to [U(\A^{\text{sep}})/U(F^{\text{sep}})]^{\Gamma} \to [U(\A^{\text{sep}})/U(F^{\text{sep}})]^{\Gamma}/f[T(\A^{\text{sep}})/T(F^{\text{sep}})]^{\Gamma}$$ is open, as desired.

It remains to show that the injection $\text{cok}^{0}(F^{\text{sep}}/F, T \xrightarrow{f} U) \to \bar{H}^{0}(\A^{\text{sep}}/\A, T \xrightarrow{f} U)$ has open image. As in \cite{KS1}, it's enough to show that the map $H^{0}(\A^{\text{sep}}/\A, T \xrightarrow{f} U) \to \text{ker}[T(\A)/T(F) \to U(\A)/U(F)]$ is open (because, as in \cite{KS1}, $T(\A)/T(F)$ is open in $[T(\A^{\text{sep}})/T(F^{\text{sep}})]^{\Gamma}$). Define the closed subgroup $B := \{t \in T(\A) \mid f(t) \in U(F)\}$ of $T(\A)$. Note that $H^{0}(\A^{\text{sep}}/\A, T \xrightarrow{f} U)$ is a closed subgroup of $B$, and we thus have a closed immersion $B/[H^{0}(\A^{\text{sep}}/\A, T \xrightarrow{f} U)] \hookrightarrow U(F) \hookrightarrow U(\A)$, where the last closed immersion has discrete image. It follows that, since $H^{0}(\A^{\text{sep}}/\A, T \xrightarrow{f} U)$ is a closed subgroup of $B$ with discrete quotient, it's open, and then the result follows from the fact that $B/T(F) = \text{ker}[T(\A)/T(F) \to U(\A)/T(F)]$. 
\end{proof}

\begin{cor}\label{compactcokernel} The group $$\text{cok}^{1}(F^{\text{sep}}/F, T \xrightarrow{f} U)_{1} := \text{cok}^{1}(F^{\text{sep}}/F, T \xrightarrow{f} U) \cap \bar{H}^{1}(\A^{\text{sep}}/\A, T \xrightarrow{f} U)_{1}$$ is compact. Moreover, when $f$ is an isogeny, the group $\text{cok}^{1}(F^{\text{sep}}/F, T \xrightarrow{f} U)$ is compact.
\end{cor}

Set $\text{ker}^{1}(W_{F}, \widehat{U} \xrightarrow{\hat{f}} \widehat{T})_{\text{red}} := \text{ker}[H^{1}(W_{F}, \widehat{U} \xrightarrow{\hat{f}} \widehat{T})_{\text{red}} \to \prod_{v \in V} H^{1}(W_{F_{v}}, \widehat{U} \xrightarrow{\hat{f}} \widehat{T})_{\text{red}}]$.

\begin{prop}\label{cokernelduality} We have a duality isomorphism $$\Hom_{\text{cts}}(\text{cok}^{1}(F^{\text{sep}}/F, T \xrightarrow{f} U), \mathbb{C}^{\times}) \xrightarrow{\sim} H^{1}(W_{F}, \widehat{U} \xrightarrow{\hat{f}} \widehat{T})_{\text{red}}/\text{ker}^{1}(W_{F}, \widehat{U} \xrightarrow{\hat{f}} \widehat{T})_{\text{red}}.$$ Moreover, the group $\text{ker}^{1}(F^{\text{sep}}/F, T \xrightarrow{f} U)$ is finite. 
\end{prop}

\begin{proof} Using that $\text{cok}^{1}(F^{\text{sep}}/F, T \xrightarrow{f} U)$ is an open subgroup of $\bar{H}^{1}(\A^{\text{sep}}/\A, T \xrightarrow{f} U)$, applying the functor $\Hom_{\text{cts}}(-, \mathbb{C}^{\times})$ to the short exact sequence \eqref{cokernelSES} with $i=1$ gives that the group $\Hom_{\text{cts}}(\text{cok}^{1}(F^{\text{sep}}/F, T \xrightarrow{f} U), \mathbb{C}^{\times})$ is canonically isomorphic to the quotient $$\Hom_{\text{cts}}(\bar{H}^{1}(\A^{\text{sep}}/\A, T \xrightarrow{f} U), \mathbb{C}^{\times})/  \Hom_{\text{cts}}(\text{ker}^{2}(F^{\text{sep}}/F, T \xrightarrow{f} U), \mathbb{C}^{\times}).$$ 

Moreover, the same short exact sequence tells us that $\Hom_{\text{cts}}(\text{ker}^{2}(F^{\text{sep}}/F, T \xrightarrow{f} U), \mathbb{C}^{\times})$ is canonically isomorphic to the subgroup $$\text{ker}[\Hom_{\text{cts}}(\bar{H}^{1}(\A^{\text{sep}}/\A, T \xrightarrow{f} U), \mathbb{C}^{\times}) \to \Hom_{\text{cts}}(H^{1}(\A^{\text{sep}}/\A, T \xrightarrow{f} U), \mathbb{C}^{\times})].$$ But now we know that $\Hom_{\text{cts}}(\bar{H}^{1}(\A^{\text{sep}}/\A, T \xrightarrow{f} U), \mathbb{C}^{\times})$ is canonically isomorphic to the group $H^{1}(W_{F}, \widehat{U} \xrightarrow{\hat{f}} \widehat{T})_{\text{red}}$ via the pairing \eqref{globalweilpairing}, that $H^{1}(\A^{\text{sep}}/\A, T \xrightarrow{f} U)$ is canonically isomorphic to $\prod'_{v} H^{1}(F_{v}^{\text{sep}}/F_{v}, T \xrightarrow{f} U)$ (by Proposition \ref{restrictedproduct}), and that each $H^{1}(F_{v}^{\text{sep}}/F_{v}, T \xrightarrow{f} U)$ has continuous dual canonically isomorphic to $H^{1}(W_{F_{v}}, \widehat{U} \xrightarrow{\hat{f}} \widehat{T})_{\text{red}}$, which gives the result.

For the finiteness of $\text{ker}^{1}(F^{\text{sep}}/F, T \xrightarrow{f} U)$, one checks that the map $$\bar{H}^{0}(\A^{\text{sep}}/\A, T \xrightarrow{f} U)  \to \text{Ker}(\restr{f_{*}}{X_{*}(T)^{\Gamma}})$$ from \eqref{twoHmaps} remains surjective when restricted to $\text{cok}^{0}(F^{\text{sep}}/F, T \xrightarrow{f} U)$, so $\bar{H}^{0}(\A^{\text{sep}}/\A, T \xrightarrow{f} U)_{1}$ surjects onto $\text{ker}^{1}(F^{\text{sep}}/F, T \xrightarrow{f} U)$ with open kernel (by Lemma \ref{opencokernels}). Since $\bar{H}^{0}(\A^{\text{sep}}/\A, T \xrightarrow{f} U)_{1}$ is compact, its quotient by an open subgroup is finite.
\end{proof}

We have the following analogue of \cite[Lem. C.3.B]{KS1}, whose adaptation we leave here (for completeness) as an exercise (Proposition \ref{cokernelduality} is the only part of this result used in this paper):

\begin{prop} The groups $\text{ker}^{i}(F^{\text{sep}}/F, T \xrightarrow{f} U)$ are finite for all $i$ and vanish unless $i=1,2,3$. For $i=1,2,3$, we have dual finite abelian groups
$$\Hom(\text{ker}^{1}(F^{\text{sep}}/F, T \xrightarrow{f} U), \mathbb{C}^{\times}) = \text{ker}^{2}(W_{F}, \widehat{U} \xrightarrow{\hat{f}} \widehat{T}), \hspace{1mm} \Hom(\text{ker}^{1}(F^{\text{sep}}/F, T \xrightarrow{f} U), \mathbb{C}^{\times}) =$$ $$\text{ker}^{1}(W_{F}, \widehat{U} \xrightarrow{\hat{f}} \widehat{T})_{red},\hspace{1mm} \Hom(\text{ker}^{3}(F^{\text{sep}}/F, T \xrightarrow{f} U), \mathbb{Q}/\Z) = \text{ker}^{1}(\Gamma, X^{*}(U) \xrightarrow{f^{*}} X^{*}(T)).$$
The groups $\text{cok}^{i}(F^{\text{sep}}/F, T \xrightarrow{f} U)$ vanish for $i \geq 4$, and for $i \leq 3$ we have duality isomorphisms
$$\Hom_{\text{cts}}(\text{cok}^{0}(F^{\text{sep}}/F, T \xrightarrow{f} U), \mathbb{C}^{\times}) = H^{2}(W_{F}, \widehat{U} \xrightarrow{\hat{f}} \widehat{T})/\text{ker}^{2}(W_{F}, \widehat{U} \xrightarrow{\hat{f}} \widehat{T}),$$
$$\Hom_{\text{cts}}(\text{cok}^{1}(F^{\text{sep}}/F, T \xrightarrow{f} U), \mathbb{C}^{\times}) = H^{1}(W_{F}, \widehat{U} \xrightarrow{\hat{f}} \widehat{T})_{\text{red}}/\text{ker}^{1}(W_{F}, \widehat{U} \xrightarrow{\hat{f}} \widehat{T})_{\text{red}},$$
$$\text{cok}^{2}(F^{\text{sep}}/F, T \xrightarrow{f} U) = \Hom(H^{1}(\Gamma, X^{*}(U) \xrightarrow{f^{*}} X^{*}(T))/\text{ker}^{1}(\Gamma, X^{*}(U) \xrightarrow{f^{*}} X^{*}(T)), \mathbb{Q}/\Z),$$
$$\text{cok}^{3}(F^{\text{sep}}/F, T \xrightarrow{f} U) = \bar{H}^{3}(\A^{\text{sep}}/\A, T \xrightarrow{f} U) = \Hom(H^{0}(\Gamma, X^{*}(U) \xrightarrow{f^{*}} X^{*}(T)), \mathbb{Q}/\Z),$$
where all groups not already defined above are defined as in \cite[Appendix C]{KS1}.
\end{prop}

We conclude with a few results involving $H^{1}(\Gamma, \widehat{U} \xrightarrow{\hat{f}} \widehat{T})$. First, define $H^{1}(\Gamma, \widehat{U} \xrightarrow{\hat{f}} \widehat{T})_{\text{red}}$ to be the quotient of $H^{1}(\Gamma, \widehat{U} \xrightarrow{\hat{f}} \widehat{T})$ by the image of $(\widehat{T}^{\Gamma})^{\circ} \subseteq H^{0}(\Gamma, \widehat{T})$. For any $v \in V$, we define the quotient $H^{1}(\Gamma_{v}, \widehat{U} \xrightarrow{\hat{f}} \widehat{T})_{\text{red}}$ of $H^{1}(\Gamma_{v}, \widehat{U} \xrightarrow{\hat{f}} \widehat{T})$ analogously, with $\Gamma$ replaced by $\Gamma_{v}$. Finally, set 
$$\text{ker}^{1}(\Gamma, \widehat{U} \xrightarrow{\hat{f}} \widehat{T})_{\text{red}} := \text{ker}[H^{1}(\Gamma, \widehat{U} \xrightarrow{\hat{f}} \widehat{T})_{\text{red}} \to \prod_{v \in V} H^{1}(\Gamma_{v}, \widehat{U} \xrightarrow{\hat{f}} \widehat{T})_{\text{red}}].$$


\begin{lem}(\cite[Lem. C.3.C]{KS1})\label{C.3.C} The map $H^{1}(\Gamma, \widehat{U} \xrightarrow{\hat{f}} \widehat{T})_{\text{red}} \to H^{1}(W_{F}, \widehat{U} \xrightarrow{\hat{f}} \widehat{T})_{\text{red}}$ maps $\text{ker}^{1}(\Gamma, \widehat{U} \xrightarrow{\hat{f}} \widehat{T})_{\text{red}}$ isomorphically onto $\text{ker}^{1}(W_{F}, \widehat{U} \xrightarrow{\hat{f}} \widehat{T})_{\text{red}}$ and we have natural isomorphisms $$H^{1}(\Gamma, \widehat{U} \xrightarrow{\hat{f}} \widehat{T})_{\text{red}} \xrightarrow{\sim} H^{2}(\Gamma, X^{*}(U) \xrightarrow{f^{*}} X^{*}(T)), \hspace{1mm} \text{ker}^{1}(\Gamma, \widehat{U} \xrightarrow{\hat{f}} \widehat{T})_{\text{red}} \xrightarrow{\sim} \text{ker}^{2}(\Gamma, X^{*}(U) \xrightarrow{f^{*}} X^{*}(T)).$$ 
\end{lem}

\begin{proof} These second two maps are induced by the boundary map coming from the commutative diagram of short exact sequences of $\Gamma$-modules
\[
\begin{tikzcd}
0 \arrow{r} & X^{*}(U) \arrow{d} \arrow{r} & \text{Lie}(\widehat{U}) \arrow{d} \arrow{r} & \widehat{U} \arrow{r} \arrow["\hat{f}"]{d} & 1 \\
0 \arrow{r} & X^{*}(T)  \arrow{r} & \text{Lie}(\widehat{T}) \arrow{r} & \widehat{T} \arrow{r} & 1,
\end{tikzcd}
\]
viewed as a short exact sequence of length-2 complexes. The proof is the same as in loc. cit.
\end{proof}

\bibliographystyle{alpha} \footnote{Competing interests: The author declares none.}

\end{document}